\documentclass{article}
\usepackage{amsmath,color,amsfonts,amsthm, amssymb}
\usepackage{amssymb,enumerate,verbatim}
\usepackage{hyperref}
\usepackage{mathtools}
\usepackage{a4wide}

\usepackage[dvipsnames]{xcolor}

\date{}

\newtheorem{lemma}{Lemma}[section]
\newtheorem{corollary}[lemma]{Corollary}
\newtheorem{claim}[lemma]{Claim}
\newtheorem{theorem}[lemma]{Theorem}

\newtheorem{definition}[lemma]{Definition}
\newtheorem{conjecture}[lemma]{Conjecture}

\newtheorem{result}{A\hspace{-0.1cm}}

\theoremstyle{definition}

\global\long\def\E{\mathbb{E}}

\global\long\def\eps{\varepsilon}
\global\long\def\P{\mathbb{P}}

\newcommand{\Symdiff}{\Delta}
\title{Decompositions into spanning rainbow structures}
\author{R. Montgomery\thanks{School of Mathematics,
University of Birmingham,
Edgbaston,
Birmingham,
B15 2TT,
UK. r.h.montgomery@bham.ac.uk}
, A. Pokrovskiy\thanks{Department of Mathematics, ETH, 8092 Zurich, Switzerland. dr.alexey.pokrovskiy@gmail.com. Research supported in part by SNSF grant 200021-175573.}
, and B. Sudakov\thanks{Department of Mathematics, ETH, 8092 Zurich, Switzerland. benjamin.sudakov@math.ethz.ch.
Research supported in part by SNSF grant 200021-175573.}
}

\newcommand{\ad}{\mathrm{d}}
\newcommand{\GG}{\stackrel{\scriptscriptstyle{\text{\sc poly}}}{\gg}}
\newcommand{\LL}{\stackrel{\scriptscriptstyle{\text{\sc poly}}}{\ll}}

\newcommand{\Remark}[1]{}

\begin{document}

\maketitle
\begin{abstract}
A subgraph of an edge-coloured graph is called rainbow if all its edges have distinct colours. The study of rainbow subgraphs goes back more than two hundred years to the work of
Euler on Latin squares and has been the focus of extensive research ever since. {Euler posed a problem  equivalent to finding properly $n$-edge-coloured  complete bipartite graphs $K_{n,n}$ which can be decomposed into rainbow perfect matchings.} While there are proper edge-colourings of $K_{n,n}$ without even a single rainbow perfect matching,
the theme of this paper is to show that with some very weak additional constraints one can find many disjoint rainbow perfect matchings. In particular, we prove that if
some fraction of the colour classes have at most $(1-o(1)) n$ edges then one can nearly-decompose the edges of $K_{n,n}$ into edge-disjoint perfect rainbow matchings. As an application of this, we
establish in a very strong form a conjecture of Akbari and Alipour and asymptotically prove a conjecture of Barat and Nagy. Both these conjectures concern rainbow perfect matchings in edge-colourings of $K_{n,n}$ with quadratically many colours. The above result also has implications to some conjectures of Snevily about subsquares of multiplication tables of groups.

Finally, using our techniques, we also prove a number of results on near-decompositions of graphs into other rainbow structures like Hamiltonian cycles and spanning trees. Most notably, we prove that any properly coloured complete graph can be nearly-decomposed into spanning rainbow trees. This asymptotically proves the Brualdi-Hollingsworth and Kaneko-Kano-Suzuki conjectures which predict that a perfect decomposition should exist under the same assumptions.
\end{abstract}

\section{Introduction}
A \emph{Latin square} of order $n$ is an $ n \times n$  array filled with $n$ symbols such that each symbol appears once in every row and column.
A \emph{partial transversal} is a collection of cells of the Latin square which do not share the same row, column or symbol. A {\em transversal} is a partial transversal of order $n$.
Latin squares were introduced by Euler in the 18th century and are familiar to the layperson in the form of Sudoku puzzles, which, when completed, are Latin squares.
Another well known example of the Latin square is a multiplication table of any finite group. The study of Latin squares have applications both inside and outside mathematics, with connections to 2-dimensional permutations, design theory,  finite projective planes, and error correcting codes.

Euler was interested in \emph{orthogonal Latin squares}---a pair of $n\times n$ Latin squares $S$ and $T$ with the property that every pair of symbols $(i,j)$ occurs precisely once in the array.
This is equivalent to Latin squares which can be decomposed into disjoint transversals (see \cite{Euler, keedwell2015latin}). He conjectured that there exist $n\times n$ Latin squares with a decomposition into disjoint transversals if, and only if, $n\not\equiv 2\pmod 4$. When $n\not\equiv 2\pmod 4$ Euler himself constructed such Latin squares.
The ``$n=6$'' case stood open for over 100 years until it was proved by Tarry in 1901. The remaining cases ``$n\not\equiv 2\pmod 4$, $n\geq 10$'' were resolved in 1959 by Bose, Parker, and Shrikande~\cite{BCP}. Surprisingly, they showed that Euler's Conjecture was false for these values of $n$ by explicitly constructing Latin squares with a decomposition into disjoint transversals.

It is a hard problem to determine which Latin squares have transversals. This question is very difficult even in the case of multiplication tables of finite groups.
In 1955 Hall and Paige \cite{hall1955complete} conjectured that the multiplication table of a group $G$ has a transversal exactly if the $2$-Sylow subgroups of $G$ are trivial or non-cyclic.
It took 50 years to establish this conjecture and its proof is based on the classification of finite simple groups (see \cite{wilcox2009reduction} and the references therein).
The most famous open problem on transversals in general Latin squares is a conjecture of Ryser and Brualdi-Stein.
\begin{conjecture}[Ryser \cite{Ryser}, Brualdi-Stein \cite{Brualdi, Stein}]
\label{BRS}
Every $n\times n$ Latin square has a partial transversal of order $n-1$ and a full transversal if $n$ is odd.
\end{conjecture}
\noindent
The best  results towards this conjecture are asymptotic and show that all Latin squares have partial transversals of size $n - o(n)$. Woolbright~\cite{woolbright78} and Brower, de Vries and Wieringa~\cite{BVW78} independently proved this with $o(n)=\sqrt n$. The error term was further improved by Hatami and Shor~\cite{hatami2008lower}, who showed that $o(n)=O(\log^2 n)$ suffices.

\emph{Generalized Latin squares} are $n\times n$ arrays filled with an {\em arbitrary} number of symbols such that  no symbol appears twice in the same row or column. They are natural extensions of Latin squares, and have also been extensively studied.
A familiar example of a generalized Latin square is a multiplication table between elements of two subsets of equal size in some group. It is generally believed that extra symbols in a Latin square should help to find transversals. The goal of this paper is to confirm that this is indeed the case. Moreover we show that, under some very weak additional conditions, a generalized Latin square has not only one but many disjoint transversals.

\begin{theorem}\label{Theorem_LatinSquare}
Let $S$ be a generalized Latin square with at most $(1-o(1))n$ symbols occurring more than $(1-o(1))n$ times. Then, $S$ has $(1-o(1))n$ pairwise disjoint transversals.
\end{theorem}

\noindent
All previous results that guaranteed transversals studied arrays which were very far from Latin squares. For example, Erd\H{o}s and Spencer~\cite{erdHos1991lopsided} showed that a transversal exists in any $n\times n$ array in which each symbol appears at most $n/16$ times. Furthermore,  Alon, Spencer and Tetali \cite{alon1995covering} found  many disjoint transversals in the case when each symbol appears $\delta n$ times, for some small but fixed $\delta>0$. On the other hand, our result shows that the only generalized Latin squares without transversals are small perturbations of Latin squares.

Theorem \ref{Theorem_LatinSquare} can be also used to attack several open problems on generalized Latin squares.
For example Akbari and Alipour conjectured the following.
\begin{conjecture}[Akbari and Alipour \cite{akbari2004transversals}]\label{ConjectureAkbariAlipour}
Every generalized Latin square with at least $n^2/2$ symbols has a transversal.
\end{conjecture}

\noindent
More generally Barat and Nagy~\cite{barat2017transversals} conjectured that under the same assumptions as above, any generalized Latin square should have a decomposition into disjoint transversals. Theorem~\ref{Theorem_LatinSquare} has implications for both of these conjectures. It is easy to show that in any generalized Latin square with  at least $\varepsilon n^2$ symbols at most $(1-\varepsilon/2)n$ symbols occur more than $(1-\varepsilon/2)n$ times (see Lemma~\ref{Lemma_many_colours_implies_few_large_colours}). Thus the following is a corollary of Theorem~\ref{Theorem_LatinSquare}.

\begin{corollary}\label{Corollary_LatinSquare}
For all $\varepsilon>0$ and sufficiently large n,
every generalized Latin square with at least $\varepsilon n^2$ symbols has $(1-\varepsilon)n$ pairwise disjoint transversals.
\end{corollary}

\noindent
For large $n$, this establishes the conjecture of Akbari-Alipour  in a very strong form, showing that the bound of $n^2/2$ can be reduced to  $\varepsilon n^2$. It also proves asymptotically the Barat-Nagy conjecture, giving a near-decomposition of the generalized Latin square into transversals.

Theorem~\ref{Theorem_LatinSquare} has also some interesting implications for transversals in actual Latin squares.
Indeed, it is not hard to show that any Latin square contains many subsquares which satisfy the assumptions of Theorem~\ref{Theorem_LatinSquare}. In fact, a random $(1-o(1))n\times(1-o(1))n$ subsquare will have this property with high probability. Thus we have the following corollary.
\begin{corollary}
Let $S$ be a random $(1-o(1))n\times(1-o(1))n$ subsquare of an $n\times n$ Latin square $L$. With high probability, $S$ has a transversal.
\end{corollary}

\noindent
This corollary reproves the result that Latin squares have partial transversals of size $n-o(n)$. However, it proves much more, that is, partial transversals of size $n-o(n)$ must \emph{be present almost everywhere} in the Latin square.

Our main theorem has additional applications to group theoretic problems and questions about rainbow structures in coloured graphs, which we discuss next.

\subsection*{Subsquares of multiplication tables}
{A natural way to obtain a generalized Latin square is to consider a subsquare $S$ of a multiplication table of a group $G$.
Snevily made the following general conjecture on transversals in subsquares of abelian groups.}
\begin{conjecture}[Snevily \cite{snevily1999cayley}] \label{Conjecture_Snevily}
{Let $S=A\times B$ be a  subsquare of the multiplication table of an abelian group $G$ defined by two $n$-element sets $A, B\subseteq G$.}
\begin{enumerate}[(i)]
\item If $G$ is an odd abelian group, then $S$ has a transversal.
\item If $G$ is an even cyclic group, then $S$ has no transversal only when both $A$ and $B$ are translates of the same even cyclic subgroup of $G$.
\end{enumerate}
\end{conjecture}
Here a ``translate of $A$'' means any set of the form $gA$ for $g\in V(G)$.
Part (i) of this conjecture has attracted a lot of attention. After work by Alon \cite{alon2000additive} and
Dasgupta, K\'arolyi, Serra and Szegedy \cite{dasgupta2001transversals}, it was solved by Arsovski~\cite{arsovski2011proof}. Part (ii) of Conjecture~\ref{Conjecture_Snevily} is still open.

Our work has implications for this conjecture, and for various generalizations for other groups and semigroups. Combining our Theorem~\ref{Theorem_LatinSquare} with the following lemma one can find not just one but many transversals in certain subsquares of multiplication tables.
\begin{lemma}
\label{Fournier}
Let $S=A\times B$ be a subsquare of the  multiplication table of a group $G$ defined by two $n$-element sets $A, B\subseteq G$.
Then, either $S$ has at most $(1-o(1))n$ symbols occurring more than $(1-o(1))n$ times or there is a subgroup $H$ of $G$ and elements $g,g' \in V(G)$ such that
$|A \Delta gH|=o(n)$ and $|B \Delta g'H|=o(n)$.
\end{lemma}
In other words, this lemma says that either a subsquare $S$ of a multiplication table is close to a translate of a subgroup, or it satisfies  the condition of Theorem~\ref{Theorem_LatinSquare}. In the latter case, we can use this theorem to nearly-decompose $S$ into disjoint transversals. Thus we have the following corollary which works in any group, not just finite or abelian groups.
 \begin{corollary}
Let $S=A\times B$ be a subsquare of the multiplication table of a group $G$ defined by two $n$-element sets $A, B\subseteq G$.
Then, one of the following holds.
\begin{itemize}
\item $S$ has $(1-o(1))n$ disjoint transversals.
\item There is a subgroup $H$ of $G$ and   elements $g, g'\in V(G)$ such that $|A \Delta gH|=o(n)$ and $|B \Delta g'H|=o(n)$.
\end{itemize}
\end{corollary}

Lemma \ref{Fournier} is implicit in the work of Fournier \cite{fournier1977sharpness} and appears as Theorem 1.3.3 in the lecture notes of Green  \cite{GreenNotes}.
It is formulated in terms of multiplicative energy, which for a subset $A$ of group $G$ is the number of quadruples $a_1, a_2, b_1, b_2 \in A$ such that
$a_1a_2^{-1}= b_1b_2^{-1}$. It follows easily from the definitions that if $S$ has more than $(1-o(1))n$ symbols occurring more than $(1-o(1))n$ times, then both
$A$ and $B$ have energy at least $(1-o(1))n^3$ and therefore are very close to cosets of some subgroups, which can further be  shown to be the same subgroup.

\subsection*{Rainbow matchings, Hamiltonian paths and cycles}
Transversals in Latin squares are closely related to rainbow subgraphs of edge-coloured graphs.
Recall that an edge-coloured graph is properly coloured if no two edges of the same colour share a vertex.
A \emph{matching} in a graph is a set of disjoint edges. We call a subgraph of a graph \emph{rainbow} if all of its edges have different colours. There is a one-to-one correspondence between $n\times n$ generalized Latin squares and proper edge-colourings of the complete bipartite graph $K_{n,n}$. Indeed, given a generalized Latin square $S=(s_{ij})$ with $m$ symbols in total, associate with it an $m$-edge-colouring of $K_{n,n}$ by setting $V(K_{n,n})=\{x_1, \dots, x_n, y_1, \dots, y_n\}$ and letting the colour of the edge $(x_i,y_j)$ be $s_{ij}$. Notice that this colouring is proper, i.e., adjacent edges receive different colours. Therefore the study of transversals in generalized Latin squares is equivalent to the study of perfect rainbow matchings in proper edge-colourings of $K_{n,n}$.
Moreover, if $S$ is symmetric, i.e.\  $s_{ij}=s_{ji}$ for all $i$ and $j$, it also defines the proper edge-colouring of the complete $n$-vertex graph $K_n$ in which the edge $ij$ is coloured by $s_{ij}$. Since $S$ is symmetric each edge has a well-defined colour. {Under this second correspondence, transversals give rainbow maximum degree $2$ subgraphs of $K_n$.}

As explained above, partial transversals in the Latin square $S$ correspond to rainbow matchings in the corresponding edge-coloured $K_{n,n}$. Thus Conjecture \ref{BRS} is equivalent to the statement that any proper $n$-edge-colouring of $K_{n,n}$ contains a rainbow matching of size $n-1$. Theorem~\ref{Theorem_LatinSquare} then follows from the following statement.
\begin{theorem}\label{Theorem_Matchings}
There is an $\alpha>0$ so that the following holds for all $1>\varepsilon\geq  n^{-\alpha}/\alpha$. Let $K_{n,n}$ be properly coloured with at most $(1-\varepsilon)n$ colours having more than $(1-\varepsilon)n$ edges. Then, $K_{n,n}$ has $(1-\varepsilon)n$ edge-disjoint perfect rainbow matchings.
\end{theorem}

\noindent
We can also find perfect rainbow matchings in graphs that are more general than $K_{n,n}$. Our proof works for all suitably pseudorandom properly coloured balanced bipartite graphs. See Lemma~\ref{Lemma_PerfectMatching_Decomposition_ExtraPseudorandom} for an example of such a result.

There is a lot of interest in studying rainbow structures in properly coloured complete graphs.
{Recall that transversals in symmetric generalized Latin  squares correspond to rainbow maximum degree $2$ subgraphs of properly coloured complete graphs.} Since paths and cycles are a special type of maximum degree $2$ subgraph,
there has been a focus on finding nearly spanning rainbow paths/cycles in properly coloured complete graphs.
For example, Andersen \cite{andersen1989hamilton} in 1989 conjectured that all properly coloured $K_n$ have a rainbow path of length $n-2$.
Hahn conjectured even more, that such a path can be found in any (not necessarily properly) coloured complete graph with at most $n/2-1$ edges of each colour (see~\cite{hahn1986path}). Hahn's conjecture was recently disproved by the second and third author \cite{pokrovskiy2017counterexample}, who showed that without the ``proper colouring'' assumption the graph might not have rainbow paths longer than $n-\Omega(\log n)$. Thus it makes sense to restrict ourselves to colourings which are proper.
The progress on Andersen's conjecture was slow, despite efforts by various researchers, e.g., see \cite{akbari2007rainbow, gyarfas2010rainbow, gyarfas2011long, gebauer2012rainbow, chen2015long}.
Until recently it was not even known how to find a rainbow path/cycle of length $(1-o(1))n$. This was proved by Alon and the second and third author \cite{alon2016random}, who showed that
any properly coloured $K_n$ contains a rainbow cycle with $n-O(n^{3/4})$ vertices.
Using our techniques one can say much more, i.e., we can nearly-decompose such a complete graph into long rainbow cycles.
This is a corollary of the following theorem.

\begin{theorem}\label{Theorem_Hamiltonian}
There is an $\alpha>0$ so that the following holds for all $1>\varepsilon\geq  n^{-\alpha}/\alpha$.
Let  $K_{n}$ be a properly coloured with at most $(1-\varepsilon)n$ colours having more than $(1-\varepsilon)n/2$ edges. Then, $K_{n}$ has $(1-\varepsilon)n/2$ edge-disjoint rainbow Hamiltonian cycles.
\end{theorem}

\begin{corollary}\label{Corollary_Hamiltonian}
There is an $\alpha>0$ so that the following holds for all $1>\varepsilon\geq  n^{-\alpha}/\alpha$.
Given a properly coloured $K_{n}$ let $U$ be a random subset of $(1-\varepsilon)n$ vertices. Then, with high probability, the subgraph induced by $U$ has $(1-2\varepsilon)n/2$ edge-disjoint rainbow Hamiltonian cycles.
\end{corollary}

\subsection*{Rainbow spanning trees}
In this paper we also study spanning rainbow trees in properly coloured complete graphs. Notice that a rainbow Hamiltonian path is a very special case of a rainbow tree. Because of this one expects the results which hold for rainbow spanning trees to be stronger than ones for paths. For example, every properly coloured $K_n$ contains a rainbow spanning tree (a star at any vertex is rainbow), whereas it is known that there are proper edge-colourings of $K_n$ without rainbow Hamiltonian paths. In fact, much more is probably true. It was conjectured by a number of authors that properly coloured complete graphs should always have decompositions into spanning rainbow trees.
\begin{conjecture}[Brualdi and Hollingsworth, \cite{brualdi1996multicolored}]\label{Conjecture_Brualdi_Hollongsworth}
Every properly $(2n-1)$-coloured $K_{2n}$ can be decomposed into edge-disjoint rainbow spanning trees.
\end{conjecture}
\begin{conjecture}[Kaneko, Kano, and Suzuki, \cite{kaneko2003three}]\label{Conjecture_Kaneko_Kano_Suzuki}
Every properly coloured $K_{n}$ contains $\lfloor n/2\rfloor$ edge-disjoint rainbow spanning trees.
\end{conjecture}
These conjectures attracted a lot of attention from various researchers (see, e.g., \cite{akbari2007multicolored, carraher2016edge, fu2016number}) who showed how to find several disjoint  spanning rainbow trees. The best known results for these problem guarantee the existence of $\varepsilon n$ edge-disjoint rainbow trees (see \cite{horn2018rainbow} for Conjecture~\ref{Conjecture_Brualdi_Hollongsworth} and \cite{balogh2017rainbow,pokrovskiy2017linearly} for Conjecture~\ref{Conjecture_Kaneko_Kano_Suzuki}).
Developing our results on Hamiltonian cycles, we are able  to improve this and show that one can find  $(1-o(1))n$ disjoint spanning  rainbow trees.

\begin{theorem}\label{Theorem_Trees}
There is an $\alpha>0$ so that the following holds for all $1>\varepsilon\geq  n^{-\alpha}/\alpha$.
Every properly coloured $K_n$ has $(1-\varepsilon)n/2$ edge-disjoint spanning rainbow trees.
\end{theorem}
This theorem proves an asymptotic version of the Brualdi-Hollingsworth and Kaneko-Kano-Suzuki conjectures. Note that unlike our results about perfect matchings and Hamiltonian cycles,
which require certain small additional conditions, this theorem is true for all proper edge-colourings.

\section{Proof overview}
Our various rainbow decomposition results build on each other. First we find decompositions into rainbow perfect matchings, then into rainbow Hamiltonian cycles, and then into rainbow spanning trees. There are other rainbow structures that we find in between these --- the actual sequence of our proofs is the following:
\begin{enumerate}
\item Near-decompositions of nearly-regular balanced  bipartite graphs into nearly-perfect rainbow matchings.
\item Near-decompositions of typical balanced  bipartite graphs into perfect rainbow matchings.
\item Near-decompositions of typical   graphs into  rainbow 2-factors.
\item Near-decompositions of typical   graphs into  rainbow Hamiltonian cycles.
\item Near-decompositions of complete  graphs into  rainbow spanning trees.
\end{enumerate}
The following definitions make precise various terms in this overview.
\begin{itemize}
\item \textbf{Near-decomposition:} A near-decomposition of a graph $G$ is a set of edge-disjoint subgraphs $H_1, \dots, H_t$ in $G$ which cover almost all the edges of $G$, i.e.\ which have $e(H_1\cup \dots \cup H_t)=(1-o(1))e(G)$.
\item\textbf{Average degree:} The average degree of a graph $G$ is $\ad(G)=2e(G)/v(G)$.
\item\textbf{Nearly-regular:} A graph is nearly regular if all the vertices $v\in V(G)$ have $d(v)=(1\pm o(1))\frac{2e(G)}{v(G)}$, i.e.\ if all its degrees are close to each other.
\item\textbf{Typical:} A graph is typical if any pair of vertices $u,v\in V(G)$ has $d(u,v)=(1\pm o(1))\frac{4e(G)^2}{v(G)^3}$, i.e.\ if all its codegrees are close to each other. This is the main notion of pseudorandomness that we use in this paper.
\item\textbf{Global boundedness:} A coloured graph is globally $b$-bounded if it has $\leq b$ edges of each colour.
\item\textbf{2-factor:} A 2-factor is a collection of vertex-disjoint cycles which span all the vertices of a graph.
\item\textbf{Balanced bipartite:} A graph is balanced bipartite if its vertices can be partitioned into two sets of the same size, so that all the edges lie between the two sets.
\end{itemize}

\subsection{Nearly-perfect rainbow matchings}
There are two main results we prove about nearly-perfect rainbow matchings --- one finds a single nearly-perfect rainbow matching in a graph, the other nearly-decomposes a graph into them. The following is an informal description of the first result:
\begin{result}
Every properly coloured, nearly-regular, globally $\ad(G)$-bounded, balanced bipartite graph $G$ has a rainbow matching $M$ of order $(1-o(1))|V(G)|/2$. Additionally, $M$ can be chosen probabilistically so that every edge of $G$ is in $M$ with roughly the same probability.
\end{result}
The precise statement of this is Lemma~\ref{Lemma_near_perfect_matching}. The proof uses R\"odl's semi-random method together with some extra ideas.
The key point in A1 is that the matching it produces is randomized. Given a properly coloured, nearly-regular, globally $\ad(G)$-bounded, balanced bipartite graph we can repeatedly apply A1 in order to produce a sequence of disjoint nearly-perfect rainbow matchings $M_1, \dots, M_t$. We can keep iterating this as long as the remaining graph satisfies the assumptions of A1 (near-regularity and global boundedness). Using the fact that the matching in A1 is randomized we can show that with high probability we can iterate A1 until there are $o(|V(G)|^2)$ edges left in the graph, i.e.\ until we have a near-decomposition into nearly-perfect rainbow matchings:
\begin{result}\label{Result_near_matching_decomposition}
Every properly coloured, nearly-regular, globally $\ad(G)$-bounded, balanced bipartite graph $G$ can be nearly-decomposed into rainbow matchings of order $(1-o(1))|V(G)|/2$.
\end{result}
The precise statement of this is Lemma~\ref{Lemma_Nearly_Rainbow_Decomposition}.
The proof of A2 iterates A1 while ensuring that the assumptions of A1 are maintained. We show this using a martingale concentration inequality.

\subsection{Rainbow perfect matchings}
The basic result we prove about near-decompositions into perfect rainbow matchings is the following:
\begin{result}\label{Result_perfect_matching_decomposition}
Let $G$ be a properly coloured, nearly-regular, globally $\ad(G)$-bounded, balanced bipartite graph.
Let $H$ be a properly coloured, typical graph on $V(G)$ which is edge-disjoint and colour-disjoint from $G$.
Then $G\cup H$ has a near-decomposition into rainbow perfect matchings.
\end{result}
The precise statement of this is Lemma~\ref{Lemma_PerfectMatching_Decomposition_Typical}.
The assumptions of this lemma (that we have two disjoint graphs, one of which is typical and the other nearly-regular and globally bounded) will reoccur several times in this paper. We pause now to explain why these are natural assumptions under which to seek spanning rainbow structures.

We look at a nearly-regular, globally $\ad(G)$-bounded graph for two reasons. Firstly A\ref{Result_near_matching_decomposition} shows that under this assumption one can find rainbow nearly-perfect matchings (so it is reasonable to try to strengthen A\ref{Result_near_matching_decomposition} to get perfect matchings). Secondly, if one starts in any properly coloured $K_{n,n}$ and selects a random subgraph $G$ by choosing every colour independently with probability $p$ (and letting $G$ be the edges of the chosen colours), then the resulting subgraph will be a nearly-regular, globally $(1+o(1))\ad(G)$-bounded graph with high probability. We prove results about nearly-regular, globally $\ad(G)$-bounded graphs so that we can say things about random subgraphs of properly coloured complete graphs.

Unfortunately one cannot hope to find perfect rainbow matchings if one just considers a nearly-regular graph $G$. This is because nearly-regular graphs might have no perfect matchings at all (e.g.\ a disjoint union of two copies of $K_{n, n+1}$ is nearly-regular, balanced bipartite, and has no perfect matching). This is the motivation for the typical graph $H$ disjoint from $G$ in A\ref{Result_perfect_matching_decomposition}. The union of a nearly-regular graph $G$ and a typical graph $H$ has a perfect matching, making A\ref{Result_perfect_matching_decomposition} more plausible.

To prove A\ref{Result_perfect_matching_decomposition} we first apply A\ref{Result_near_matching_decomposition} to $G$ to get a near-decomposition of $G$ into nearly-rainbow matchings. Then we use edges of $H$ to modify the matchings one-by-one to turn them into perfect matchings. The modifications we use are simple switchings where we exchange 2 edges of a matching $M$ for 3 edges of $H$ in order to get a larger matching $M'$. Using a sequence of switchings we will obtain perfect matchings.

\subsubsection*{Proving Theorem~\ref{Theorem_Matchings}}
A\ref{Result_perfect_matching_decomposition} can be used to prove Theorem~\ref{Theorem_Matchings}. To do this, we need two intermediate results. The first concerns choosing a random set of colours in a properly coloured graph.
\begin{result}\label{Result_random_subgraph}
Let $G$ be properly coloured and typical. Choose every colour independently with probability $p$, and let $H$ be the subgraph formed by the edges of the chosen colours. Then, with high probability, $H$ is typical.
\end{result}
This result says that the subgraph chosen by a random set of colours is pseudorandom. A result like this was first used by Alon and the second and third author when studying rainbow cycles in graphs~\cite{alon2016random}.

Applying A\ref{Result_random_subgraph} to the complete bipartite graph $K_{n,n}$ from Theorem~\ref{Theorem_Matchings} gives a typical subgraph $H$ which can be used in A\ref{Result_perfect_matching_decomposition}. The graph $G$ formed by the colours from $K_{n,n}$ unused in $H$ will be nearly-regular with high probability. However, we cannot yet apply A\ref{Result_perfect_matching_decomposition} since the graph $G$ might not be globally $\ad(G)$-bounded. Indeed, $G$ may have colour classes of size $n$, whereas the average degree of $G$ will be $(1\pm o(1))(n-p)$ (where $p$ is the parameter from A\ref{Result_random_subgraph}). To get around this we have another intermediate result saying that there is a subgraph $G'$ of $G$ which is globally $\ad(G')$-bounded.
\begin{result}\label{Result_improve_boundedness}
Let $G$ be a properly coloured balanced bipartite graph with $\leq (1-\varepsilon)n$ colours having $\geq (1-\varepsilon)n$ edges and $\delta(G)\geq (1-\varepsilon^2)n$. Then $G$ has a spanning subgraph $G'$ with $\ad(G')\geq (1-2\varepsilon)n$
which is globally $\ad(G')$-bounded and nearly-regular.
\end{result}
See Lemma~\ref{Lemma_regular_subgraph_few_large_colours_bipartite} for a precise statement of A\ref{Result_improve_boundedness}.
This is proved in two stages. First, for every colour $c$ with $\geq (1-\varepsilon)n$ edges, we randomly delete every colour $c$ edge with a small probability $q$. The remaining graph $G_1$ will be globally $(1-o(1))\ad(G_1)$-bounded with high probability, but might no longer be nearly-regular. We then apply a ``regularization'' lemma to $G_1$ which deletes a small number of edges from $G_1$ to make it nearly-regular, without overly affecting the global boundedness. The resulting graph $G'$ is then globally $\ad(G')$-bounded and nearly-regular. Plugging $G'$ into A\ref{Result_perfect_matching_decomposition} together with the graph $H$ from A\ref{Result_random_subgraph} we obtain Theorem~\ref{Theorem_Matchings}.

\subsection{Rainbow 2-factors}
Rainbow 2-factors are intermediate structures we use between finding perfect matchings and Hamiltonian cycles.
The main result about 2-factors that we need is a direct analogue of A\ref{Result_perfect_matching_decomposition}.
\begin{result}\label{Result_2factor_decomposition}
Let $G$ be a properly coloured, nearly-regular, globally $\frac12\ad(G)$-bounded graph.
Let $H$ be a properly coloured, typical graph on $V(G)$ which is edge-disjoint and colour-disjoint from $G$.
Then, $G\cup H$ has a near-decomposition into rainbow $2$-factors.
\end{result}
See Lemma~\ref{Lemma_2_Factor_Decomposition} for a precise statement of this.
The main difference betwen A\ref{Result_perfect_matching_decomposition} and A\ref{Result_2factor_decomposition} is that the global boundedness in A\ref{Result_2factor_decomposition} is $\frac12\ad(G)$ (rather than $\ad(G)$ as it was in A\ref{Result_perfect_matching_decomposition}).  The reason for this is that to find a rainbow 2-factor we would need $|V(G)|$ colours in the graph, which is forced by global $\frac12\ad(G)$-boundedness (but not by $\ad(G)$-boundedness). Thus the global $\frac12\ad(G)$-boundedness condition is natural because it is the weakest global boundedness we can impose on the graph to guarantee enough colours for a rainbow $2$-factor

The proof of A\ref{Result_2factor_decomposition} consists of using A\ref{Result_perfect_matching_decomposition} to find matchings in the graph, which are then put together to get 2-factors. To see how we might do this, we randomly partition $V(G\cup H)$ and $C(G\cup H)$ into  vertex sets $U_1, \dots, U_k$ and colour sets $C_1, \dots, C_k$ of the same size.
Then, using variants of A\ref{Result_random_subgraph} we can show that the subgraphs $G_{C_i}[U_j, U_k]$ are nearly-regular, while the subgraphs $H_{C_i}[U_j, U_k]$ are typical. By  A\ref{Result_perfect_matching_decomposition}, these subgraphs have near-decompositions into families $\mathcal{M}_{i,j,k}$ of perfect rainbow matchings for all distinct $i,j,k$. By taking unions of these matchings for suitable $i,j,k$ we obtain rainbow 2-factors. I.e.,\ $\bigcup_{i=1}^k\mathcal M_{i,i+1\pmod k,i}$ is a family of rainbow 2-factors.

\subsection{Rainbow Hamiltonian cycles}
The main result about Hamiltonian cycles that we need is a direct analogue of A\ref{Result_perfect_matching_decomposition} and A\ref{Result_2factor_decomposition}.
\begin{result}\label{Result_Hamiltonian_decomposition}
Let $G$ be a properly coloured, nearly-regular, globally $\frac12\ad(G)$-bounded graph.
Let $H$ be a properly coloured, typical graph on $V(G)$ which is edge-disjoint and colour-disjoint from $G$.
Then $G\cup H$ has a near-decomposition into rainbow Hamiltonian cycles.
\end{result}
See Lemma~\ref{Lemma_Hamiltonian_Decomposition_gap} for a precise statement of this.
The proof of  A\ref{Result_Hamiltonian_decomposition} consists of first splitting the colours of $H$ at random into two subgraphs $H_1$ and $H_2$. Using a result like A\ref{Result_random_subgraph}, we have that $H_1$ and $H_2$ are both typical. Applying A\ref{Result_2factor_decomposition} to $G$ and $H_1$, we get a near-decomposition of $G\cup H_1$ into rainbow 2-factors. Then we use the typical graph $H_2$ to modify the 2-factors one-by-one into Hamiltonian cycles. This modification is done by ``rotations'' --- switching a small number of edges on a 2-factor for edges of $H_2$ in order to decrease the number of cycles in the 2-factor. After a small number of rotations like this, we create a Hamiltonian cycle.

Theorem~\ref{Theorem_Hamiltonian} is proved using A\ref{Result_Hamiltonian_decomposition}. The proof is similar to the proof of Theorem~\ref{Theorem_Matchings} --- starting with a properly coloured $K_n$, we use analogues of A\ref{Result_random_subgraph} and A\ref{Result_improve_boundedness} to get the graphs $G$ and $H$ needed in A\ref{Result_Hamiltonian_decomposition}.

\subsection{Rainbow spanning trees}
Here we explain the proof of Theorem~\ref{Theorem_Trees} --- that the Brualdi-Hollingsworth and Kaneko-Kano-Suzuki conjectures hold asymptotically. The starting point of this is to observe that a near-decomposition into rainbow Hamiltonian cycles gives a near-decomposition into rainbow spanning trees. Because of this, our results about Hamiltonian cycles have implications for spanning tree decompositions.
The first implication is that if we have a properly coloured $K_n$ with $\leq (1-\varepsilon)n$ colours having $\geq (1-\varepsilon)n/2$ edges, then this $K_n$ has a near-decomposition into rainbow spanning  trees (by Theorem~\ref{Theorem_Hamiltonian}).

Thus it remains to look at colourings of $K_n$ with   $\geq (1-\varepsilon)n$ colours having $\geq (1-\varepsilon)n/2$ edges. In this section we will focus on the case when the colouring has exactly $n-1$ colours each having exactly $n/2$ edges. This is the setting of the Brualdi-Hollingsworth Conjecture and is substantially easier to deal with. To deal with this case we need the following result on how the colours in a random subset of vertices behave.
\begin{result}\label{Result_random_subset}
Let $K_n$ be properly coloured and choose a subset of $(1-\varepsilon)n$ vertices $U\subseteq V(K_n)$ at random. Then, $K_n[U]$ is globally $(1-2\varepsilon)n/2$-bounded.
\end{result}
See Lemma~\ref{Lemma_Random_Subgraph_General} (c) for a precise statement of this.
Notice that the subgraph $K_n[U]$ from A\ref{Result_random_subset} is globally $(1-2\varepsilon)n/2$-bounded and has  $\ad(K_n[U])=(1-\varepsilon)n$.

Randomly partition $K_n[U]$ into graphs $G'$ and $J$, with every edge placed in $J$ independently with probability $p\ll \varepsilon$.
Then randomly partition the colours of $G'$ into sets $C_G$ and $C_H$, with each colour ending up in $C_H$ independently with probability $p$. Let $G''$ and $H$ be the subgraphs of $G'$ consisting of edges with colours in $C_G$ and $C_H$ respectively.
Using results like A\ref{Result_random_subset} it can be shown that $G''$, $H$, and $J$ are all nearly-regular and typical. Since $G''\subseteq G$, we have that $G''$ is also globally $(1-2\varepsilon)n/2$-bounded. Since $p\ll \varepsilon$ and $G$ had $\ad(G)=(1-\varepsilon)n$, we have that $\ad(G'')\approx (1-\varepsilon-2p)n\geq (1-2\varepsilon)n$. Thus $G''$ and $H$ satisfy the assumptions of A\ref{Result_Hamiltonian_decomposition}, which gives a near-decomposition of $G''\cup H$ into rainbow Hamiltonian paths.

We now have a set of rainbow paths of length $(1-\varepsilon)n$ and an edge-disjoint typical subgraph $J$. We turn the paths into spanning rainbow trees by extending each path one vertex at a time using edges of $J$. The operations we use to extend the trees are very simple: we always have a collection of rainbow trees $T_1, \dots, T_{(1-\varepsilon)n}$ which we want to enlarge. To enlarge a tree $T_i$, we find three edges $e_1, e_2, e_3$ outside $T_1, \dots, T_{(1-\varepsilon)n}$ and two edges $f_1, f_2$ on $T_i$ so that $T'_i=T_i\cup \{e_1, e_2, e_3\}\setminus \{f_1,f_2\}$ another rainbow tree. Replacing $T_i$ by $T_i'$ gives us a collection of larger rainbow trees, so by iterating this process we would eventually get rainbow spanning trees. The remaining question is then ``how can we find the edges $e_1, e_2, e_3$, $f_1, f_2$ which we use to enlarge $T_i$?'' This is where the typicality of the graph $J$ is used. The fact that $J$ is pseudorandom means that its edges are suitably spread out around $V(K_n)$, and allows us to find edges in $J$ to switch with edges of $T_i$.

\section{Preliminaries}
Here we collect some useful notation and results which will be used later in the paper.
\subsection{Basic notation}
For a graph $G$, the set of edges of $G$ is denoted by $E(G)$ and the set of vertices of $G$ is denoted by $V(G)$.
For a vertex $v$ in a graph $G$, the set of edges in $G$ through $v$ is denoted by $E_G(v)$, the set of colours of edges going through $v$ is denoted by $C_G(v)$, the set of neighbours of $v$ in $G $ is denoted by $N_G(v)$, and  $d_G(v)=|N_G(v)|$.
For a coloured graph $G$ and a colour $c$, the set of colour $c$ edges in $G$ is denoted by $E_G(c)$ and the set of vertices touching colour $c$ edges in $G$ is denoted by $V_G(c)$. In all of these, we omit the ``$G$'' subscript when the graph $G$ is clear from context.
We will use additive notation for adding and deleting vertices and edges from graphs.

For a graph $G$ and a set of vertices $A$, let $G[A]$ denote the induced subgraph of $G$ on $A$. For disjoint sets of vertices $A$ and $B$, we use $G[A,B]$ to denote the bipartite subgraph of $G$ on $A\cup B$ consisting of all edges between $A$ and $B$. For any event $E$, we let $\mathbf{1}_E$ be the indicator function for $E$, taking the value 1 when $E$ occurs, and 0 otherwise.

 For two functions $f(x_1, \dots, x_t)$ and $g(y_1, \dots, y_s)$, we use $f(\pm x_1, \dots, \pm x_t)  =  g(\pm y_1, \dots, \pm y_s)$ to mean that ``$\max_{\sigma_i\in \{-1, +1\}} f(\sigma_1 x_1, \dots, \sigma_t x_t)$ $\leq$ $\max_{\sigma_i\in \{-1, +1\}} g(\sigma_1 y_1, \dots, \sigma_s y_s)$ and also that $\min_{\sigma_i\in \{-1, +1\}} f(\sigma_1 x_1, \dots, \sigma_t x_t)$ $\geq$ $\min_{\sigma_i\in \{-1, +1\}} g(\sigma_1 y_1, \dots, \sigma_s y_s)$''.
The most frequently used case of this notation will to say  $x=y\pm z$ for some $z\geq 0$, in which case the notation  is equivalent to both ``$y-z\leq x\leq y+z$'' and ``$|x-y|\leq z$''.

Notice that $a=b\pm c$, $b=d\pm e\implies a=d\pm c\pm e$.
Also notice that for any $a,b, b'$ with $|b'|\geq |b|$, we have $a\pm b=a\pm b'$.
Finally notice that the notation is transitive $f(\pm x_1, \dots, \pm x_t)  =  g(\pm y_1, \dots, \pm y_s)$ and $  g(\pm y_1, \dots, \pm y_s) = h(\pm z_1, \dots, \pm z_r)  \implies f(\pm x_1, \dots, \pm x_t)  =  h(\pm z_1, \dots, \pm z_r)$.

We will often use the following which hold for any  $0\leq x<0.5$.
\begin{align}
\frac{1}{1-x}&\leq 1+2x \text{ and } \frac{1}{1+x}\geq 1-2x \label{Eq_Denominator_Bounds}\\
1-x&=(1\pm x^2)e^{-x} \label{Eq_Exponential_Bounds}\\
1+x&\leq e^{x} \label{Eq_Exponential_Bound_Plus}\\
(1-x)^t&\geq 1-tx \label{Eq_Power_Bound}\\
\sum_{i=1}^{T}e^{-(i-1)x}&=(1\pm x^2\pm 2e^{-xT})x^{-1} \label{Eq_Exponential_Bound_Sum}
\end{align}
The last inequality comes from  $\sum_{i=1}^{T}e^{-(i-1)x}=(1\pm x^2)\sum_{i=1}^{T}(1-x)^{i-1}=(1\pm x^2)\frac{1-(1-x)^T}{x}=(1\pm x^2\pm 2e^{-xT})x^{-1}$.

Throughout the paper most of our results will be either about balanced bipartite graphs or about general graphs. When dealing with balanced bipartite graphs, they will always come with a specific bipartition into two parts usually labelled by ``$X$'' and ``$Y$'' with $|X|=|Y|=n$. When dealing with general graphs, they will usually have $v(G)=n$.
Whenever we define a graph $G$, if we do not specifically say that $G$ is balanced bipartite, we implicitly mean that $G$ is a general graph.

We make a two definitions about graphs, which vary slightly depending on whether the graph they are talking about is balanced bipartite or not.
\begin{definition}
\

\begin{itemize}
\item A balanced bipartite graph $G$ with parts $X$ and $Y$ is $(\gamma, \delta,n)$-regular if $|X|=|Y|=(1\pm \gamma)n$ and $d_G(v)=(1\pm \gamma)\delta n$ for every vertex $v\in V(G)$.
\item A general graph $G$ is $(\gamma, \delta,n)$-regular if $|G|=n$ and $d_G(v)=(1\pm \gamma)\delta n$ for every vertex $v\in V(G)$.
\end{itemize}
\end{definition}

\begin{definition}\label{Definition_Typical}
\

\begin{itemize}
\item A balanced bipartite graph $G$ with parts $X$ and $Y$ graph is $(\gamma, \delta, n)$-typical if it is $(\gamma, \delta, n)$-regular and  we have $d(x,y)=(1\pm \gamma)\delta^2 n$ for any pair of vertices $x,y\in X$ or $x,y\in Y$.
\item A general graph is $(\gamma, \delta, n)$-typical if it is $(\gamma, \delta, n)$-regular and for any pair of vertices $x,y$ we have $d(x,y)=(1\pm \gamma)\delta^2 n$.
\end{itemize}
\end{definition}

\begin{definition}
 A  graph $G$ is globally $b$-bounded if $G$ has $\leq b$ edges of  each colour, i.e.\ if $|E_G(c)|\leq b$ for all colours $c$.
\end{definition}

\begin{definition}
 A  graph $G$ is locally $\ell$-bounded if $G$ has $\leq \ell$ edges of  each colour passing through any vertex $v\in V(G)$, i.e.\ if $\Delta(E_G(c))\leq \ell$ for all colours $c$.
\end{definition}

\subsection{Asymptotic notation}
For a number $C\geq 1$ and $x,y\in (0,1]$, we use ``$x\ll_C y$'' to mean ``$x\leq \frac{y^C}{C}$. We will write ``$x\LL y$'' to mean that there is some absolute constant $C$ for which the proof works with ``$x\LL{} y$'' replaced by ``$x\ll_{C} y$''. This notation parallels more standard notation $x\ll y$ which means ``there is a fixed positive continuous function $f$ on $(0,1]$ for which the remainder of the proof works with ``$x\ll y$'' replaced by ``$x\leq f(y)$'' (equivalently ``$x\ll y$'' can be interpreted as ``for all $x\in (0,1]$, there is some $y\in (0,1]$ such that the remainder of the proof works with $x$ and $y$''). The two notations ``$x\LL{} y$'' and ``$x\ll y$'' are largely interchangeable --- most of our proofs remain correct with all instances of ``$\LL$'' replaced by ``$\ll$''. The advantage of using ``$\LL$'' is that it proves polynomial bounds on the parameters (rather than bounds of the form ``for all $\varepsilon>0$ and  sufficiently large $n$''). This is important towards the end of this paper, where the proofs need polynomial bounds on the parameters.

While the constants $C$ will always be implicit in each instance  of ``$x\LL{} y$'',  it is possible to work  them out explicitly. To do this one should go through the lemmas in the paper in numerical order, choosing the constants $C$ for earlier lemmas before later lemmas. This is because an inequality $x\ll_C y$ in a later lemma may be needed to imply an inequality $x\ll_{C'} y$ from an earlier lemma. Within an individual lemma we will often have several inequalities of the form $x\LL y$. There the constants $C$ need to be chosen in the reverse order of their occurrence in the text. The reason for this is the same --- as we prove a lemma we may use an inequality  $x\ll_C y$  to imply another inequality $x\ll_{C'} y$ (and so we should choose $C'$ before choosing $C$).

Throughout the paper, there are four operations we perform with the ``$x\LL{} y$'' notation:
\begin{enumerate}[(a)]
\item We will use $x_1\LL x_2\LL\dots\LL x_k$ to deduce finitely many inequalities of the form ``$p(x_1, \dots, x_k)\leq q(x_1, \dots, x_k)$'' where $p$ and $q$ are {monomials} with non-negative coefficients and $\min\{i: p(0, \dots, 0, x_{i+1}, \dots, x_k)=0\}< \min\{j: q(0, \dots, 0, x_{j+1}, \dots, x_j)=0\}$  e.g.\ $1000x_1\leq x_2^5x_4^2x_5^3$ is of this form.
\item We will use $x\LL y$ to deduce finitely many inequalities of the form ``$x\ll_C y$'' for a fixed constant $C$.
\item For $x\LL y$ and  fixed constants $C_1, C_2$, we can choose a variable $z$ with $x\ll_{C_1} z\ll_{C_2}y$.
\item For $n^{-1}\LL 1$ and any fixed constant $C$, we can deduce $n^{-1}\ll_C \log^{-1} n\ll_C 1$.
\end{enumerate}

{To see that (a) is possible, we need to show that for any finite collection $\mathcal I$ of inequalities of the given form, we can choose constants $C_1, \dots, C_{k-1}$ so that $0<x_1\ll_{C_1} x_2\ll_{C_2}\dots\ll_{C_{k-1}} x_k<1$ implies all the inequalites in  $\mathcal I$.
To see this, first consider a single inequality ``$p(x_1, \dots, x_k)\leq q(x_1, \dots, x_k)$'' of the form in (a).  From the assumptions on $p$ and $q$, we know that $p(x_1, \dots, x_k)=D_px_1^{\ell_1}\dots x_k^{\ell_k}$ and $q(x_1, \dots, x_k)=D_qx_1^{r_1}\dots x_k^{r_k}$ for some $D_p, D_q>0$ and $\min\{i: \ell_i\neq 0\}< \min\{i: r_i\neq 0\}$. Now, it is easy to check that for $C = r_1+\dots+r_k+D_p/D_q$, we have $0<x_1\LL_C x_2\LL_C\dots\LL_C x_k<1$ $\implies$ $p(x_1, \dots, x_k)\leq q(x_1, \dots, x_k)$.
Now given a finite collection $\mathcal I$ of inequalities of the given form, for each $I\in \mathcal I$, we can choose a constant $C_I$ so that $0<x_1\LL_{C_I} x_2\LL_{C_I}\dots\LL_{C_I} x_k<1 \implies I$. Letting $C=\max_{I\in \mathcal I} C_I$ gives a single constant for which $0<x_1\LL_C x_2\LL_C\dots\LL_C x_k<1$ implies all the inequalities in $\mathcal I$.

We remark that occasionally we will use a slight strengthening of (a), when $p$ and $q$ are  \emph{multinomials} with non-negative coefficients and $\min\{i: p(0, \dots, 0, x_{i+1}, \dots, x_k)=0\}< \min\{j: q(0, \dots, 0, x_{j+1}, \dots, x_j)=0\}$  e.g.\ $50x_1x_2+5x_2^2\leq x_3^5x_4^2 + x_1^2x_5^3$ is of this form. This strengthening can be reduced to the monomial version. To do this, consider multinomials $p$ and $q$ with non-negative coefficients and an integer $i$ for which $p(0, \dots, 0, x_{i+1}, \dots, x_k)=0$ and $q(0, \dots, 0, x_{i+1}, \dots, x_k)\neq 0$. Let $D_p$ be the sum of the coefficients of $p$ and notice that the monomial $\hat p=D_p x_i$ satisfies $\hat p\geq p$ (for $0<x_1\leq\dots\leq x_k<1$).
Letting $D_q$ be the smallest coefficient of $q$ and $d$ the degree of $q$, notice that the monomial $\hat q= D_qx_{i+1}^{d}$ satisfies $\hat q\leq q$ (for $0<x_1\leq\dots\leq x_k<1$).  Thus we can use the monomial version of (a) to get constants $C_1, \dots, C_{k-1}$ so that $0<x_1\ll_{C_1} x_2\ll_{C_2}\dots\ll_{C_{k-1}} x_k<1$ implies $\hat p\leq \hat q$ and hence also $p\leq q$.}

Notice that (b) is just a special case of (a) since the inequality ``$x\ll_C y$'' is of the form of the inequalities in (a). Operation (b) is important because it allows us to plug one instance of the ``$\LL$'' notation into another one. As an example, suppose that we have proved a lemma which assumes ``$a\LL b$''. This means that we have proved that there is some explicit constant $C$ for which the lemma holds with ``$a\LL b$'' replaced by ``$a\ll_C b$''. Now if we subsequently have variables $x,y$ with $x\LL y$, then (b) guarantees that we can plug $x$ and $y$ into the earlier lemma with $a=x$ and $b=y$.

For operation (c), notice that for $C=C_1C_2^{C_1C_2}$, if we have numbers $x,y$ with  $x\ll_C y$ then the number $z=y^{C_2}/C_2$ satisfies $x\ll_{C_1} z \ll_{C_2}y$. Operation (c) is important  because it allows us to introduce new variables inside our proof. For example if we have a lemma which assumes $x\LL y$, then in the proof of the lemma we can say ``choose $z$ with $x\LL z\LL y$''. Here the constants $C_1$ and $C_2$ in ``$x\ll_{C_1} z\ll_{C_2} y$'' are chosen first, and operation (c) guarantees that we can later choose a constant for ``$x\ll_C y$''.

For operation (d), notice that ``$n^{-1}\ll_C \log^{-1} n\ll_C 1$'' is equivalent to ``$\frac1{C^{1/c}}n^{1/c}\geq \log n\geq C$'' which is true for sufficiently large $n$. Operation (d) is important because it allows us to use $n^{-1}\LL 1$ to deduce any instance of $n^{-1}\LL \log^{-1} n\LL 1$.

How does our ``$\LL$'' notation compare with the standard  ``$\ll$'' notation? Versions of the operations (a), (b), and (c) work with  the  ``$\ll$'' notation as well. Particularly (a) is more versatile with  ``$\ll$'', because it is possible to show that $x_1\ll x_2\ll\dots\ll x_k$ can be used to deduce finitely many inequalities of the form ``$p(x_1, \dots, x_k)\leq q(x_1, \dots, x_k)$'' where $p$ and $q$ are {arbitrary positive continuous functions on $(0,1]$ satisfying $\min\{i: p(0, \dots, 0, x_{i+1}, \dots, x_k)=0\}< \min\{j: q(0, \dots, 0, x_{j+1}, \dots, x_j)=0\}$} (rather than multinomials).
Operation (d) however has no analogue for the ``$\ll$'' notation (the natural analogue would be that ``for  $n^{-1}\ll 1$ and any positive continuous $f, g$  on $(0,1]$ we can deduce $n^{-1}\leq f(\log^{-1} n) \leq g(1)$''. However this is not true for $f(x)=0.5e^{-1/x}$). Because of this, in our proofs the ``$\LL$'' and ``$\ll$'' notations are interchangeable whenever operation (d) is not used (while when operation (d) is used, we need to use  the ``$\LL$'' notation).

\subsection{Probabilistic tools}
We will use the following cases of the Bonferroni Inequalities.
\begin{lemma}[Bonferroni Inequalities]\label{BIineq}{
Let $X_1,\ldots,X_n$ be events in a probability space. Then,
\begin{align*}
\P(\cup_{i=1}^nX_i)\geq \sum_{i=1}^n\P(X_i)-\sum_{i=1}^n\sum_{j=1}^{i-1}\P(X_i\cap X_j).
\end{align*}}
\end{lemma}

Given a probability space $\Omega=\prod_{i=1}^n\Omega_i$ and a random variable $X:\Omega\to \mathbb{R}$ we make the following definitions.
\begin{itemize}
\item Supose that there is a constant $C$ such that changing $\omega\in \Omega$ in any one coordinate changes $X(\omega)$ by at most $C$. Then we say that $X$ is $C$-Lipschitz.
\item For $i \in \{1, \dots, n\}$ we say that $X$ is uninfluenced by $i$ if $\omega_j=\omega'_j$ for $j\neq i \implies X(\omega)=X(\omega')$. Otherwise we say that $X$ is influenced by $i$.
\end{itemize}

We will use the following concentration inequalities
\begin{lemma}[Azuma's Inequality]\label{Lemma_Azuma}
Suppose that $X$ is $C$-Lipschitz and influenced by $\leq m$ coordinates in $\{1, \dots, n\}$. Then{, for any $t>0$,}
$$\P\left(|X-\E(X)|>t \right)\leq 2e^{\frac{-t^2}{mC^2}}$$
\end{lemma}

Notice that the bound in the above inequality can be rewritten as $\P\left(X\neq \E(X)\pm t \right)\leq 2e^{\frac{-t^2}{mC^2}}$.
A sequence of random variables $X_0, X_1, X_2, \dots$ is a supermartingale if $\E(X_{t+1}| X_{0}, \dots, X_t)\leq X_t$  for all $t$.

{
\begin{lemma}[Azuma's Inequality for Supermartingales]\label{Lemma_Azuma_supermartingales}
Suppose that $Y_0,Y_1,\ldots,Y_n$ is a supermartingale with $|Y_{i}-Y_{i-1}|\leq C$ for each $i\in [n]$. Then, for any $t>0$,
$$\P\left(Y_n> Y_0+ t \right)\leq e^{\frac{-t^2}{2nC^2}}$$
\end{lemma}
}
\begin{lemma}[Chernoff Bound]\label{Chernoff}\label{Lemma_Chernoff}
Let $X$ be the binomial random variable with parameters $(n,p)$.
Then for $\varepsilon\in (0,1)$ we have
$$\mathbb P\big(|X-pn|> \varepsilon pn\big)\leq
2e^{-\frac{pn\varepsilon^2}{3}}.$$
\end{lemma}
\begin{lemma}[Greenhill, Isaev, Kwan,  McKay]\label{Lemma_Kwan}
Let $\binom{[N]}{r}$ be the set of $r$-subsets of $\{1, \dots, N\}$ and let $h:\binom{[N]}{r}\to \mathbb{R}$ be given. Let $C$ be a uniformly random element of $\binom{[N]}{r}$. Suppose that there exists $\alpha\geq 0$ such that $|h(A)-h(A')|\leq \alpha$ for any $A, A'\in \binom{[N]}{r}$ with $|A\cap A'|=r-1$. Then for any $t>0$,
$$\mathbb P\big(|h(C)-\E h(C)|\geq t\big)\leq
2e^{-\frac{2t^2}{\alpha^2\min(r, N-r)}}.$$
\end{lemma}

\section{Finding one rainbow matching probabilistically}
The goal of this section is to prove that every properly coloured $d$-regular, globally $(1+o(1))d$-bounded balanced bipartite graph has a nearly-spanning rainbow matching $M$.
This matching is found using a randomized process, which allows us to prove that every edge ends up in $M$ with at least the expected probability ${d}^{-1
}$.
It will be more convenient for us to prove the result for graphs which are approximately regular rather than regular. Thus, throughout this section we will always be deal with $(\gamma, \delta, n)$-regular graphs for suitable parameters. See Lemma~\ref{Lemma_near_perfect_matching} for a precise statement of the result we prove.

The random process that we use to find a rainbow matching is a variation of the semi-random method introduced by R\"odl.
We remark that in the case when the graph $G$ has exactly $d$ edges of each colour, then our results follow directly from standard versions of the R\"odl Nibble (this is done by first expressing the problem in terms of finding a matching in an uncoloured $3$-uniform hypergraph, and then using e.g.\ Theorem~4.7.1 from \cite{alon2004probabilistic}).
Thus the difficult case of the result we aim to prove is when $G$ is a graph in which some colour classes have size much smaller than $d$. We deal with this situation by using a balancing coin flips approach to keep our graphs nearly-regular.

\subsubsection*{Random process}
Let $G$ be a coloured balanced bipartite graph which is $(\gamma, \delta, n)$-regular and globally $(1+o(1))d$-bounded. We describe a randomized process which will find a rainbow matching $M$ of size $(1-o(1))n$ in $G$ with high probability.
The process will last for $T$ rounds.
In each round we will focus on some subgraph $G_t$ of $G$ and partition $G_t$ into a rainbow matching $M_t$ and a vertex-disjoint, colour-disjoint graph $G_{t+1}$. At the end of the process we will have a collection of vertex-disjoint, colour-disjoint matchings $M_1, \dots, M_T$, and so letting $M=M_1\cup\dots \cup M_T$ we get a rainbow matching. We will prove that with high probability $e(M)=(1-o(1))n$.

\subsubsection*{Individual rounds}
To partition $G_t$ into  $M_{t+1}$ and $G_{t+1}$, in each round we use a random process which we call an \emph{$(\alpha, b)$-random edge-assignment}.
Let the parts of the bipartition of $G_t$ be called $X$ and $Y$.
The definition of the $(\alpha, b)$-random edge-assignment is the following:
\begin{itemize}
\item First we activate every vertex of $X$ with probability $\alpha$.
\item For every activated vertex $x$ we choose a random neighbour $y_x$  of $x$  in $Y$.
\item Let $M_{t+1}$ be the largest matching formed by the isolated edges of the form $xy_x$ whose colour is not the colour of any other chosen edge $x'y_{x'}$.
\item Let $H$ be the subgraph of $G$ on {$V(G)\setminus V(M_{t+1})$} consisting of all the edges whose colours do not occur on any chosen edge.
\item Delete every edge $xy$ on $H$ with probability $\frac{\alpha {b}}{d(x)}-\frac{\alpha |E({c(xy)})|}{d(x)}$ to get $G_{t+1}$.
\end{itemize}
Suppose that $G_t$ is $(\gamma_t, \delta_t, n_t)$-regular and globally $(1+\gamma_t)\delta_tn_t$-bounded.
We will run an $(\alpha, (1+\gamma_t)\delta_t n_t)$-random edge-assignment on $G_t$ and estimate the probabilities of edges and vertices of $G_t$ ending up in $M_t$ or $G_{t+1}$.
\begin{align}
\P(v\in  V(G_{t+1}))&\approx 1-\alpha\approx e^{-\alpha} &\text{ for any vertex   $v\in V(G_t)$} \notag\\
\P(\{x,y\} \subseteq V(G_{t+1}))&\approx 1-2\alpha\approx e^{-2\alpha} &\text{ for any pair   $\{x, y\}\subseteq V(G_t)$} \notag\\
\P(y\in N_{G_{t+1}}(x)|x\in V(G_t))&\approx 1-2\alpha\approx e^{-2\alpha} &\text{ for any  $y\in N_{G_t}(x)$} \notag\\
\P(e\in E(G_{t+1}))&\approx 1-3\alpha \approx e^{-3\alpha}  &\text{ for any edge   $e\in E(G_t)$} \label{Eq_Proof_Sketch_HProb}\\
\P(e\in E(M_{t+1}))&\approx\frac{\alpha}{\delta_tn_t} &\text{ for any edge   $e\in E(G_t)$} \label{Eq_Proof_Sketch_MProb}
\end{align}
Using linearity of expectation, we can estimate  the expected number of vertices, degrees of vertices, and sizes of colour classes in $G_{t+1}$.
\begin{align*}
\E(|X\cap V(G_{t+1})|)&=\E(|Y\cap V(G_{t+1})|)\approx  e^{-\alpha} n_t\\
\E(|E_{G_{t+1}}(c)|)&\lesssim e^{-2\alpha}\delta_t n_t&\text{ for any colour $c$}\\
\E(d_{G_{t+1}}(x))&\approx e^{-2\alpha}\delta_tn_t = (e^{-\alpha}\delta) (e^{-\alpha}n_t) &\text{ for any vertex $x\in V(H_{\omega})$}
\end{align*}
It can be shown that the quantities above are Lipschitz, and so by Azuma's Inequality they are concentrated around their expectation with high probability. This implies that with high probability $G_{t+1}$ is  $(\gamma_{t+1}, e^{-\alpha}\delta_t, e^{-\alpha}n_t)$-regular and globally $(1+\gamma_{t+1})(e^{-\alpha}\delta_t)(e^{-\alpha} n_t)$-bounded for some suitable error $\gamma_{t+1}$.

\subsubsection*{Iterating}
Let $G_0=G$ be a coloured graph which is $(\gamma, \delta, n)$-regular and globally $(1+o(1))\delta n$-bounded. We iteratively construct graphs $G_1, \dots, G_T$ and matchings $M_1, \dots, M_{T}$---at step $t$ we run an $(\alpha, (1+o(1))e^{-2\alpha t}\delta n)$-random edge-assignment on $G_t$ in order to obtain $M_{t+1}$ and $G_{t+1}$.

From the previous section we have that, with suitable errors $\gamma_1, \dots, \gamma_T$, the following hold for all $t$ with high probability:
\begin{enumerate}[(i)]
\item $G_t$ is $(\gamma_t, e^{-\alpha t}\delta, e^{-\alpha t}n)$-regular.
\item $G_t$ is globally $(1+\gamma_t)(e^{-\alpha t}\delta)(e^{-\alpha t}n)$-bounded.
\end{enumerate}
In particular, if $T= \omega(\alpha^{-1})$, then (i) implies that $|V(G_{T})|\lesssim (1+\gamma_T)e^{-\alpha T}n= o(n)$. Since $M_1, \dots, M_{T}$ are vertex-disjoint, colour-disjoint rainbow  matchings with $|V(G)|=|V(G_{T})|+|\bigcup_{i=1}^{T}V(M_i)|$, we get that $M=\bigcup_{i=1}^{T}V(M_i)$ is a rainbow matching of size {$(1-o(1))n$} in $G$.

\subsubsection*{Showing that the matching is random}
It remains to show that for any edge $e\in E(G)$, the probability that $e$ is in $M$ is (approximately) at least $(\delta n)^{-1}$.
First notice that (\ref{Eq_Proof_Sketch_HProb}) implies
$\P(e\in E(G_t))
=\prod_{i=0}^t \P(e\in G_i| e\in G_{i-1})
\gtrsim  e^{-3t\alpha}$.
Combining this with (\ref{Eq_Proof_Sketch_MProb}), and \eqref{Eq_Exponential_Bound_Sum}, we get
\begin{align*}
\P(e\in E(M_1\cup \dots \cup M_T))&=\sum_{t=0}^T \P(e\in G_t)\P(e\in M_{t+1}|e\in G_t)\\
&\gtrsim \sum_{t=0}^T \left(e^{-3t\alpha}\right)\left(\frac{\alpha}{(e^{-\alpha t}\delta)(e^{-\alpha t}n)} \right)
= \frac{\alpha}{\delta n}\sum_{t=0}^T e^{-\alpha t}\gtrsim \frac{1}{\delta n}.
\end{align*}
This concludes the proof sketch in this section. The main thing we need to do in the full proof is to keep track of the errors $\gamma_t$ and make sure that they do not get too big.

\subsection{Formal definition of the random edge assignment}
Here we formally define the probability space of the  $(\alpha, b)$-random edge-assignment which runs on a graph $G$.
The process will depend on two parameters $\alpha$ and ${b}$.
The graph $G$ will be a globally ${b}$-bounded balanced bipartite graph with parts $X$ and $Y$.
The process has a coordinate for every vertex in $X$, and a coordinate for every edge $e\in E(G)$ (the balancing coin flips):
\begin{itemize}
\item \textbf{Vertex choices:} For $x\in X$, the vertex $x$ is activated with probability $\alpha$. Every activated vertex \emph{chooses} a neighbour $y_x$ of $x$ uniformly at random from its neighbours.
\item \textbf{Balancing coin flips:} For $xy\in E(G)$, the edge $xy$ is \emph{killed} with probability $\frac{\alpha {b}}{d(x)}-\frac{\alpha |E({c(xy)})|}{d(x)}$.
\end{itemize}
We say that an edge $xy\in E(G)$ is \emph{chosen} if $x$ is activated and chooses $y$. We say that a colour $c$ is \emph{chosen} if some colour $c$ edge is chosen.
We construct a matching $M$ and graphs $\Gamma, H$ depending on the process as follows.
\begin{align*}
M&=\{xy\in E(G): \text{$xy$ is chosen, and no $x'y'\in E(G)\setminus\{xy\}$ is chosen} \\
   &\hspace{8cm}\text{with  $y'=y$ or $c(x'y')=c(xy)$}\}\\
V(\Gamma)&=V(G)\\
E(\Gamma)&=\{e\in E(G): \text{$c(e)$ is not chosen and $e$ is not killed}\}\\
H&=\Gamma[V(G)\setminus V(M)]
\end{align*}
We say that the $M, \Gamma$, and $H$ are \emph{produced} by the process.
Notice that by the definitions of $M$ and $H$ we always have that $M$ is a rainbow matching, that $M$  and $H$ partition $V(G)$, and that $M$  and $H$ share no colours.

\subsection{Probabilities}
To analyze various features  of $(\alpha, {b})$-random edge-assignments, we need estimates of the probability of various events. The following lemma computes all the probability estimates required.

\begin{lemma} \label{Lemma_Nibble_Probabilities}
Suppose that we have $d, {b},\alpha, \ell$ with  $(1+\gamma)d\geq {b}$ and $\ell d^{-1}\leq \alpha\leq \gamma\leq 0.01$.

Let $G$ be a  coloured balanced bipartite graph  which is $(\gamma, d/n, n)$-regular, globally ${b}$-bounded, and locally $\ell$-bounded.
Let $M, \Gamma, H$ be produced by an $(\alpha, {b})$-random edge-assignment on $G$. Then the following probability bounds  (\ref{Eq_Prob_Edge_Chosen}) -- (\ref{Eq_Prob_EdgeInGraph}) hold.
\end{lemma}
\begin{proof}
Let the bipartition classes of $G$ be $X$ and $Y$.
We will often use the following
\begin{equation}\label{Eq_dInverse}
\frac1{d(v)}= (1\pm 2\gamma)\frac{1}{d} \text{ \hspace{0.5cm} for any $v\in V(G)$}.
\end{equation}
This comes from the  $(\gamma, d/n, n)$-regularity of $G$ and (\ref{Eq_Denominator_Bounds}).

\begin{equation}\label{Eq_Prob_Edge_Chosen}
\P(\text{$xy$ chosen})=\frac{\alpha}{d}(1\pm 2\gamma) \text{ \hspace{0.5cm} for any $xy\in E(G)$}.
\end{equation}
This comes from $\P(\text{$xy$ chosen})=\frac{\alpha}{d(x)}$ and (\ref{Eq_dInverse}).

\begin{equation}\label{Eq_Prob_Edge_Killed}
\P(\text{$xy$ killed})=\left(\frac{\alpha {b}}{d}-\frac{\alpha |E({c(e)})|}{d}\right)(1\pm 2\gamma) \text{ \hspace{0.5cm} for any $xy\in E(G)$}.
\end{equation}
This comes from $\P(\text{$xy$ killed})=\frac{\alpha {b}}{d(x)}-\frac{\alpha |E({c(e)})|}{d(x)}$ and (\ref{Eq_dInverse}).

\begin{equation}\label{Eq_Prob_2_Chosen_Edges}
\P(\text{$e$ and $e'$ chosen})\leq \frac{\alpha^2}{d^2}(1+5\gamma) \hspace{0.5cm} \text{ for $e\neq e'\in E(G)$}.
\end{equation}
If $e\cap X=e'\cap X$, then both $e$ and $e'$ cannot be chosen, so we may assume that $e=xy$ and $e'=x'y'$ for $x\neq x'$. The events that $xy$ and $x'y'$ are chosen are independent which gives $\P(\text{edges $xy$ and $x'y'$ chosen})= \frac{\alpha^2}{d(x)d(x')}.$ Now (\ref{Eq_Prob_2_Chosen_Edges}) comes from (\ref{Eq_dInverse}) and $\and \gamma\leq 0.01$.

\begin{equation}\label{Eq_Prob_Colour}
\P(\text{$c$ chosen})= \frac{\alpha|E(c)|}{d}(1\pm 4\gamma) \text{ \hspace{0.5cm} for any $c\in C(G)$}.
\end{equation}
By the union bound and (\ref{Eq_Prob_Edge_Chosen}) we have that $c$ is chosen with probability $\leq \sum_{e\in E(c)}\P(e \text{ chosen})\leq |E(c)| (1+2\gamma)\frac{\alpha}{d}$.
By the Bonferroni inequalities (see Lemma~\ref{BIineq}),  (\ref{Eq_Prob_Edge_Chosen}) and (\ref{Eq_Prob_2_Chosen_Edges}) we have the bound
$\P(\text{colour $c$ chosen})$ $\geq$ $\sum_{e\in E(c)}\P(\text{$e$ chosen})-$ $\sum_{\substack{e, e'\in E(c)\\ e\neq e'}}\P(\text{$e$ and $e'$ chosen})\geq |E(c)|(1-2\gamma)\frac{\alpha}{d}- \binom{|E(c)|}2(1+5\gamma)\frac{\alpha^2}{d^2}$.
The lower bound in (\ref{Eq_Prob_Colour}) then comes from $|E(c)|\leq {b}\leq (1+\gamma)d$ and $\alpha\leq \gamma\leq 0.01$.

\begin{equation}\label{Eq_Prob_Edge_Chosen_But_Not_Deleted}
\P(\text{$xy$ chosen and $xy\not\in E(M)$})\leq \frac{3\alpha^2}{d} \text{ \hspace{0.5cm} for any $xy\in E(G)$}.
\end{equation}
From the definition of $M$, the only way  $xy\not\in E(M)$ can hold for a chosen edge $xy$ is if another edge $x'y'$ is chosen with either $y'=y$ or $c(x'y')=c(xy)$. By the union bound we have $\P(\text{$xy$ chosen and $xy\not\in E(M)$})\leq\sum_{x'\in N(y)\setminus\{x\}} \P(\text{$xy$ chosen and $x'y$ chosen})+ \sum_{x'y'\in C(xy)\setminus\{xy\}}\P(\text{$xy$ chosen and $x'y'$ chosen})$. Using (\ref{Eq_Prob_2_Chosen_Edges}), $|E(c)|\leq b\leq (1+\gamma)d$, $\Delta(G)\leq (1+\gamma)d$, and $\gamma\leq 0.01$, this is at most $(d(y)+|E(c)|)(1+5\gamma)\frac{\alpha^2}{d^2}\leq \frac{3\alpha^2}{d}$, as required.

\begin{equation}\label{Eq_Prob_Edge_J}
\P(\text{$e\not\in \Gamma$})= (1\pm 9\gamma)\frac{\alpha b}{d} \text{ \hspace{0.5cm} for any $e\in E(G)$}.
\end{equation}
Since $e$ is killed  independently of any colour being chosen, we have $\P(\text{$e\not\in \Gamma$})=\P(\text{$c(e)$ chosen})+\P(\text{$e$ killed})-\P(\text{$c(e)$ chosen})\P(\text{$e$ killed})$.
Combining this with (\ref{Eq_Prob_Colour}), (\ref{Eq_Prob_Edge_Killed}),  $|E({c(e)})|\leq {b}\leq (1+\gamma)d$,  and $\alpha\leq \gamma\leq 0.01$, we get $\P(\text{$e\not\in \Gamma$})=\frac{\alpha |E({c(e)})|}{d}(1\pm 4\gamma)+ \left(\frac{\alpha {b}}{d}-\frac{\alpha |E({c(e)})|}{d}\right)(1\pm 2\gamma)-  (1\pm 7\gamma)\frac{\alpha |E({c(e)})|}{d}\left(\frac{\alpha {b}}{d}-\frac{\alpha |E({c(e)})|}{d}\right)=  (1\pm 9\gamma)\frac{{b}\alpha}{d}$.

\begin{equation}\label{Eq_Prob_EdgeInMatching}
\P(xy\in E(M))=\frac{\alpha}{d}(1\pm 5\gamma) \text{ \hspace{0.5cm} for any $xy\in E(G)$}.
\end{equation}
Recall that $M$ contains only chosen edges. Using this, the upper bound comes from~(\ref{Eq_Prob_Edge_Chosen}), while the lower bound comes from~(\ref{Eq_Prob_Edge_Chosen}),~(\ref{Eq_Prob_Edge_Chosen_But_Not_Deleted}), and $\alpha\leq \gamma$.

\begin{equation}\label{Eq_Prob_Vertex}
 \P(v\in  V(M))= \alpha(1\pm 7\gamma) \text{ \hspace{0.5cm} for any $v\in V(G)$}.
\end{equation}
Recall that $M$ is a matching, which implies that  the events ``$\text{$vu\in E(M)$}$'' are disjoint for $u\in N(v)$. Using (\ref{Eq_Prob_EdgeInMatching}) and that $d(v)=(1\pm\gamma)d$, this gives $\P(v\in V(M))=\sum_{u\in N(v)}\P(\text{$vu\in E(M)$})=d(v)\cdot\frac{\alpha}{d}(1\pm 5\gamma)=\alpha(1\pm 7\gamma).$

\begin{equation}\label{Eq_Prob_2Vertices}
 \P(u, v\in  V(M))\leq  3\alpha^2 \text{\hspace{0.5cm} for any vertices $u\neq v\in V(G)$.}
\end{equation}
Notice that $\P(u,v\in V(M))\leq\sum_{\substack{z\in N(u),\\ w\in N(v)}}\P(\text{$uz, vw$ chosen})=\P(\text{$uv$ chosen})+\sum_{\substack{z\in N(u),\\ w\in N(v),\\ uz\neq vw}}\P(\text{$uz, vw$ chosen})$.
Here the first term is defined to be zero if there is no edge $uv$ in $G$.
   Using (\ref{Eq_Prob_Edge_Chosen}), (\ref{Eq_Prob_2_Chosen_Edges}), and $\Delta(G)\leq (1+\gamma)d$, we get that this is at most $(1+2\gamma)\frac{\alpha}{d}+ (1+\gamma)^2d^2\cdot (1+5\gamma)\frac{\alpha^2}{d^2}$ which, combined with $d^{-1}\leq \alpha\leq \gamma\leq 0.01$, implies the result.

\begin{equation}\label{Eq_Prob_2Vertices_OR}
\P({\{u,v\}\cap  V(M)\neq \emptyset})= 2\alpha(1\pm 11\gamma)  \text{\hspace{0.5cm} for any vertices $u\neq v\in V(G)$.}
\end{equation}
This comes from the Bonferroni inequalities together with (\ref{Eq_Prob_Vertex}), (\ref{Eq_Prob_2Vertices}), and $\alpha\leq \gamma$.

\begin{equation}\label{Eq_Prob_Edge1Chosen_And_vInM}
\P(xy \text{ chosen and } x'\in V(M))\leq \frac{\alpha^2}d (1+7\gamma)\hspace{0.5cm} \text{for $xy\in E(G)$ and $x'\not\in\{x,y\}$.}
\end{equation}
By the union bound, (\ref{Eq_Prob_2_Chosen_Edges}), and $\Delta(G)\leq (1+\gamma)d$, this probability is $\leq \sum_{y'\in N(x')}\P(\text{$xy$ and $x'y'$ chosen})\leq d(x')(1+5\gamma)\frac{\alpha^2}{d^2}\leq (1+7\gamma)\frac{\alpha^2}d$.

\begin{equation}\label{Eq_Prob_EdgeNotJ_and_VertexM}
\P(\text{$e\not\in \Gamma$ and $v\in V(M)$})\leq 6\alpha^2\hspace{0.5cm} \text{for $e\in E(G)$ and $v\in V(G)$.}
\end{equation}
By the union bound $\P(\text{$e\not\in \Gamma$ and $v\in V(M)$})\leq \P(\text{$e$ killed  and $v\in V(M)$})+ \P(\text{$c(e)$ chosen  and $v\in V(M)$})$.
Using (\ref{Eq_Prob_Edge_Killed}), (\ref{Eq_Prob_Vertex}), and ${b}\leq (1+\gamma)d$, the first term can be bounded above by $\P(\text{$e$ killed and }  v\in V(M))=\P(\text{$e$ killed})\P(v\in V(M))\leq \left(\frac{\alpha {b}}{d}-\frac{\alpha |E({c(e)})|}{d}\right)(1+ 2\gamma)\alpha(1+ 7\gamma)\leq  3\alpha^2$.
Let $E_{c(e),v}$ be the set of $\leq \ell$ colour $c(e)$ edges through $v$. The second term can be bounded by
$$ \P(\text{$c(e)$ chosen  and $v\in V(M)$})\leq \sum_{e'\in E_{c(e),v}} \P(\text{$e'$ chosen}) + \sum_{\substack{e'\not\in E_{c(e),v},\\ c(e')=c(e)}}\P(\text{$e'$ chosen and $v\in V(M)$}).$$
Using (\ref{Eq_Prob_Edge_Chosen}), (\ref{Eq_Prob_Edge1Chosen_And_vInM}), $|E_{c(e),v}|\leq \ell\leq \alpha d$, $d^{-1}\leq \alpha\leq \gamma\leq 0.01$, and ${b}\leq (1+\gamma)d$, this is at most  $(1+2\gamma)(\alpha d)\frac{\alpha}{d}+ |E(c)|\frac{\alpha^2}d (1+7\gamma)\leq 3\alpha^2$.

\begin{equation}\label{Eq_Prob_EdgeJ_and_VertexNotM}
\P(\text{$e\in \Gamma$ and $v\not\in V(M)$})= 1-\alpha-\frac{\alpha{b}}{d}\pm 22\alpha\gamma= (1\pm 23\alpha\gamma)\left(1-\alpha-\frac{\alpha{b}}{d}\right) \hspace{0.5cm}\text{for $e\in E(G)$ and $v\in V(G)$.}
\end{equation}
This comes from ``$\P(A\text{ and }B)=1-\P(\overline A)-\P(\overline B)+\P(\overline A\text{ and }\overline B)$'' together with (\ref{Eq_Prob_Edge_J}), (\ref{Eq_Prob_Vertex}), (\ref{Eq_Prob_EdgeNotJ_and_VertexM}), $\alpha\leq \gamma\leq 0.01$, and ${b}\leq (1+\gamma)d$.

\begin{equation}\label{Eq_Prob_EdgeInGraph}
\P(xy\not\in E(H))=\P(xy\not\in E(\Gamma)\text{ or } x\in V(M) \text{ or } y\in V(M))=\left(2\alpha+\frac{{b}\alpha}{d}\right)(1\pm 40\gamma)  \hspace{0.5cm}\text{for $xy\in E(G)$.}
\end{equation}
This comes from the Bonferroni inequalities together with (\ref{Eq_Prob_Edge_J}), (\ref{Eq_Prob_Vertex}), (\ref{Eq_Prob_2Vertices}),  (\ref{Eq_Prob_EdgeNotJ_and_VertexM}), and $\alpha\leq \gamma$.
\end{proof}

\subsection{Expectations}
Using the probabilities in the previous section, it is immediate to compute the expectations of relevant quantities.
\begin{lemma}\label{Lemma_expectations}
Suppose that we have $d, \ell, \alpha, \gamma$ with  $\ell d^{-1}\leq \alpha\leq \gamma \leq 0.01$.

Let $G$ be a  coloured bipartite graph  which is $(\gamma, d/n, n)$-regular, globally $(1+\gamma)d$-bounded, and locally $\ell$-bounded.
Let $M, \Gamma,$ and $H$ be the graphs produced by an $(\alpha, (1+\gamma)d)$-random edge-assignment on $G$.
The following hold:
\begin{itemize}
\item $\E(|X\cap V(H)|)=(1\pm (1+10\alpha)\gamma)(1-\alpha)n$
\item $\E(|Y\cap V(H)|)=(1\pm (1+10\alpha)\gamma)(1-\alpha)n$
\item $\E(|E_{H}(c)|)\leq \E(|E(c)\setminus V(M)|)\leq  (1+ 24\alpha \gamma)(1-2\alpha)|E_G(c)|$ for any colour $c$.
\item $\E(d_{H}(x))=\E(|N_{\Gamma}(x)\setminus V(M)|)= (1\pm (1+26\alpha)\gamma)\left(1-2\alpha\right)d$ for any vertex $x\in V(H)$.
\end{itemize}
\end{lemma}
\begin{proof}
These are immediate from linearity of expectation, (\ref{Eq_Prob_Vertex}), (\ref{Eq_Prob_2Vertices_OR}), (\ref{Eq_Prob_EdgeJ_and_VertexNotM}), and the $(\gamma, d/n, n)$-regularity of $G$.
\end{proof}

\subsection{Concentration}
By Azuma's Inequality, the random variables considered in the previous section are concentrated around their expectations.
\begin{lemma}\label{Lemma_nibble}
{Suppose that we have $n, \delta,\gamma, \alpha, \ell$ with $n^{-0.001} \leq  \alpha\leq \gamma\leq 0.00001, \delta\leq 1$ and $\ell\leq n^{0.001}$.}

Let $G$ be a  coloured bipartite graph with bipartition classes $X$ and $Y$ which is $(\gamma, \delta, n)$-regular, locally $\ell$-bounded, and globally $(1+\gamma) \delta n$-bounded.
Let $M, \Gamma,$ and $H$ be the graphs produced by an $(\alpha, (1+\gamma)\delta n)$-random edge-assignment on $G$.
The following hold with probability $\geq 1-n^{-2}$:
\begin{enumerate}[(i)]
\item $|X\cap V(H)|=|Y\cap V(H)|= (1\pm (1+12\alpha)\gamma)(1- \alpha)n$.
\item $|E_{H}(c)|\leq |E(c)\setminus V(M)|\leq (1+ 26\alpha \gamma)(1-2\alpha)\delta n$ for every colour $c$.
\item $d_{H}(v)=|N_{\Gamma}(v)\setminus V(M)|=(1\pm (1+30\alpha)\gamma)\left(1-2\alpha \right)\delta n $ for every vertex $v\in V(H)$.
\end{enumerate}
\end{lemma}
\begin{proof}
First we prove the Lipschitzness of the relevant random variables.
\begin{claim}
$|X\cap V(H)|$, $|Y\cap V(H)|$, $|E(c)\setminus V(M)|$, and $|N_{\Gamma}(v)\setminus V(M)|$ are all $26 \ell$-Lipschitz for any colour $c$ and vertex $v$.
\end{claim}
\begin{proof}
{Consider two $(\alpha, (1+\gamma)d)$-random edge-assignments which differ on one coordinate---Edge-Assignment 1 which produces graphs $M_1, \Gamma_1, H_1$ and Edge-Assignment 2 which produces graphs $M_2, \Gamma_2, H_2$. Furthermore, let $C_1$ and $C_2$ be the colours chosen respectively by the two edge-assignments, {and let $K_1$ and $K_2$ the edges killed respectively by the two edge-assignments.} We will show that $|V(M_1)\Symdiff V(M_2)|\leq 20$, {$|E(K_1)\Symdiff E(K_{2})|\leq 1$} and $|C_1 \Symdiff C_2|\leq 2$.

First, notice that $C_1$ and $C_2$ only differ in the colour of some edge $xy$ if $xy$ is chosen by one assignment and not the other.

Suppose that the coordinate on which the two edge-assignments differ is a balancing coin flip on an edge $xy$.
Notice that $M_1=M_{2}$, $C_1=C_2$ and $K_1$ and $K_{2}$ can differ only on the edge $xy$, so that, as required, $|V(M_1)\Symdiff V(M_2)|\leq 20$, $|K_1\Symdiff K_{2}|\leq 1$ and $|C_1 \Symdiff C_2|\leq 2$.

Suppose that the coordinate on which the two edge-assignments differ is a vertex-activation choice for a vertex $x\in X$, which is, say, activated in Edge-Assignment 1 but not Edge-Assignment 2. Say that $y$ is chosen by $x$ in Edge-Assignment 1. Either $M_1=M_2$, or $M_1=M_2+xy$, or $M_1$ is $M_2$ with up to two edges removed --- edges $x'y'$ with $y'=y$ or $c(xy)=c(x'y')$. Thus, we have $|V(M_1)\Delta V(M_2)|\leq 4$. As $C_1=C_2\cup\{c(xy)\}$ and $K_1=K_2$, we have $|E(K_1)\Symdiff E(K_{2})|\leq 1$ and $|C_1 \Symdiff C_2|\leq 2$.

Suppose finally that the coordinate on which the two edge-assignments differ is a vertex-choice for a vertex $x\in X$. Note that if $x$ is not activated then the outcome of the edge-assignments is the same and $C_1= C_2$, so we can assume that $x$ is activated.
 Let $y^1_x$ and $y_x^2$ be the vertices chosen by $x$ in Edge-Assignments 1 and 2 respectively.

Notice that $xy_x^1$ and $xy_x^2$ are the only edges which may be chosen by one, but not both assignments. Hence $c(xy_x^1)$ and $c(xy_x^2)$ are the only colours which may be chosen by one, but not both assignments, so that $|C_1\Symdiff C_2|\leq 2$.
The two rainbow matchings $M_1$ and $M_2$ can only differ on edges sharing a vertex or a colour with one of the edges $xy_x^1$ or $xy_x^2$.
Notice that $M_1$ has at most one edge touching each of the vertices $x, y_x^1$, and $y_x^2$ (since $M_1$ is a matching), and has at most one edge of each of the colours $c(xy_x^1)$ and $c(xy_x^2)$ (since $M_1$ is rainbow). Thus, $e(M_1\setminus M_2) \leq 5$. Similarly, $e(M_2\setminus M_1)\leq 5$. This implies that $|V(M_1)\Symdiff V(M_2)|\leq 20$. Furthermore, $K_1=K_2$, so certainly $|K_1\Symdiff K_{2}|\leq 1$

Thus, we always have that $|V(M_1)\Symdiff V(M_2)|\leq 20$, $|K_1\Symdiff K_{2}|\leq 1$ and $|C_1 \Symdiff C_2|\leq 2$.
By the definition of $H_1$ and $H_2$, we then have $V(H_1)\Symdiff V(H_2)= V(M_1)\Symdiff V(M_2)$ which implies $|(X\cap V(H_1))\Symdiff (X\cap V(H_2)|,
 |(Y\cap V(H_1))\Symdiff (Y\cap V(H_2)| \leq 20$.
 For a colour $c$, $E(c)\setminus V(M_{1})$ and $E(c)\setminus V(M_{2})$ can only differ on colour $c$ edges passing through $V(M_{1})\Symdiff V(M_{2})$. Combined with local $\ell$-boundedness, this gives  $|(E(c)\setminus V(M_{1})) \Symdiff (E(c)\setminus V(M_{2}))|\leq \ell|V(M_1)\Symdiff V(M_2)|\leq 20 \ell$.
For a vertex $v$, $N_{\Gamma_{1}}(v)\setminus V(M_{1})$ and $N_{\Gamma_{2}}(v)\setminus V(M_{2})$ can differ only on vertices of $V(M_{1})\Symdiff V(M_{2})$, {on vertices of $K_1\Symdiff K_{2}$}, or on vertices $z\in N(v)$ with $vz$ having colour in $C_1\Symdiff C_2$.
Combined with local $\ell$-boundedness, this gives $|(N_{\Gamma_{1}}(v)\setminus V(M_{1}))\Symdiff (N_{\Gamma_{2}}(v)\setminus V(M_{2}))|\leq 2\ell+22\leq 26\ell$.}
\end{proof}

{
Notice that $|X\cap V(H)|$,  $|Y\cap V(H)|$, and $|E(c)\setminus V(M)|$ are influenced only by the choices of the vertices $x\in X$ and which vertices in $X$ are activated, but not which edges are killed. Furthermore, $|N_{\Gamma}(v)\setminus V(M)|$ is influenced only by the choices of the vertices $x\in X$ and which vertices in $X$ are activated and which edges between $v$ and $N_G(v)$ are killed. There are at most $(1+\gamma)n$ vertices in $X$, and $d_G(v)\leq (1+\gamma)n$ neighbours of $v$. Overall we have that the quantities $|X\cap V(H)|$, $|Y\cap V(H)|$, $|E(c)\setminus V(M)|$, and $|N_{\Gamma}(v)\setminus V(M)|$ are each influenced by at most $3n$ coordinates.
}

Notice that $n^{-0.001}\leq  \alpha\leq \gamma\leq 0.00001, \delta\leq 1$ and $\ell\leq n^{0.001}$ implies that $\ell(\delta n)^{-1}\leq \alpha\leq \gamma\leq 0.01$.
Fix $t=\alpha\gamma\delta n/10$. By Lemma~\ref{Lemma_expectations}, we have
\begin{align*}
\E(|X\cap V(H)|)\pm  t&= (1\pm (1+12\alpha)\gamma)(1- \alpha)n
\\
\E(|Y\cap V(H)|)\pm  t&= (1\pm (1+12\alpha)\gamma)(1- \alpha)n   \\
\E(|E(c)\setminus V(M)|) + t&\leq (1+ (1+26\alpha)\gamma)(1-2\alpha)\delta n\\
\E(|N_{\Gamma}(v)\setminus V(M)|)\pm  t &= (1\pm (1+30\alpha)\gamma)\left(1-2\alpha\right)\delta n
\end{align*}

By Azuma's Inequality  we have that for any given $c, v$ any of (i) -- (iii) fail to hold with probability $\leq 2e^{-\frac{t^2}{3n(26\ell)^2}}\leq 2e^{-\frac{\alpha^2\gamma^2\delta^2n^{0.9}}{300000}}\leq 2e^{- n^{0.8}}$ (using $n^{-0.001}\leq \ell^{-1} \leq  \alpha\leq \gamma\leq 0.00001$).
Taking a union bound over all $c, v$ we have that all of (i) -- (iii) hold with probability $> 1 - 8n^2e^{-n^{0.8}}\geq 1-n^{-2}$ (using $n^{-0.001}\leq 0.001$).
\end{proof}

The following version of the above lemma will be more convenient to apply.
\begin{corollary}\label{Corollary_nibble}
Suppose that we have $n, \delta, \gamma, \alpha, \ell$ with  $n^{-1}\LL \alpha\leq \gamma\LL \delta\leq 1$ and $\ell\LL n$.
Let $G$ be a coloured balanced bipartite graph which is $(\gamma, \delta, n)$-regular, locally $\ell$-bounded, and globally $(1+\gamma)\delta n$-bounded.
Let $H$ be produced by an $(\alpha, (1+\gamma)\delta n)$-random edge-assignment on $G$.

With probability $\geq 1-n^{-2}$, the graph $H$ is $(e^{35\alpha}\gamma, e^{-\alpha}\delta, e^{-\alpha}n)$-regular and globally $(1+e^{35\alpha}\gamma)(e^{-\alpha}\delta)(e^{-\alpha} n)$-bounded.
\end{corollary}
\begin{proof}
Notice that  $n^{-1}\LL\alpha\leq \gamma\LL   \delta\leq 1$ and $\ell\LL n$ implies $n^{-0.001}\leq   \alpha\leq \gamma\leq 0.00001, \delta\leq 1$ and $\ell\leq n^{0.001}$. Let $X, Y$ be the bipartition classes of $G$.
By Lemma~\ref{Lemma_nibble}, we have that with probability $\geq 1- n^{-2}$ all of (i), (ii), and (iii) hold. Notice that, from (\ref{Eq_Exponential_Bound_Plus}) and $\alpha\leq \gamma\LL 1$, we have
\begin{equation}\label{Eq_Corollary_nibble}
(1\pm (1+30\alpha)\gamma)(1\pm 4\alpha^2)=(1\pm (1+30\alpha)\gamma \pm 5\alpha^2)=(1\pm e^{35 \alpha}\gamma).
\end{equation}

From (i), (\ref{Eq_Exponential_Bounds}), and (\ref{Eq_Corollary_nibble}), we have $|X\cap V(H)|=|Y\cap V(H)|= (1\pm (1+12\alpha)\gamma)(1- \alpha)n=(1\pm (1+12\alpha)\gamma)(1\pm \alpha^2)e^{-\alpha}n=(1\pm e^{35\alpha}\gamma)(e^{-\alpha})n$.

From (iii), (\ref{Eq_Exponential_Bounds}), (\ref{Eq_Corollary_nibble}),  we have that for all vertices $v\in V(G)$ we have $d_{H}(v)=(1\pm (1+30\alpha)\gamma)\left(1-2\alpha \right)\delta n=(1\pm (1+30\alpha)\gamma)(1\pm 4\alpha^2)e^{-2\alpha }\delta n=(1\pm e^{35\alpha}\gamma)(e^{-\alpha}\delta)( e^{-\alpha}n)$. These show that $H$ is $(e^{35\alpha}\gamma, e^{-\alpha}\delta, e^{-\alpha}n)$-regular.

From (ii), (\ref{Eq_Exponential_Bounds}),  and (\ref{Eq_Corollary_nibble}), we have that for every colour $c$ we  have $|E_{H}(c)|\leq (1+ (1+26\alpha)\gamma)(1-2\alpha)\delta n\leq (1+ (1+26\alpha)\gamma)(1+ 4\alpha^2)e^{-2\alpha}\delta n\leq (1+ e^{35\alpha}\gamma)e^{-2\alpha}\delta n$. This shows that $H$ is  globally $(1+e^{35\alpha}\gamma)(e^{-\alpha}\delta)(e^{-\alpha} n)$-bounded.
\end{proof}

\subsection{Finding a nearly-perfect matching}
Here we prove the main result of this section.
By iterating the $(\alpha,  b)$-random edge-assignment process on a properly coloured graph $G$ we can find a nearly spanning rainbow matching $M$ in $G$. The following lemma does this and shows that the resulting rainbow matching is random-like in a sense that every edge is in $M$ with at least (approximately) the right probability.
\begin{lemma}\label{Lemma_near_perfect_matching}
Suppose that we have $n, \delta, \gamma, p, \ell$ with $1\geq \delta \GG  p \GG \gamma \GG n^{-1}$ and $n\GG \ell$.

Let $G$ be a locally $\ell$-bounded, $(\gamma, \delta, n)$-regular, globally $(1+\gamma) \delta n$-bounded,  coloured, balanced bipartite graph. Then $G$ has  a random rainbow matching $M$ which has size $\geq (1-2p)n$ and
\begin{align}
\P(e\in E(M))&\geq(1-  9p)\frac{1}{\delta n}\hspace{0.5cm}\text{ for each $e\in E(G)$.} \label{Eq_Near_Matching_Edge_Probability_Lower_Bound}
\end{align}
\end{lemma}

\begin{proof}
Fix $H_0=G$, $\alpha=\gamma$, and $T=\alpha^{-1}\ln(p^{-1})$. Without loss of generality, we may suppose that $\gamma$ and $\alpha$ are chosen so that $T$ is an integer (to see this replace $\gamma$ by $\gamma'= \frac{\ln(p^{-1})}{\lfloor \gamma^{-1} \ln(p^{-1}) \rfloor}$. This ensures that $T'=\gamma'^{-1}\ln(p^{-1})$ is an integer. Notice that $p\GG 2\gamma\geq \gamma'\geq \gamma$ holds, so we could perform the  proof of the lemma with $\gamma$ replaced by $\gamma'$).  Notice that this gives $p=e^{-\alpha T}$.
Fix the following constants:
\begin{align*}
\gamma_t= e^{35\alpha t}\gamma\hspace{1cm}
\delta_t=e^{-\alpha t}\delta\hspace{1cm}
n_t=e^{-\alpha t}n.
\end{align*}
Using $p=e^{-\alpha T}$  we have $n_T= pn$, $\gamma_T= p^{-35}\gamma\leq p$, and $\delta_T=p \delta$.

We construct graphs $H_{1}, \dots, H_T$ and matchings $M_1, \dots, M_T$ recursively as follows.
\begin{itemize}
\item For $t\geq 0$, if $H_t$ is not  $(\gamma_t, \delta_t, n_t)$-regular or globally $(1+\gamma_t)\delta_t n_t$-bounded then stop the process at step $t$.
\item Otherwise, if $H_t$ is  $(\gamma_t, \delta_t, n_t)$-regular and globally $(1+\gamma_t)\delta_t n_t$-bounded, then we run an $(\alpha, (1+\gamma_t) \delta_t n_t)$-random edge-assignment on $H_{t}$ to get a graph $H_{t+1}$ and a matching $M_{t+1}$.
\end{itemize}

Notice that for all $t$,  $H_t$ is locally $\ell$-bounded and we have ${n_t^{-1}}\leq n_T^{-1}=p^{-1}n^{-1}\LL   \alpha= \gamma\leq \gamma_t\leq p^{-35}\gamma \LL p\delta\leq\delta_t\leq   1$\ and $\ell\LL n$.
Let $A_t$ be the event that the process has not stopped at any of the steps $1, \dots, t$. The events $A_t$ are clearly decreasing.
Since $\gamma_0=\gamma, \delta_0=\delta$, and $n_0=n$, the assumptions of the lemma imply that $\P(A_0)=1$.
 From Corollary~\ref{Corollary_nibble} we have $\P(A_t|  A_{t-1})\geq 1-n_{t-1}^{-2}$ (in this application we have $\gamma=\gamma_{t-1}$, $\delta=\delta_{t-1}$, $n=n_{t-1}$, $\alpha=\alpha$, $\ell=\ell$).
  This implies $\P(A_0\cap A_1\cap \dots \cap A_T)=\P(A_0)\P(A_1|A_0)\P(A_2|A_1)\dots \P(A_T|A_{T-1})\geq  \prod_{i=1}^T(1-n_{t-1}^{-2})\geq  (1-n_T^{-2})^T=(1- p^{-2}n^{-2})^T\geq 1-\frac1{\gamma^2 p^4 n^2}>0$ (using $p, \gamma\GG n^{-1}$).

Define $M$ to be the rainbow matching $M_1\cup \dots\cup M_T$ conditional on the events $A_0, \dots, A_T$ occuring (to see that $M$ is a rainbow matching, recall that $H_i$ and $M_i$ were vertex-disjoint and colour-disjoint). As $A_T$ holds, $H_T$ has $(1\pm \gamma_T)n_T$ vertices, so that M is a matching of size $\geq n-(1+\gamma_T)n_T\geq (1- 2p)n$.

\begin{claim}
The following hold {for each $t=1,\ldots,T$} and $e\in E(G)$.
\begin{align}
 \P(e\in E(H_t), A_t| e\in E(H_{t-1}), A_{t-1})&\geq (1- T^{-1}p)\left(1-3\alpha\right)\label{Eq_Prob_EdgeInGraph_Conditional}\\
 \P(e\in E(M_{t+1})| e\in E(H_t), A_t)&=(1\pm p)e^{2t\alpha} \frac{\alpha}{\delta n}
 \label{Eq_Prob_EdgeInMatching_Conditional}
\end{align}
\end{claim}
\begin{proof}
Using Corollary~\ref{Corollary_nibble}, notice that $\P(\overline{A_t}| e\in E(H_{t-1}), A_{t-1})\leq n_{t-1}^{-2}\leq n_T^{-2}\leq p^{-2}n^{-2}$ (this application of Corollary~\ref{Corollary_nibble} is the same as our previous one).
Using (\ref{Eq_Prob_EdgeInGraph}), (\ref{Eq_Prob_EdgeInMatching}),  and  $p^{-2} n^{-2}\leq {\gamma\alpha}{}$ (which comes from $1 \GG  p\GG \gamma \GG n^{-1}$) gives:
\begin{align*}
 \P(e\in E(H_t), A_t| e\in E(H_{t-1}), A_{t-1})&\geq \P(e\in E(H_t)| e\in E(H_{t-1}), A_{t-1})- \P(\overline{A_t}| e\in E(H_{t-1}), A_{t-1})\\
 &\geq \left(1-2\alpha-\frac{(1+\gamma_t)\delta_t}{\delta_t}\alpha \right) - 150\alpha\gamma_t - p^{-2}n^{-2}\\
 &\geq(1-160p^{-45}\alpha\gamma)\left(1-3\alpha\right)\\
 \P(e\in E(M_{t+1})| e\in E(H_t), A_t)&=(1\pm 5\gamma_t)\frac{\alpha}{\delta_t n_t}
 =(1\pm 6p^{-45}\gamma)e^{2t\alpha}\frac{\alpha}{\delta n}.
\end{align*}
Now the claim follows from $160p^{-45}\alpha\gamma\leq T^{-1}p$ and $6p^{-45}\gamma\leq p$ (which both come from $1 \GG  p\GG \gamma$).
\end{proof}

Let $t\leq T$. Notice that the events ``$e\in H_t$ and $A_{t}$ holds'' are decreasing with $t$.
Using (\ref{Eq_Exponential_Bounds}), (\ref{Eq_Power_Bound}), (\ref{Eq_Prob_EdgeInGraph_Conditional}), and $p\geq 10\alpha^2T$  we have
\begin{align*}
\P(e\in E(H_t), A_t)&=\prod_{i=1}^t \P(e\in H_i, A_i| e\in H_{i-1}, A_{i-1})
\geq\left(1- T^{-1}p\right)^t\left(1-3\alpha\right)^{t}
\geq (1- 3p)e^{-3t\alpha}.
\end{align*}

Combining the above with  (\ref{Eq_Exponential_Bound_Sum}), (\ref{Eq_Prob_EdgeInMatching_Conditional}), $p=e^{-\alpha T}$, and  $p\geq \alpha$  we get:
\begin{align*}
\P(e\in E(M_1\cup \dots \cup M_T))&=\sum_{t=1}^T \P(e\in H_{t-1}, A_{t-1})\P(e\in M_t|e\in H_{t-1}, A_{t-1})\\
&\geq (1- p)(1- 3p)\frac{\alpha}{\delta n}\sum_{t=1}^{T} e^{-\alpha {(t-1)}}\\
&\geq (1-4p)(1- \alpha^{2}- 2e^{-\alpha T})\frac{1}{\delta n}\\
&\geq (1- 8p)\frac{1}{\delta n}.
\end{align*}
Now (\ref{Eq_Near_Matching_Edge_Probability_Lower_Bound}) comes from $\P(e\in E(M_1\cup \dots \cup M_T)|A_0\cap \dots\cap A_T)\geq \P(e\in E(M_1\cup \dots \cup M_T))-\P(\overline{A_0\cap \dots\cap A_T})\geq \P(e\in E(M_1\cup \dots \cup M_T))-2{\gamma^2 p^4 n^{-2}}\geq (1- 9p)\frac{1}{\delta n}$.
\end{proof}

\section{Random and pseudorandom subgraphs}
In this section we collect intermediate lemmas which we will need concerning random and pseudorandom subgraphs. We will often use the typicality of complete graphs.
\begin{lemma}\label{Lemma_Typicality_K_n}
For $\gamma \GG n^{-1}$,
$K_n$ is $(\gamma,1,n)$-typical and $K_{n,n}$ is $(\gamma,1,n)$-typical.
\end{lemma}
\begin{proof}
Notice that $K_n$ is $(\frac2n,1,n)$-typical while $K_{n,n}$ is $(0,1,n)$-typical. Combined with $\gamma \GG n^{-1}$, this implies the lemma.
\end{proof}
\subsection{Random subgraphs}
We will need a number of results of the form ``for a nice graph $G$, a random subgraph $H$ of $G$ is still nice''.
Here ``nice'' can mean that $G$ is $(\gamma, \delta, n)$-regular, $(\gamma, \delta, n)$-typical, or globally bounded.
We will look at four different kinds of ``random subgraphs'' $H$.

\begin{lemma}[Random subgraphs of a general graph]\label{Lemma_Random_Subgraph_General}
Let $1\geq \gamma, \delta, p, \mu\GG n^{-1}$.
Let $G$ be a properly coloured,  globally $\mu n$-bounded $(\gamma, \delta, n)$-regular/$(\gamma, \delta, n)$-typical general graph.
\begin{enumerate}[(a)]
\item \textbf{Random set of colours:} Let $H_1$ be a subgraph of $G$ formed by choosing each colour with probability $p$. Then $H_1$ is $(2\gamma, p\delta, n)$-regular/$(2\gamma, p\delta, n)$-typical with probability $1-o(n^{-1})$.
\item \textbf{Random set of edges:} Let $H_2$ be a subgraph of $G$ formed by choosing each {edge} with probability $p$. Then $H_2$ is $(2\gamma, p\delta, n)$-regular/$(2\gamma, p\delta, n)$-typical and globally $(1+\gamma) p\mu n$-bounded with probability $1-o(n^{-1})$.
\item \textbf{Random set of vertices:}
For $pn\in \mathbb{Z}$ with $p<1$, let  $A\subseteq V(G)$ be a subset of order $pn$ chosen uniformly at random out of all  such subsets.
Then $G[A]$ is globally $(1+\gamma)\left({\mu p^2}\right) n$-bounded and $(2\gamma, \delta, pn)$-regular/$(2\gamma, \delta, pn)$-typical  with probability $1-o(n^{-1})$.
\item \textbf{Two disjoint random sets of vertices:}
For $pn\in \mathbb{Z}$ with $p<1/2$, let $A,B\subseteq V(G)$ be two disjoint subsets of order $pn$ chosen uniformly at random out of all pairs of such subsets.
Then $G[A,B]$ is a globally $(1+\gamma)\left({2\mu p^2}\right) n$-bounded, $(2\gamma, \delta, pn)$-regular/$(2\gamma, \delta, pn)$-typical balanced bipartite graph with probability $1-o(n^{-1})$.
\end{enumerate}
\end{lemma}
\begin{proof}
Notice that the following bounds on expectations are true by linearity of expectation for all vertices $u\neq v$ and colours $c$.
\begin{align*}
\E(d_{H_1}(v))= \E(d_{H_2}(v))&=pd_{G}(v)=(1\pm \gamma)p\delta n &  &\mkern-85mu\text{when $G$ is $(\gamma, \delta, n)$-regular.}\\
\E(d_{H_1}(u,v))= \E(d_{H_2}(u,v))&=p^2d_{G}(u,v)=(1\pm \gamma)p^2\delta^2 n & &\mkern-85mu\text{when $G$ is $(\gamma, \delta, n)$-typical.}\\
\E(|N(v)\cap A|)= \E(|N(v)\cap B|)&=p d_{G}(v)=(1\pm \gamma)\delta (pn)  & &\mkern-85mu\text{when $G$ is $(\gamma, \delta, n)$-regular.}\\
\E(|N(u,v)\cap A|)= \E(|N(u,v)\cap B|)&=p d_{G}(u,v)=(1\pm \gamma)\delta^2 (pn) & &\mkern-85mu\text{when $G$ is $(\gamma, \delta, n)$-typical.}\\
\E(|E_{H_2}(c)|)&= p|E_{G}(c)|\leq p\mu n & &\\
\E(|E_{G}(c)\cap A|)= \E(|E_{G}(c)\cap B|)&= \frac{pn(pn-1)}{n(n-1)}|E_{G}(c)|={ \Big(p^2-\frac{p(1-p)}{n-1}\Big)|E_{G}(c)|}\leq p^2\mu n & &\\
\E(|E_{G}(c)\cap (A\cup B)|)&= \frac{2pn(2pn-1)}{n(n-1)}|E_{G}(c)|\leq 4p^2|E_{G}(c)|-\frac{2p(1-2p)}{n-1}\leq 4p^2\mu n & &
\end{align*}
First we prove (a) and (b).
Notice that the random variables $d_{H_1}(v)$, $d_{H_2}(v)$, $d_{H_1}(u,v)$, $d_{H_2}(u,v)$ and $|E_{H_2}(c)|$ are all 2-Lipshitz (using the fact that the colouring is proper), and are all influenced by $\leq 2n$ coordinates. By Azuma's Inequality (Lemma~\ref{Lemma_Azuma}), we have that the probability that any of these deviate from their expectation by more than $\gamma p^2\delta^2\mu n$ is $\leq 2e^{\frac{-(\gamma p^2\delta^2\mu n)^2}{8n}}= o(n^{-3})$ (using $1\geq \gamma, \delta, p, \mu\GG n^{-1}$ which implies $\gamma^2p^4\delta^4\mu^2 n\geq 40 \log n$). Taking a union bound over all pairs of vertices and colours, we obtain (a) and (b).

It remains to prove (c) and (d).
Notice that the functions $|N(v)\cap A|$, $|N(v)\cap B|$, $|N(u,v)\cap A|$, $|N(u,v)\cap B|$, $|E_{G}(c)\cap A|$ and  $|E_{G}(c)\cap B|$ each satisfy the assumptions
of Lemma~\ref{Lemma_Kwan} with $\alpha=1$, $r=pn$, $N=n$. Also  $|E_{G}(c)\cap (A\cup B)|$ satisfies the assumptions of  Lemma~\ref{Lemma_Kwan} with $\alpha=1$, $r=2pn$, $N=n$. Finally, notice that we have $0< \min(r, N-r)< n$ for all of these.
By Lemma~\ref{Lemma_Kwan} we have that the probability that any of these functions deviate from their expectation by more than $\gamma p^2\delta^2\mu n/4$ is $\leq
2e^{-\frac{2(\gamma p^2\delta^2\mu n/4)^2}{n}}=o(n^{-3})$ (using $1\geq \gamma, \delta, p, \mu\GG n^{-1}$ which implies $\gamma^2p^4\delta^4\mu^2 n\geq 40 \log n$).
Taking a union bound over all pairs of vertices and all colours, we obtain (c) and the ``$(2\gamma, \delta, pn)$-regular/$(2\gamma, \delta, pn)$-typical'' part of (d).
We also get that with probability $1-o(n^{-1})$ we have $|E_{G}(c)\cap A|, |E_{G}(c)\cap B|= p^2|E_{G}(c)|-\frac{p(1-p)}{n-1}{|E_{G}(c)|}\pm \gamma p^2\delta^2\mu n/4= p^2|E_{G}(c)|\pm \gamma\mu n/3 $ for all colours $c$ (using $\gamma, p,\mu \GG n^{-1}$). Similarly we have $|E_{G}(c)\cap (A\cup B)|= 4p^2|E_{G}(c)|\pm \gamma \mu n/3$.
These give $e(G[A,B]\cap E_G(c))= |E_G(c)\cap (A\cup B)| - |E_G(c)\cap A|-|E_G(c)\cap B|= 4p^2E_{G}(c)- 2\cdot p^2E_{G}(c)\pm  \gamma \mu n\leq (1+\gamma)2p^2 \mu n$ (the last inequality coming from global $\mu n$-boundedness). This implies the global boundedness part of  (d).
\end{proof}

We will need a balanced bipartite version of part of the above lemma.
\begin{lemma}[Random subgraphs of balanced bipartite graph]\label{Lemma_Random_Subgraph_Bipartite}
Let $1\geq \gamma, \delta, p, \mu\GG n^{-1}$.
Let $G$ be a properly coloured,  globally $\mu n$-bounded $(\gamma, \delta, n)$-regular/$(\gamma, \delta, n)$-typical balanced bipartite graph.
\begin{enumerate}[(a)]
\item \textbf{Random set of colours:} Let $H_1$ be a subgraph of $G$ formed by choosing each colour with probability $p$. Then $H_1$ is $(2\gamma, p\delta, n)$-regular/$(2\gamma, p\delta, n)$-typical with probability $1-o(n^{-1})$.
\item \textbf{Random set of edges:} Let $H_2$ be a subgraph of $G$ formed by choosing each colour with probability $p$. Then $H_2$ is $(2\gamma, p\delta, n)$-regular/$(2\gamma, p\delta, n)$-typical and globally $(1+\gamma) p\mu n$-bounded with probability $1-o(n^{-1})$.
\end{enumerate}
\end{lemma}
\begin{proof}
Let $u,v$ be vertices, and $c$ a colour.
Notice that the following bounds on expectations are true by linearity of expectation.
\begin{align*}
\E(d_{H_1}(v))= \E(d_{H_2}(v))&=pd_{G}(v)=(1\pm \gamma)p\delta n &\text{when $G$ is $(\gamma, \delta, n)$-regular}\\
\E(d_{H_1}(u,v))= \E(d_{H_2}(u,v))&=p^2d_{G}(u,v)=(1\pm \gamma)p^2\delta^2 n &\text{when $G$ is $(\gamma, \delta, n)$-typical}\\
\E(|E_{H_2}(c)|)&= p|E_{G}(c)|\leq p\mu n
\end{align*}
Notice that the random variables $d_{H_1}(v)$, $d_{H_2}(v)$, $d_{H_1}(u,v)$, $d_{H_2}(u,v)$ and $|E_{H_2}(c)|$ and are all 2-Lipshitz (using the fact that the colouring is proper), and are all influenced by $\leq 2n$ coordinates. By Azuma's Inequality (Lemma~\ref{Lemma_Azuma}), we have that the probability that any of these deviate from their expectation by more than $\gamma p^2\delta^2\mu n$ is $\leq 2e^{\frac{-(\gamma p^2\delta^2\mu n)^2}{8n}}= o(n^{-3})$ (using $1\geq \gamma, \delta, p, \mu\GG n^{-1}$ which implies $\gamma^2p^4\delta^4\mu^2 n\geq 40 \log n$). Taking a union bound over all pairs of vertices and colours, we obtain (a) and (b).
\end{proof}

The following lemma gives another property of the random subgraph formed by choosing every edge independently with probability $p$. This time we are concerned with how many vertices a  small set of colours covers.
\begin{lemma}\label{Lemma_Erdos_Renyi_Subgraph}
Let $1\geq p, \varepsilon \GG k^{-1} \GG \nu \GG n^{-1}$.
Let $G$ be a properly coloured  graph with all colours covering $\geq (1-\nu)n$ vertices. Let $H$ be a random subgraph formed by choosing every edge with probability $p$. Then, with high probability, any set of $k$ colours of $H$ covers  $\geq (1-\varepsilon)n$ vertices.
\end{lemma}
\begin{proof}
Let $S$ be a set of $k$ colours and $G_S$, $H_S$ the subgraphs of $G$ and $H$ consisting of colour $S$ edges. Notice that $e(G_S)\geq k(1-\nu)n/2$. By the Handshaking Lemma, we have $\sum_{v\in V(G_S)} d_{G_S}(v)=2e(G_S)\geq k(1-\nu)n$.
Let $L$ be the set of vertices in $G_S$ of degree $\geq k/2$.
Using $\Delta(G_S)\leq k$, we have $|L|k+(n-|L|)k/2\geq \sum_{v\in V(G_S)} d_{G_S}(v)\geq k(1-\nu)n$, which is equivalent to $|L|\geq (1-2\nu)n$.

For a vertex $v\in L$ we have $\P(d_{H_S}(v)=0)=(1-p)^{d_{G_S}(v)}\leq (1-p)^{k/2}\leq e^{-pk/2}\leq \varepsilon/4$, as $p,\varepsilon \GG k^{-1}$.
Let $X$ be the number of isolated vertices in $L$.
By linearity of expectation $\E(X)\leq \varepsilon n/4$.
Notice that $X$ is $2$-Lipschitz and is influenced by $\leq e(G_S)\leq kn/2$ edges.
By Azuma's Inequality (Lemma~\ref{Lemma_Azuma}) applied with $t=\varepsilon n/4$ we have $\P(X\geq \varepsilon n/2)\leq 2e^{-\frac{(\varepsilon n/4)^2}{4(kn/2)}}= 2e^{-\frac{\varepsilon^2 n}{32k}}\leq n^{-3k}$ (as $\varepsilon,k^{-1}\GG n^{-1}$ implies that $\varepsilon^2/k^2\geq \frac{400\log n}{n}$). Thus, with probability $\geq 1-n^{-3k}$, $H_S$ has $\leq X+(n-|L|)\leq \varepsilon n$ isolated vertices.
Taking a union bound over all sets $S$ of $k$ colours gives the result.
\end{proof}

\subsection{$(e,m)$-Dense graphs}
In this paper it is convenient to use two different notions of pseudorandomness. The first of these is $(\gamma, \delta, n)$-typicality (See Definition~\ref{Definition_Typical}). The second is the following.
\begin{definition}
\
\begin{itemize}
\item A general graph $G$ is $(e,m)$-dense if for {any} $\lambda\geq 1$ and disjoint sets $A$, $B$ with $|A|=|B|=\lambda m$, we have $e(A,B)\geq \lambda^2 e$.
\item A balanced bipartite graph $G$ with parts $X$ and $Y$ is $(e,m)$-dense if for {any}  $\lambda\geq 1$ and sets $A\subseteq X$, $B\subseteq Y$ with $|A|=|B|=\lambda m$, we have $e(A,B)\geq \lambda^2 e$.
\end{itemize}
\end{definition}
We remark that most of the time we will use the above definition with $\lambda=1$. Thus the definition should be thought of saying that there are $e$ edges between any two sets of vertices of size $m$. Notice that if $G$ is $(e,m)$-dense, then it is also $(e',m')$-dense for any $e'\leq e$ and $m'\geq m$.

How is the above definition related to $(\gamma, \delta, n)$-typicality? In fact, $(\gamma, \delta, n)$-typicality is a stronger concept.
We prove two lemmas relating typicality and density. The following is a variation of a lemma proved by the third author together with Alon and Krivelevich in \cite{alon1999list}.
\begin{lemma}\label{Lemma_Codegrees_Discrepency}
Every  $(\gamma, p,n)$-typical graph $H$ has the following for every pair of subsets $A$, $B$ with $|B|\geq \gamma^{-1}p^{-2}$:
$$|e(A,B)- p|A||B||\leq 2|A|^{\frac12}|B|\gamma^{\frac12}n^{\frac12}p.$$
\end{lemma}
\begin{proof}
Let $\mathrm{Adj}_H$ be the adjacency matrix of $H$, and let $M=\mathrm{Adj}_H-pJ$ where $J$ is the appropriately-sized all-ones matrix.
Notice that for every pair of distinct vertices $x,x'$, we have
\begin{equation}
\sum_{v\in V(H)}M_{x,v}M_{x',v}= d_H(x,x')-p(d(x)+d(x'))+p^2n\leq (1+\gamma)p^2 n- 2(1-\gamma)p^2 n+p^2n\leq 3\gamma p^2 n.\label{exponew}
\end{equation}
Next notice that we have
\begin{align*}
|e(A,B)- p|A||B||^2 &=\left(\sum_{x\in A}\sum_{y\in B} M_{x,y}\right)^2
 \leq |A|\sum_{x\in A}\left( \sum_{y\in B} M_{x,y}\right)^2
 \leq |A|\sum_{x\in V(H)}\left( \sum_{y\in B} M_{x,y}\right)^2\\
 &= |A|\sum_{x\in V(H)}\left( \sum_{y\in B} M_{x,y}^2\right) +|A|\sum_{x\in V(H)}\left(\sum_{y\neq y'\in B} M_{x,y}M_{x,y'}\right)\\
 &\leq n|A||B|+ |A| \sum_{y\neq y'\in B}\left(\sum_{x\in V(H)} M_{x,y}M_{x,y'}\right)\\
 &\overset{\eqref{exponew}}{\leq} n|A||B|+ |A|\sum_{y\neq y'\in B} 3\gamma n p^2
 \leq n|A||B|+ |A||B|^2 3\gamma n p^2
 \leq 4 |A||B|^2 \gamma n  p^2
\end{align*}
Here the first inequality comes from the Cauchy-Schwarz inequality
and the last inequality comes from $|B|\geq \gamma^{-1}p^{-2}$.
Taking square roots gives the result.
\end{proof}

The following version of the above is more convenient to apply.
\begin{lemma}\label{Lemma_Codegrees_Pseudorandom}
Let $1\geq p, \mu \GG \gamma\GG n^{-1}$.
Every  $(\gamma, p, n)$-typical graph {(which is either balanced bipartite or general)} is $(0.99p(\mu n)^2, \mu n)$-dense.
\end{lemma}
\begin{proof}
First we deal with the case when $G$ is a general graph.
Notice that $p, \mu \GG \gamma\GG n^{-1}$ implies $\mu n\geq \gamma^{-1} p^{-2}$.
By Lemma~\ref{Lemma_Codegrees_Discrepency} we have  that for any $\lambda\geq 1$ and  pair of subsets $A$, $B$ with $|A|=|B|=\lambda \mu n$ we have
$|e(A,B)- p(\lambda \mu n)^2|\leq 2(\lambda \mu n)^{\frac12}(\lambda \mu n)\gamma^{\frac12} n^{\frac12}p \leq 0.01p (\lambda \mu n)^2$ (the  last inequality is $2\gamma^{\frac12}\leq 0.01\lambda^{\frac12}\mu^{\frac12}$ which comes from $\gamma\LL \mu$ and $\lambda\geq 1$).

Now suppose that $G$ is a balanced bipartite $(\gamma, p, n)$-typical graph with parts $X, Y$. Add a copy of the Erd\H{o}s-R\'enyi random graph $G(n,p)$ to both $X$ and $Y$ to get a graph $H$. Notice that for any vertex $v$ we have $\E(d_H(v))=(1\pm \gamma)2pn$ and that any pair of vertices $u,v$ have $\E(d_H(u,v)) = (1\pm \gamma)2p^2n$. Notice that these quantities are each $1$-Lipschitz affected by $\leq 2n$ coordinates. By Azuma's Inequality (Lemma~\ref{Lemma_Azuma}) and the union bound we get that with high probability $H$ is a $(2\gamma, p, 2n)$-typical general graph.
By the general graph version of this lemma with $\mu'=\mu/2$, $H$ is  $(0.99p(\mu n)^2, \mu n)$-dense. This implies that between any sets $A\subseteq X$, $B\subseteq Y$ with $|A|=|B|=\lambda \mu n$ we have $e_G(A,B)=e_H(A,B)\geq \lambda^2 0.99p(\mu n)^2$, i.e.\ that $G$ is $(0.99p(\mu n)^2, \mu n)$-dense (as a balanced bipartite graph).
\end{proof}
The following lemma shows that it is possible to delete a small number of edges from any graph so that its complement is pseudorandom.  Here $\overline H$ denotes the set of edges on $V(H)$ not present in $H$.
\begin{lemma}\label{Lemma_Bipartite_Complement_expander}
Let $1\geq p, \mu \GG n^{-1}$.
Every  $d$-regular balanced bipartite graph $G$ on $2n$ vertices has a $(d-\lfloor p n\rfloor)$-regular spanning subgraph $H$ such that $\overline H$  is $(0.48p (\mu n)^2, \mu n)$-dense.
\end{lemma}
\begin{proof}
Choose $\gamma$ with $1\geq p, \mu \GG \gamma \GG n^{-1}$.
Consider an arbitrary $1$-factorization of $K_{n,n}$ in which every colour either only occurs on $G$ or only occurs outside $G$ (this exists  since every regular bipartite graph has a $1$-factorization).
By Lemma~\ref{Lemma_Typicality_K_n}, $K_{n,n}$ is $(\gamma, 1, n)$-typical.
Let $E$ be a subgraph of $K_{n,n}$  formed by choosing every colour with probability $0.5p$.
By Lemma~\ref{Lemma_Random_Subgraph_Bipartite} (a),   $E$ is $(2\gamma, 0.5 p, n)$-typical with high probability.
By Lemma~\ref{Lemma_Codegrees_Pseudorandom}, applied with $\gamma'=2\gamma$, $p'=0.5p$, $\mu=\mu$, $E$ is  $(0.48p (\mu n)^2, \mu n)$-dense.

Since $E$ and $G\setminus E$   are unions of perfect matchings, they are regular.
Since $E$ is $(2\gamma, 0.5p, n)$-typical, the graph $G\setminus E$ is $d'$-regular for some $d'\geq d-(1+2\gamma)0.5 pn\geq d-\lfloor pn\rfloor$ (using $\gamma\LL   1$). Therefore we can find some $(d-\lfloor pn\rfloor)$-regular subgraph $H$ of $G$ which is edge-disjoint from $E$ (using that $G$ is bipartite).  Since $H^c$ contains $E$, we have that $\overline H$ is $(0.48p(\mu n)^2, \mu n)$-dense as required.
\end{proof}

We will need two lemmas showing that deleting a small number of edges from an $(e, m)$-dense graph does not change the pseudorandomness too much.
\begin{lemma}\label{Lemma_Expander_Delete_Edges}
Let $G$ be  $(e, m)$-dense and $H$  a subgraph of $G$. Then $G\setminus H$ is  $(e-e(H),m)$-dense.
\end{lemma}
\begin{proof}
For any $\lambda\geq 1$ and sets $A\subseteq X$, $B\subseteq Y$ with $|A|=|B|=\lambda m$ we have $e_{G\setminus H}(A,B)\geq e_G(A,B)-e(H)\geq \lambda^2 e-e(H)\geq \lambda^2(e-e(H))$.
\end{proof}
\begin{lemma}\label{Lemma_Expander_Delete_Matching}
Let $G$ be  $(e, m)$-dense and $M$ a matching in $G$. Then $G\setminus E(M)$ is  $(e-m,m)$-dense.
\end{lemma}
\begin{proof}
For $\lambda>1$, let $A$ and $B$ be sets with $|A|=|B|=\lambda m$.
Since $G$ is $(e, m)$-dense, we have $e_G(A,B)\geq \lambda^2 e$.
Since $M$ is a matching there can be at most $\lambda m$ edges of $M$ between $A$ and $B$.
Therefore $e_{G\setminus M}(A,B)\geq \lambda^2 e-\lambda m\geq \lambda^2(e-m)$.
\end{proof}

\section{Regularization lemmas}
In our proofs we will need a number of intermediate lemmas saying that a graph $G$ can be modified into a regular graph. Broadly speaking there are three types of modifications that we will need: deleting a small number of edges, adding edges from a disjoint dense graph,  or  adding a small number of vertices.
\subsection{Regularization by deleting edges}
Here we will prove results about finding a regular subgraph by deleting edges from a graph with very high minimum degree. The goal of this section is to prove Lemmas~\ref{Lemma_regular_subgraph_few_large_colours} and~\ref{Lemma_regular_subgraph_few_large_colours_bipartite}.
The following theorem of Cristofides, K\"uhn, and Osthus is a result of the type we want in this section (see Theorem~12 in \cite{christofides2012edge}).
\begin{theorem}[Cristofides, K\"uhn, Osthus]\label{Theorem_Regular_Subgraph_CKO}
Let $G$ be a graph with minimum degree $\delta\geq n/2$ and $r$ an even number with $r\leq \frac12(\delta+\sqrt{n(2\delta-n)})$. Then $G$ has a spanning $r$-regular subgraph.
\end{theorem}
The following version of this will be a bit easier to apply.
\begin{lemma} \label{Lemma_Regular_subgraph}
Let $n^{-1}\LL \varepsilon\LL 1$. Let $G$ be an $n$-vertex  graph  with $\delta(G)\geq(1-\varepsilon)n$. Then, $G$ has a spanning $2\lceil(1-\varepsilon-8\varepsilon^2)n/2\rceil$-regular subgraph.
\end{lemma}
\begin{proof}
Set $\delta=(1-\varepsilon)n\geq n/2$ and $r=2\lceil(1-\varepsilon-8\varepsilon^2)n/2\rceil$. Notice that $r$ is even and has $r\leq \frac12(\delta+\sqrt{n(2\delta-n)})$ (using $\sqrt{n(2\delta-n)}=n\sqrt{1-2\varepsilon}\geq n(1-\varepsilon-2\varepsilon^2)$, which holds for $\varepsilon\leq 1/2$).
 Apply Theorem~\ref{Theorem_Regular_Subgraph_CKO} to get the lemma.
\end{proof}
We will also need a balanced bipartite version of this lemma. To prove it we use the following theorem of Ore and Ryser (see \cite{ore1962theory}).
\begin{theorem} [Ore, Ryser]\label{Theorem_Ore_Ryser}
A balanced bipartite graph with parts $X, Y$ has no spanning $d$-regular subgraph if and only if there is a set $T\subseteq Y$ with
$d|T|>\sum_{x\in X}\min(|N(x)\cap T|, d)$.
\end{theorem}
Using this we can prove a bipartite version of Lemma~\ref{Lemma_Regular_subgraph}.
\begin{lemma} \label{Lemma_Regular_subgraph_bipartite}
Let $n^{-1}\LL\varepsilon\LL 1$. Let $G$ be a balanced bipartite graph with vertex classes $X$ and $Y$ with $|X|=|Y|=n$ and $\delta(G)\geq(1-\varepsilon)n$. Then $G$ has a spanning $\lfloor(1-\varepsilon-8\varepsilon^2)n\rfloor$-regular subgraph.
\end{lemma}
\begin{proof}
Notice that there is an $\hat{\varepsilon}\in \mathbb{R}$ with $(1-\hat{\varepsilon}-8\hat{\varepsilon}^2)n=\lfloor(1-\varepsilon-8\varepsilon^2)n\rfloor$ and $1\GG \varepsilon+n^{-1}\geq \hat{\varepsilon}\geq \varepsilon\GG n^{-1}$.
Notice that $\delta(G)\geq (1-\hat{\varepsilon})n$.
Fix $d=(1-\hat{\varepsilon}-8\hat{\varepsilon}^2)n=\lfloor(1-\varepsilon-8\varepsilon^2)n\rfloor$.

If the lemma does not hold then by Theorem~\ref{Theorem_Ore_Ryser}, there is a set $T\subseteq Y$ with $d|T|>\sum_{x\in X}\min(|N(x)\cap T|, d)$. Fix $|T|=(1-\tau)n$ and $\alpha n=|\{x\in X: \min(|N(x)\cap T|, d)=d\}|$.
Notice that $\delta(G)\geq (1-\hat{\varepsilon} )n$ implies $|N(x)\cap T|\geq |T|-\hat{\varepsilon} n$ for all $x\in X$.
Using this $d|T|>\sum_{x\in X}\min(|N(x)\cap T|, d)$ implies
$d(1-\tau)n>(1-\tau-\hat{\varepsilon} )(1-\alpha)n^2+ d\alpha n$. {Plugging in the value of $d$ gives
$(1-\hat{\varepsilon}-8\hat{\varepsilon}^2)(1-\tau)>(1-\tau-\hat{\varepsilon})(1-\alpha)+ (1-\hat{\varepsilon}-8\hat{\varepsilon}^2)\alpha$, and therefore $(\hat{\varepsilon}  +8\hat{\varepsilon}^2-\alpha)\tau>8\hat{\varepsilon}^2(1-\alpha)$.
As $\alpha\leq 1$ and $\tau\geq 0$, we must have that $\alpha<\hat{\varepsilon}+8\hat{\varepsilon}^2$.
Therefore, $\tau>\frac{8\hat{\varepsilon}^2(1-\alpha)}{\hat{\varepsilon}+8\hat{\varepsilon}^2-\alpha }$ and $\alpha<\hat{\varepsilon}+8\hat{\varepsilon}^2<1.1\hat{\varepsilon}\leq 0.002$ give $\tau >4\hat{\varepsilon}$.}

Notice that $|N(x)\cap T|\leq (1-\tau)n\leq (1-4\hat{\varepsilon})n\leq d$ which implies that $(1-\hat{\varepsilon}-8\hat{\varepsilon}^2)(1-\tau)n^2=d|T|>\sum_{x\in X}\min(|N(x)\cap T|, d)=\sum_{x\in X}|N(x)\cap T|= e(X,T)\geq |T|\delta(G)\geq (1-\tau)(1-\hat{\varepsilon})n^2$, a contradiction to ``$8\hat{\varepsilon}^2> 0$''.
\end{proof}
We want versions of the above lemmas for coloured graphs. Our lemmas will furthermore provide regular subgraphs with a better global boundedness than the starting graph. To do this, we use the following lemma which shows that any properly coloured graph has a subgraph with better global boundedness. It is proved by selecting the subgraph randomly.
\begin{lemma}\label{Lemma_subgraph_few_large_colours}
Let $1\GG \varepsilon\GG n^{-1}$ and $k\in \{1,2\}$.
Let  ${G}$ be a properly coloured, globally $n/k$-bounded graph on $\leq 2n$ vertices with $\leq (1-20\varepsilon)n$ colours having $\geq (1-20\varepsilon)n/k$ edges and $\delta({G})\geq (1-\varepsilon^2)n$. Then ${G}$ has a spanning subgraph ${H}$ with $\delta({H})\geq (1-\varepsilon+18\varepsilon^2)n$ which is globally $(1-\varepsilon)n/k$-bounded.
\end{lemma}
\begin{proof}
We say that a colour is \emph{large} if it has $\geq (1-20\varepsilon)n/k$ edges in ${G}$. Other colours are called \emph{small}.
For a vertex $v$, let $\ell_{H}(v)$ and $\ell_{G}(v)$ be the numbers of large colours through $v$ in ${H}$ and ${G}$ respectively. Similarly let $s_{G}(v)$ be the number of small colours through $v$.
Notice that $\ell_{G}(v)\leq (1-20\varepsilon)n$ and $s_{G}(v)= d_{G}(v)-\ell_{G}(v)\geq 20\varepsilon n- \varepsilon^2 n$ always hold.
   Let ${H}$ be the subgraph of ${G}$ formed by deleting every edge having a large colour independently  with probability $p=\varepsilon+\varepsilon^2$.
   The following hold for all  vertices $v$ and large colours $c$ by linearity of expectation.
   \begin{align*}
   \E(|E_H(c)|)&=(1-p)|E_G(c)|\leq (1-\varepsilon-\varepsilon^2)n/k,\\
   \E(d_H(v))&=(1-p)\ell_G(v)+s_G(v)=(1-p)d_G(v)+ps_G(v)\\
   &\geq (1-\varepsilon-\varepsilon^2)(1-\varepsilon^2)n+(\varepsilon+\varepsilon^2)(20\varepsilon- \varepsilon^2 )n= (1-\varepsilon+18\varepsilon^2+20\varepsilon^3)n .
   \end{align*}
Notice that $|E_H(c)|$ and $d_H(v)$ are both $1$-Lipschitz and affected by $\leq n$ edges. By Azuma's Inequality (Lemma~\ref{Lemma_Azuma}), the probability that either of these deviates from its expectation by more than $\varepsilon^3 n/4$ is $\leq 2e^{\frac{-(\varepsilon^3 n/4)^2}{n}}= o(n^{-2})$ (using $n\GG\varepsilon^{-1}$). Taking a union bound, we have that with high probability all large colours have  $|E_H(c)|\leq (1-\varepsilon)n/k$ and all vertices have $d_H(v)\geq (1-\varepsilon+18\varepsilon^2)n$.
Also, small colours $c$ always have $|E_H(c)|\leq (1-\varepsilon)n/k$. Thus with high probability $H$ is globally $(1-\varepsilon)n/k$-bounded
 and has $\delta({H})\geq (1-\varepsilon+18\varepsilon^2)n$, as required.
\end{proof}
By combining this with Lemmas~\ref{Lemma_Regular_subgraph} and~\ref{Lemma_Regular_subgraph_bipartite} we prove the main results of this section.
\begin{lemma}[Regularization lemma for high degree general graphs]\label{Lemma_regular_subgraph_few_large_colours}
Let $n^{-1}\LL \gamma\LL \varepsilon\LL 1$, and  let  $G$ be a properly coloured $n$-vertex graph with  $\leq (1-20\varepsilon)n$ colours having $\geq (1-20\varepsilon)n/2$ edges and $\delta(G)\geq (1-\varepsilon^2)n$. Then $G$ has a spanning subgraph $H$
which is globally $(1-\varepsilon)n/2$-bounded and  $(\gamma, \delta, n)$-regular for some $\delta \geq 1-\varepsilon+9\varepsilon^2$.
\end{lemma}
\begin{proof}
First apply Lemma~\ref{Lemma_subgraph_few_large_colours} with $k=2$ in order to get a subgraph $G'$ with $\delta(G')\geq (1-\varepsilon+18\varepsilon^2)n$ which is globally $(1-\varepsilon)n/2$-bounded.
 Then apply Lemma~\ref{Lemma_Regular_subgraph} to $G'$ with $\varepsilon'=\varepsilon- 18\varepsilon^2$ to get a subgraph $H$  which is  $r$-regular for $r\geq  (1-(\varepsilon- 18\varepsilon^2)-8(\varepsilon- 18\varepsilon^2)^2)n\geq (1-\varepsilon+ 9\varepsilon^2)n$ (using $\varepsilon\LL 1$).
\end{proof}
\begin{lemma}[Regularization lemma for high degree bipartite graphs]\label{Lemma_regular_subgraph_few_large_colours_bipartite}
Let $n^{-1}\LL \gamma\LL \varepsilon\LL 1$, and  let  $G$ be a properly coloured balanced bipartite graph on $2n$ vertices with $\leq (1-20\varepsilon)n$ colours having $\geq (1-20\varepsilon)n$ edges and $\delta(G)\geq (1-\varepsilon^2)n$. Then $G$ has a spanning subgraph $H$
which is globally $(1-\varepsilon)n$-bounded and  $(\gamma, \delta, n)$-regular for some $\delta \geq 1-\varepsilon+9\varepsilon^2$.
\end{lemma}
\begin{proof}
First apply Lemma~\ref{Lemma_subgraph_few_large_colours} with $k=1$ in order to get a subgraph $G'$ with $\delta(G')\geq (1-\varepsilon+18\varepsilon^2)n$ which is globally $(1-\varepsilon)n$-bounded.
 Then apply Lemma~\ref{Lemma_Regular_subgraph_bipartite} to $G'$ with $\varepsilon'=\varepsilon- 18\varepsilon^2$ to get a subgraph $H$ which is  $r$-regular for $r\geq  (1-(\varepsilon- 18\varepsilon^2)-8(\varepsilon- 18\varepsilon^2)^2)n-1\geq (1-\varepsilon+ 9\varepsilon^2)n$ (using $\varepsilon\LL 1$).
\end{proof}

\subsection{Regularization using a disjoint dense graph}
The following two lemmas take a graph $G$ which is close to being regular and a disjoint dense graph $E$, and modify $G$ slightly using edges of $E$ in order to produce a truly regular graph. Lemma~\ref{Lemma_Regularization_Expander} will be applied later in the paper, while Lemma~\ref{Lemma_Regularization_Expander_Induction} is a technical lemma to facilitate its proof.
\begin{lemma}\label{Lemma_Regularization_Expander_Induction}
For $d> m$, let $G$ be a balanced bipartite graph on $2n$ vertices with $\delta(G)\geq d-1$ and $2k=\sum_{v\in V(G)}\max(0, d(v)-d)$, let $E$ be an edge-disjoint  $(1, d-m)$-dense graph, and let $M$ be a  matching in $E$ of size $m-k$ such that $d_G(v)=d-1 \iff v\in V(M)$. Then, there is a subgraph $H\subseteq G$ and a matching $N$ in $E$ of size $m$ so that $H\cup N$ is $d$-regular.
\end{lemma}
\begin{proof}
The proof is by induction on $k$. The initial case is when $k=0$. In this case notice that every $v\in V(G)$ must have $\max(0, d(v)-d)=0$ which implies $\Delta(G)\leq d$. Since $\delta(G)\geq d-1$ and $d_G(v)=d-1 \iff v\in V(M)$ we have that $G\cup M$ is $d$-regular

Suppose that $k\geq 1$ and that the lemma holds for all $k'<k$. Let $X$ and $Y$ be the parts of $G$. Notice that since $\delta(G)\geq d-1$ and both $X$ and $Y$ have exactly $e(M)$ degree $d-1$ vertices, we must have $k=\sum_{x\in X}\max(0, d(x)-d)=e(G)-dn+e(M)=\sum_{y\in Y}\max(0, d(y)-d)$. In particular this implies that $X$ and $Y$ each have $\leq k$ vertices of degree $\geq d+1$.
Let $x\in X$, $y\in Y$ be vertices with $\max(0, d(x)-d), \max(0, d(y)-d)\geq 1$. We have $d(x)\geq d+1$, $d(y)\geq d+1$.
Since there are $m-k$ vertices in $Y$ of degree $d-1$ and $\leq k$ vertices in $Y$ of degree $\geq d+1$, $Y$ has at least $n-m$ vertices of degree $d$. Similarly  $X$ has at least $n-m$ vertices of degree $d$.
This implies that $N(x)$ and $N(y)$ have subsets of size $d-m$ consisting of vertices of degree $d$.
Since $E$ is  $(1,d-m)$-dense, there is an edge $uv\in E$ with $u\in N(x)$, $v\in N(y)$, and $d(v)=d(w)= d$.

Let $G'=G-xu-yv$ and $M'=M\cup uv$. We have that $M'$ is a matching because $d_G(u)=d_G(v)=d\neq d-1$ which implies $u,v\not\in V(M)$.  Notice that $\delta(G')\geq d-1$, $\sum_{w\in V(G)}\max(0, d_{G'}(w)-d)=  2k-2$, $e(M')=m-k+1$, and $d_{G'}(v)=d-1 \iff v\in V(M')$. By induction there is a subgraph $H$ of $G'$ and a matching $M$ in $E$ with the required properties.
\end{proof}

The following version of the above lemma will be easier to apply.
\begin{lemma}[Regularization using a disjoint dense graph]\label{Lemma_Regularization_Expander}
For $d > 2m$, let $G$ be a balanced bipartite graph on $2n$ vertices with $e(G)\leq dn+m$, $\delta(G)\geq d$, and  $E$ an edge-disjoint $(1, d/2)$-dense graph. Then there is a subgraph $H\subseteq G$  and a matching $M \subseteq E$ of size $\leq m$ so that $H\cup M$ is $d$-regular.
\end{lemma}
\begin{proof}
Since $\delta(G)\geq d$, we have that $\sum_{v\in V(G)}\min(0, d(v)-d)=2e(G)-2dn\leq 2m$.
The result follows by applying Lemma~\ref{Lemma_Regularization_Expander_Induction} with $m'=k=\sum_{v\in V(G)}\min(0, d(v)-d)/2\leq m$ and $M=\emptyset$. To see that $E$ is $(1, d-m')$-dense for this application, note that $d>2m$ is equivalent to $d-m>d/2$.
\end{proof}

\subsection{Regularization by adding vertices}
Here we show that a nearly-regular graph can be made regular by adding a small number of vertices (and edges adjacent to those vertices).
We will need the Gale-Ryser Theorem concerning which degree sequences are realisable by bipartite graphs (see \cite{ryser1963combinatorial}).
\begin{theorem}[Gale, Ryser]
Let $x_1\geq \dots \geq x_m$, and $y_1\geq \dots\geq y_n$ be non-negative numbers. There exists a bipartite graph with parts $X$, $Y$ with degree sequence $x_1, \dots, x_m$ in $X$ and $y_1, \dots, y_n$ in $Y$ if, and only if, we have
$\sum_{i=1}^n y_i=\sum_{i=1}^m x_i$ and $\sum_{i=1}^t y_i\leq \sum_{i=1}^m \min(t, x_i)$ for $t=1, \dots, n.$
\end{theorem}
We use the Gale-Ryser Theorem to prove the regularization lemma of this section.
\begin{lemma}[Regularization by adding vertices]\label{Lemma_Regularization_Add_Vertices}
Suppose that $0.01\geq \delta\GG \gamma \GG n^{-1}$.
Every $(\gamma, \delta, n)$-regular balanced bipartite graph $G$ on $2n$ vertices is contained in a $d$-regular balanced bipartite graph $G'$ with parts of size $\leq (1+9\gamma)n$ for $d=\lceil(1+5\gamma)\delta n\rceil$, where, additionally, all edges of $G'\setminus G$ touch $V(G')\setminus V(G)$.
\end{lemma}
\begin{proof}
Let $m = dn - e(G)$. Since $e(G)= \sum_{x\in X}d(x)=(1\pm\gamma)\delta n^2$, we have $2\gamma \delta n^2\leq m\leq 8\gamma \delta n^2$.
Let $X'$ and $Y'$ be two sets of new vertices with $|X'|=|Y'|=\left\lceil\frac md \right\rceil$.
Notice that $\gamma n\leq |X'|, |Y'|\leq 9\gamma n$.
Choose a graph $H$ between $X'$ and $Y'$ with exactly $d\left\lceil\frac md \right\rceil-m=d|X'|-m$ edges (a graph with this many edges exists since $\gamma\GG n^{-1}$ implies $|X'||Y'| \geq \gamma^2 n^2\geq d$).  Moreover, choose $H$ so that  $\Delta(H)$ is as small as possible. This ensures that $\delta(H)\geq \Delta(H)-1$. For each vertex $v\in V(G)\cup V(H)$ let $k_v= d-d_{G\cup H}(v)$.
Notice that $\delta(H)\geq \Delta(H)-1$ implies that any integer $t$ has either ``$t\leq d - \Delta(H)$'' or    ``$t\geq d-\delta(H)$'', and so
\begin{equation}\label{Eq_Regularization_Add_Vertices}
\sum_{x'\in X'} \min(t, k_{x'})=
\begin{cases}
\sum_{x'\in X'} t &\text{ if $t\leq d - \Delta(H)$}
\\
\sum_{x'\in X'} k_{x'} &\text{ if $t\geq d-\delta(H)$}
\end{cases}
\end{equation}
Since $G$ is $(\gamma, \delta, n)$-regular,  for $v\in X\cup Y$ we have $k_v\leq 9\gamma d$.
Notice that $\sum_{x'\in X'} k_{x'}=d|X'|- e_H(X', Y') = m= d|Y|- e_G(X,Y)=\sum_{y\in Y} k_y$.
Order the vertices of $Y$ as $y_1, \dots, y_n$ such that $k_{y_1}\geq k_{y_2}\geq\dots\geq k_{y_n}$.
Notice that, for all $t=1,\ldots,n$, $\sum_{i=1}^t k_{y_i}\leq 9\gamma d t \leq \gamma n t \leq |X'| t=\sum_{x'\in X'} t$. We also have $\sum_{i=1}^t k_{y_i}\leq \sum_{y\in Y} k_y= \sum_{x'\in X'} k_{x'}$.
Combining these with (\ref{Eq_Regularization_Add_Vertices}) we get that for all $t$ we have $\sum_{i=1}^t k_{y_i}\leq \sum_{x'\in X'} \min(t, k_{x'})$. By the Gale-Ryser Theorem, there is a graph $J_1$ between $X'$ and $Y$ such that $d_{J_1}(v)=k_v$.
By symmetry,  there is a graph $J_2$ between $X$ and $Y'$ with $d_{J_2}(v)=k_v$ for all $v\in X\cup Y'$. Now $G'=G\cup H\cup J_1\cup J_2$ is $d$-regular. Notice that all edges of $G'\setminus G$ touch $X'\cup Y'=V(G')\setminus V(G)$, completing the proof of the lemma.
\end{proof}

\section{Completion}\label{Section_Completion}
Our strategy for finding rainbow perfect matchings and Hamiltonian cycles is to first find nearly-perfect matchings or near-Hamiltonian cycles and then modify them.
In this section, we collect the ``modification lemmas'' which we use to complete nearly-spanning structures into truly spanning ones.

In all lemmas of this section we will have a dense graph which is disjoint from the matchings/cycles which we are trying to complete. The matchings/cycles are turned into what we want by modifying them gradually using edges of the dense graph.

\subsection{Perfect matchings}
The following lemma extends a matching by one edge.
\begin{lemma}\label{Lemma_rotation_matching}
Suppose that we have the following edge-disjoint subgraphs in a balanced bipartite graph $G$ with parts $X$, $Y$ of size $n$.
\begin{itemize}
\item   A matching $M$ with $e(M)\geq (1-\theta)n$.
\item  A $(1, \theta n)$-dense graph $E$.
\item  Graphs $D_X, D_Y$ with $d_{D_X}(x')\geq 2\theta n$ for each $x'\in X$ and  $d_{D_Y}(y')\geq 2\theta n$ for each $y'\in Y$.
\end{itemize}
Let $x\in X$ and $y\in Y$ be vertices outside $M$.
Then, there are vertices $u, v, m_u, m_v$ with $uv\in E$, $xm_u\in D_X$, $ym_v\in D_Y$, $um_u, vm_v\in M$ such that $M'=(M\cup\{xm_u, uv, ym_v\})\setminus\{um_u, vm_v\}$ is a matching.
\end{lemma}
\begin{proof}
Let $\sigma$ be the permutation which exchanges vertices of $M$ and fixes all other vertices.
Notice that since $d_{D_X}(x)\geq 2\theta n$,  $d_{D_Y}(y)\geq 2\theta n$ and $e(M)\geq (1-\theta)n$, we have $|N_{D_X}(x)\cap V(M)|, |N_{D_Y}(y)\cap V(M)|\geq  \theta n$.
Since $E$ is $(1, \theta n)$-dense, it has an edge $uv$ from $\sigma(N_{D_X}(x)\cap V(M))$ to $\sigma(N_{D_Y}(y)\cap V(M))$. Now taking $m_u=\sigma(u)$ and $m_v=\sigma(v)$ gives vertices satisfying the lemma.
\end{proof}

By iterating the above lemma, we can turn nearly-perfect rainbow matchings into perfect ones.
\begin{lemma}[Completing a matching]\label{Lemma_completion_matching}
Let $1\GG\theta, p \GG \varepsilon\GG n^{-1}.$ Suppose that we have the following colour-disjoint subgraphs in a properly coloured, balanced bipartite graph $G$ with parts $X,Y$ of size $n$.
\begin{itemize}
\item  A rainbow matching $M_0$ with $e(M_0)\geq (1- \varepsilon)n$.
\item  A $(p (\theta n)^2, \theta n)$-dense graph $E$.
\item  Graphs $D_X, D_Y$ with $d_{D_X}(x)\geq 2\theta n$ for each $x\in X$ and  $d_{D_Y}(y)\geq 2\theta n$ for each $y\in Y$.
\end{itemize}
Then, there is a perfect rainbow matching $N$ in $M_0\cup E\cup D_X\cup D_Y$ using $\leq \varepsilon n$ edges in each of $E$, $D_X$, and $D_Y$, where, additionally edges of $D_X$ in $N$ pass through $X\setminus V(M_0)$ and edges of $D_Y$ in $N$ pass through $Y\setminus V(M_0)$.
\end{lemma}
\begin{proof}
We will repeatedly apply Lemma~\ref{Lemma_rotation_matching} to produce rainbow matchings $M_1, \dots, M_{n-e(M_0)}$ with $e(M_i)=e(M_0)+i$ and $V(M_0)\subseteq V(M_i)$.
We will maintain  that $M_{i}$ always has at most $i$ edges of each of $E, D_X, D_Y$, with the edges of $M_i\cap D_X$ passing through $X\setminus V(M_0)$ and the edges of $M_i\cap D_Y$ passing through $Y\setminus V(M_0)$. When finished, $N=M_{n-e(M_0)}$ will then satisfy all the requirements of the lemma.

At the $(i+1)$st application, we apply Lemma~\ref{Lemma_rotation_matching} with  $M=M_{i}$,  vertices $x_i\in X\setminus V(M_i),y_i\in Y\setminus V(M_i)$, and  $E^{i}= E\setminus C(M_{i})$, $D_X^{i}=D_X\setminus C(M_{i}), D_Y^{i}=D_Y\setminus C(M_{i})$ -- the graphs $E, D_X, D_Y$ with all the edges with colour in $C(M_{i})$ removed.
Using Lemma~\ref{Lemma_Expander_Delete_Matching} and $e(E\cap M_i)\leq i$, the graph $E^{i}$ is  $(p(\theta n)^2-i\theta n, \theta n)$-dense (and so $(1, \theta n)$-dense since $\theta, p \GG \varepsilon$ implies $p(\theta n)^2\geq \varepsilon n\theta n$).
Also $e(D_X\cap M_i), e(D_Y\cap M_i)\leq i$ gives $d_{D_X^i}(x),d_{D_Y^i}(y)\geq 3\theta n-i\geq 2\theta n$ for each $x\in X$ and $y\in Y$. These show that the assumptions of Lemma~\ref{Lemma_rotation_matching} hold for  $M_i$, $E^i$, $D_X^i, D_Y^i$, $x_i$, $y_i$ and so we can apply it to obtain a matching $M_{i+1}$ containing one more edge than $M_i$. The matching $M_{i+1}$ is necessarily rainbow since it is a union of a submatching of $M_i$ (which is rainbow), and one  edge from each of $E^i$, $D_X^i,$ and $D_Y^i$ (which are all colour-disjoint from each other and from $M_i$). From Lemma~\ref{Lemma_rotation_matching}, the new edges of $M_{i+1}$ in $D_X$ and $D_Y$  pass through $x_i\in X\setminus V(M_0)$ and $y_i\in Y\setminus V(M_0)$ respectively, as required.
\end{proof}

\subsection{Hamiltonian cycles}
The following lemma joins two long cycles together using edges from some disjoint, dense graphs.
\begin{lemma}\label{Lemma_rotation_2_large_cycles}
Let $1\geq\lambda\GG p\GG\theta\GG n^{-1}$.
Let $C_1, C_2$ be two vertex-disjoint cycles of length $\geq \lambda n$. Let $E$, $F, G$ be  $\left(p(\theta n)^2, \theta n\right)$-dense graphs.
Suppose that $C_1, C_2$, $E, F,$ and $G$  are all edge-disjoint.
Let $xy\in E(C_1)\cup E(C_2)$.

Then, there is a cycle $C$ in $C_1\cup C_2\cup E\cup F\cup G$ with vertex set $V(C_1)\cup V(C_2)$ containing $1$ edge of each of $E, F, G$, and, additionally, $xy\in E(C)$.
\end{lemma}
{
\begin{proof} Without loss of generality, assume that $|C_1|\leq |C_2|$. Choose arbitrary orientations of $C_1$ and $C_2$. Let $\sigma$ be the permutation mapping $v$ to its successor (in the cycle containing $v$). Let $X_1$ be the set of vertices $x\in V(C_1)$ with $|N_E(x)\cap V(C_2)|< \theta n+2$ and let $X_2$ be the set of vertices $x\in V(C_1)$ with $|N_F(\sigma(x))\cap V(C_2)|< \theta n+2$.

Suppose $|X_1|\geq \lambda n/3$. Pick a set $Y_1\subset V(C_2)$ with $|Y_1|=|X_1|$ and note that $e_E(X_1,Y_1)\leq |X_1|\theta n< p|X_1|^2$ (as $\lambda \GG p,\theta$), contradicting that $E$ is $(p(\theta n)^2,\theta n)$-dense. Thus, $|X_1|<\lambda n/3$. Similarly, $|X_2|<\lambda n/3$, and thus we may pick $x_0\in V(C_1)\setminus (X_1\cup X_2\cup\{x,y\})$.

Let $y_0=\sigma(x_0)$. We have $|(N_E(x_0)\cap V(C_2))\setminus\{x,y\}|\geq \theta n$ and $|(N_F(y_0)\cap V(C_2))\setminus\{x,y\}|\geq \theta n$. Since $G$ is   $\left(p(\theta n)^2, \theta n\right)$-dense, there is an edge $zw\in E(G)$ with $\sigma(z)\in (N_E(x_0)\cap V(C_2))\setminus\{x,y\}$ and $\sigma(w)\in (N_F(y_0)\cap V(C_2))\setminus\{x,y\}$. Now $C=C_1\cup C_2-x_0y_0-z\sigma(z)-w\sigma(w)+x_0\sigma(z)+y_0\sigma(w)+zw$ is a cycle with the required properties (see Figure~\ref{Figure_Rotations}).
\end{proof}
}

\begin{figure}
    \includegraphics[width=1\textwidth]{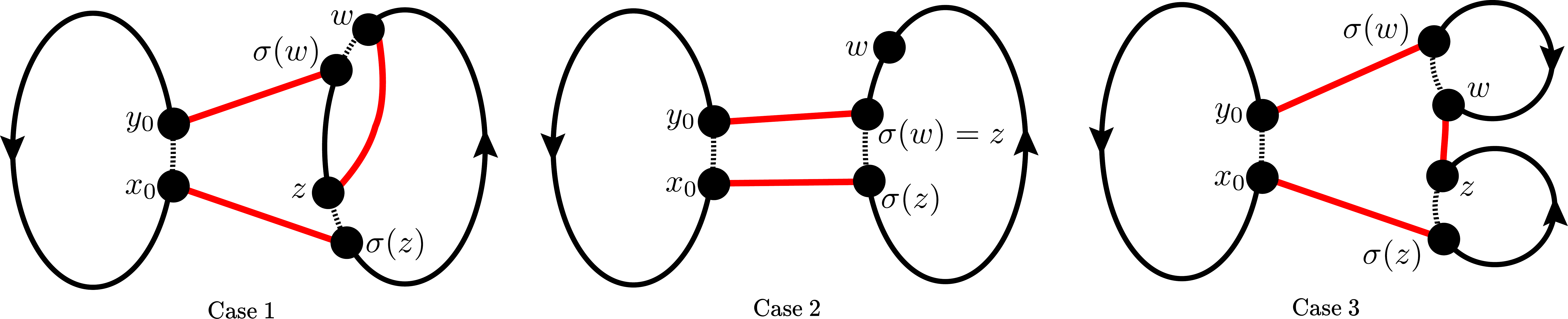}
  \caption{The three kinds of rotations we use in Lemmas~\ref{Lemma_rotation_2_large_cycles} and Lemma~\ref{Lemma_rotation_small_cycle}.
The dashed edges are removed from the cycles, while the solid red edges are added.  In each case, the resulting graph is a single cycle.
  In Lemma~\ref{Lemma_rotation_2_large_cycles}} Cases 1 and 2 are relevant, while in Lemma~\ref{Lemma_rotation_small_cycle} all three cases are relevant.

  Case 2 should be thought of as a degenerate version of Case 1 when we have $z=\sigma(w)$ (or symmetrically $w=\sigma(z)$). Note that this case never actually occurs because Lemmas~\ref{Lemma_rotation_2_large_cycles} and Lemma~\ref{Lemma_rotation_small_cycle} assume that $E, G$ are disjoint from all $C_i$. We include Case 2 in the figure for purposes of exposition.
\label{Figure_Rotations}
\end{figure}

Next we prove a similar lemma which joins cycles together. In the next two lemmas we will have both directed and undirected graphs on the same vertex set. The cycles we obtain in both lemmas will contain a mixture of edges from the directed and undirected graphs. In such a situation, ``cycle'' just means a graph which is a cycle after turning all directed edges into undirected edges, i.e.\ we do not care about directions of edges in cycles at all (the actual purpose of the directed edges in the lemmas is to control degrees through certain vertices).
\begin{lemma}\label{Lemma_rotation_small_cycle}
Let $1\geq \delta   \GG\theta \GG n^{-1}$ and $m\leq \theta n/2$.
Let $\mathcal C= \{C_0,\dots, C_m\}$ be a  2-factor. For each $i=0, \dots, m$ let $x_iy_i\in E(C_i)$. Let $E$ be   $(1 , \theta n)$-dense and $D_X$, $D_Y$  digraphs with $d^+(D_X), d^+(D_Y)\geq \delta n$.
Suppose that $\mathcal C$, $E$, $D_X,$ and $D_Y$ are all edge-disjoint.
Suppose that $|C_0|\leq \delta n/2$.

Then there is a 2-factor $\mathcal C'$ with {$\leq m$} cycles in which each cycle contains an edge $x_iy_i$ for some $i\geq 1$. Additionally, $\mathcal C'$ contains $1$ edge from $E$,  $1$ edge from $F$ which starts at $x_0$, and $1$ edge from $G$ which starts from $y_0$.
\end{lemma}
\begin{proof}
Choose arbitrary orientations of $C_1, \dots, C_m$. Let $\sigma$ be the permutation mapping $v$ to its successor (in the cycle containing $v$).
Let $U=\{x_i, y_i: i\geq 1\}$ and notice that $|U|\leq 2m\leq  \theta n$. We have $|N^+_{D_X}(x_0)\setminus (U\cup C_0)|\geq \delta n-\theta n-\delta n/2\geq  \delta n/3$ and $|N^+_{D_Y}(y_0)\setminus (U\cup C_0)|\geq  \delta n-\theta n-\delta n/2\geq  \delta n/3$. Since $E$ is   $(1 , \theta n)$-dense and $\theta \leq \delta n/3$, there is an edge $zw\in E$ with $\sigma(z)\in N^+_{D_X}(x_0)\setminus (U\cup C_0)$ and $\sigma(w)\in N^+_{D_Y}(y_0)\setminus (U\cup C_0)$. Now $\mathcal C'= \mathcal C-x_0y_0-z\sigma(z)-w\sigma(w)+x_0\sigma(z)+y_0\sigma(w)+zw$ is a 2-factor with the required properties (see Figure~\ref{Figure_Rotations}).
\end{proof}

By iterating the above lemmas we obtain the following lemma which turns a rainbow $2$-factor into a rainbow Hamiltonian cycle. Once again we have have a combination of directed and undirected graphs. When we say a directed graph is properly coloured, we mean that the underlying undirected graph is properly coloured.
\begin{lemma}[Completing a Hamiltonian cycle]\label{Lemma_completion_Hamiltonian}
Let $1\GG\delta \GG p\GG\theta\GG n^{-1}$ and $m\leq p\theta n$. Suppose that we have the following  colour-disjoint, edge-disjoint, properly coloured graphs on a set of $n$ vertices.
\begin{itemize}
\item A rainbow $2$-factor $\mathcal C=\{C_0,\dots, C_m\}$.
\item $\left(p(\theta n)^2+m \theta n, \theta n\right)$-dense graphs $E, F, G$.
\item Digraphs $D_X, D_Y$  with $d^+(D_X), d^+(D_Y)\geq \delta n+m$.
\end{itemize}
For each $i=0, \dots, m$, let $x_i, y_i$ be a pair of adjacent vertices in $C_i$.
Then, there is a rainbow Hamiltonian cycle $C$ in  $E\cup F\cup G \cup D_X\cup D_Y\cup C_0\cup\dots\cup C_m$ so that, additionally, $C\setminus (C_0\cup\dots\cup C_m)$ has $\leq m$ edges of each of $E, F, G, D_X, D_Y$, the edges of $C$ in $D_X$  all start in $\{x_0, \dots, x_m\}$, and the edges of $C$ in $D_Y$ all start in $\{y_0, \dots, y_m\}$.
\end{lemma}
\begin{proof}
Choose  $\lambda$ with $\delta\GG\lambda\GG p$.
The proof is by induction on $m$.  The initial case is when $m=0$ in which case the lemma is trivial. Suppose that $m\geq 1$. If  $|C_0|, |C_1|\geq  \lambda n$, then apply Lemma~\ref{Lemma_rotation_2_large_cycles} to $C_0, C_1, E, F, G$. If  $|C_0|$ or $|C_1| \leq  \lambda n$, then apply Lemma~\ref{Lemma_rotation_small_cycle} to $\mathcal C, E, D_X, D_Y$.
In either case we get a new $2$-factor $\mathcal C'$ with {$\leq m$} cycles, where each cycle contains an edge $x_iy_i$ for some $i$. Additionally $\mathcal C'$ contains at most one edge from each of $E, F, G, D_X, D_Y$, with remaining edges from $\mathcal C$. This implies that $\mathcal C'$ is rainbow since $\mathcal C$ was rainbow and $E, F, G, D_X, D_Y$, were all colour-disjoint from each other and from $\mathcal C.$ From Lemma~\ref{Lemma_rotation_small_cycle} we also have that if $\mathcal C'$ has edges in $D_X$ and $D_Y$, then they start at $x_0$ and $y_0$ respectively.

Let $E', F', G', D'_x, D'_Y$ be $E, F, G, D_X, D_Y$ with colours of $\mathcal C'\setminus \mathcal C$ deleted to get  $\left(p(\theta n)^2+m\theta n-\theta n, \theta n\right)$-dense graphs $E', F', G'$ (using Lemma~\ref{Lemma_Expander_Delete_Matching}), and  digraphs $D_X', D_Y'$ with $d^+(D_X'), d^+(D_Y')\geq \delta n+m-1$. The lemma holds by induction.
\end{proof}

\section{Near-decompositions into rainbow structures}
In this section, we prove our main results on matchings and Hamiltonian cycles. Most of the results here are of the form ``every properly coloured graph with certain properties can be nearly-decomposed into rainbow matchings/2-factors/Hamiltonian cycles''. These results build on one another. First we find near-decompositions into nearly-perfect matchings. Then we use completion results from the previous section to find near-decompositions into perfect matchings. We use these to find near-decompositions into 2-factors. Then we again use completion results from the previous section to find near-decompositions into Hamiltonian cycles.

\subsection{Nearly-perfect matchings}
In this section we show that every properly coloured $d$-regular, globally $d$-bounded bipartite graph has a near-decomposition into nearly-perfect rainbow matchings. This is proved by iteratively finding such matchings individually using Lemma~\ref{Lemma_near_perfect_matching}.
For a $d$-regular, globally $d$-bounded bipartite graph $G_d$, consider the following recursive process producing matchings $M_d, \dots, M_{\varepsilon d}$.
\begin{itemize}
\item[P1:]  For $t=d, \dots, \varepsilon d$, apply Lemma \ref{Lemma_near_perfect_matching} to $G_t$ in order to partition its edges into a randomized rainbow matching $M_t$ and a graph $G_{t-1}$.
\end{itemize}
We emphasise that this process is run in \emph{decreasing} order with $t$. The reason for this is that if we define the process like this, then the graphs $G_t$  turn out to be approximately $t$-regular for every $t$.

If we were able to run this process for $(1-o(1))d$ many steps, then we would obviously produce the desired $(1-o(1))d$ edge-disjoint nearly-perfect rainbow matchings in $G_0$. To show that we can run it for that long, we need to show that with high probability $G_t$ satisfies the assumptions of Lemma~\ref{Lemma_near_perfect_matching}.
There are two assumptions of that lemma which need to be maintained: $(\gamma, t/n, n)$-regularity and global $(1+\gamma) t$-boundedness.

\subsubsection*{Maintaining global boundedness}
Recall that in the matching $M_t$ produced in $G_t$ by Lemma~\ref{Lemma_near_perfect_matching}, every edge ends up in $M_t$ with  probability roughly $t^{-1}$. This means that at step $t-1$ of P1,
for any colour $c$ we have
\begin{align*}
\E(E_{G_{t-1}}(c))&= \P(c\not\in C(M_t))|E_{G_{t}}(c)|+\P(c\in C(M_t))(|E_{G_{t}}(c)|-1)\\
                  &\approx \left(1-\frac{|E_{G_{t}}(c)|}{t}\right)|E_{G_{t}}(c)|+\frac{|E_{G_{t}}(c)|}{t}(|E_{G_{t}}(c)|-1)\\
                  &\leq t-1
\end{align*}
Here, the last inequality is equivalent to ``$|E_{G_{t}}(c)|\leq t$''. Thus one would expect the global $t$-boundedness of $G_t$ to be preserved throughout the entire process. By using Azuma's Inequality, we can show that this happens with high probability.

\subsubsection*{Maintaining regularity}
Here we explain how to preserve   $(\gamma, t/n, n)$-regularity between the applications of Lemma~\ref{Lemma_near_perfect_matching}.
First notice that if Lemma~\ref{Lemma_near_perfect_matching} produced \emph{perfect} matchings, then there would be nothing to check---then $G_t$ would always be $t$-regular (and hence $(0, t/n,n)$-regular).
But the matchings produced by Lemma~\ref{Lemma_near_perfect_matching} have size $(1-o(1))n$, and so over time, one would expect the maximum degree of the graph to become bigger than $t$ after a  large number of steps. One thing that we will never lose is the minimum degree---the graphs will always have $\delta(G_t)\geq t$, since  there can be at most one edge from each matching $M_i$ present at any vertex.

To preserve regularity, we introduce another step to our process in addition to P1. Fix some large constant $k$, and do the following:
\begin{itemize}
\item[P2:] Whenever $t\equiv 0\pmod k$, modify $G_t$ slightly to make it into a $t$-regular graph.
\end{itemize}
This step ensures that $G_t$ is $(\gamma, t/n, n)$-regular for all $t$ and suitable $\gamma$. Indeed, for any $t$, there is some $k'\leq k$ with $G_{t+k'}$ $(t+k')$-regular and $\Delta(G_t)\leq \Delta(G_{t+k})$. Thus we have $t\leq \delta(G_t)\leq \Delta(G_t)\leq t+k$, which  implies that $G_t$ is $(k/t, t/n, n)$-regular.

Step P2 is performed using Lemma~\ref{Lemma_Regularization_Expander}. This lemma turns a graph $G_t$ with $\delta(G_t)\geq t$ into a $t$-regular graph $G_t'$ by deleting some edges and adding a small matching $N$ disjoint from $G_t$.  Edges in this matching $N$ are given a new ``dummy colour'' which was previously unused in $G_t$. While these dummy colours can end up in our matchings $M_t$, the total number of dummy colours is small (at most $n/k$), and so after deleting the dummy colours we still have nearly-perfect matchings in $G$.

\subsection*{A concentration lemma}
The following lemma will be used to show that the global boundedness of a graph decreases suitably after repeated applications of Lemma~\ref{Lemma_near_perfect_matching}.
\begin{lemma}\label{Lemma_martingale_colours_packing}
For $C\varepsilon\leq 0.1$ and $m\leq n$, suppose that we have random variables $X_0, X_1, X_2, \dots, X_{m}$,  with $n-m-1\geq \frac{n}{C}$ and $\frac nC\leq X_0\leq n$, { such that,  for every $t=0,\ldots,m-1$, and for any values of $X_0,\ldots,X_t$, we have}
$$X_{t+1}= \begin{cases}
X_t-1 \text{ with probability $(1-\varepsilon)\frac{1}{n-t}X_t$}\\
X_t \text{ with probability $1-(1-\varepsilon)\frac{1}{n-t}X_t$}.
\end{cases}$$

Then  we have
$$\P\left(X_m\geq  \left(1+3C\varepsilon \right) \left(1-\frac{m}{n} \right)X_0 \right)\leq e^{\frac{-\varepsilon^2 n }{ 18C^4}}.$$
\end{lemma}
\begin{proof}
Let $q=\P(X_{t+1}=X_t-1|X_0, \dots, X_t)= (1-\varepsilon)\frac{1}{n-t}X_t$,  and notice that $\E(X_{t+1}|X_0, \dots, X_{t})= X_t(1-q)+ (X_t-1)q= \left(\frac{n-t-1+\varepsilon}{n-t}\right)X_t$.
Let $Y_t=(1-C\varepsilon/n)^t\frac{X_t}{n-t}$. Notice that
$\E(Y_{t+1}|Y_0, \dots, Y_t)= (1-C\varepsilon/n)^{t+1}\frac{\E(X_{t+1}| X_{0}, \dots, X_{t})}{n-t-1}
=  (1-C\varepsilon/n)^{t+1}\frac{n-t-1+\varepsilon}{(n-t)(n-t-1)}X_t
= (1-C\varepsilon/n)\left(1+\frac{\varepsilon}{n-t-1}\right)Y_t
\leq (1-C\varepsilon/n)\left(1+{C\varepsilon}/{n}\right)Y_t\leq Y_t$.
This shows that $Y_t$ is a supermartingale.

Notice that for $t\leq m$ we have
\begin{align*}
|Y_{t+1}-Y_t|&= (1-C\varepsilon/n)^t\left|\frac{(1-C\varepsilon/n)X_{t+1}}{n-t-1} - \frac{X_t}{n-t}\right|\\
&=(1-C\varepsilon/n)^t\left|\frac{X_{t+1}-X_t}{(n-t-1)}-\frac{C\varepsilon X_{t+1}}{n(n-t-1)} +\frac{X_t}{(n-t-1)(n-t)}\right|\\
&\leq (1-C\varepsilon/n)^t\left(\frac{1}{n-t-1} + \frac{C\varepsilon }{n-t-1} +\frac{n}{(n-t-1)(n-t)}\right)\\
&\leq \frac{3C^2}n.
\end{align*}
The third inequality uses the triangle inequality, the 1-Lipschitzness of $X_t$, and $X_t\leq X_0\leq n$. the fourth inequality comes from $n-t-1\geq n-m-1\geq \frac{n}{C}$ and $C\varepsilon\leq 0.1$. Hence $Y_t$ is $\frac {3C^2}n$-Lipschitz.

{By Lemma~\ref{Lemma_Azuma_supermartingales}}, we have
$$\P\left(Y_m\geq  Y_0 + s \right)\leq e^{\frac{-s^2 n^2}{18 C^4 m}}.$$
Substituting $Y_0=X_0/n$, $s=\gamma X_0/n$, and using $X_0\geq n/C$  and $m\leq n$ gives
$$\P\left(Y_m\geq  (1+\gamma)\frac{X_0}{n} \right)\leq e^{\frac{-\gamma^2 X_0^2 }{18 C^4m}}\leq e^{\frac{-\gamma^2 n }{ 18C^6}}.$$
Substituting $Y_m=(1-C\varepsilon/n)^m\frac{X_m}{n-m}$, $\gamma=C\varepsilon$, and using $1+3C\varepsilon\geq \frac{1+C\varepsilon}{(1-C\varepsilon/n)^m}$ gives the lemma
$$ \P\left(X_m\geq  \left(1+3C\varepsilon \right) \left(1-\frac{m}{n} \right)X_0 \right)\leq\P\left(X_m\geq  (1+C\varepsilon)\frac{(n-m)X_0}{(1-C\varepsilon/n)^mn} \right)\leq e^{\frac{-\varepsilon^2 n }{ 18C^4}}.\qedhere$$
\end{proof}

\subsection*{Analysis of the random process}
Now we prove the first decomposition result of this paper. All our other decomposition results build on this.
It produces a near-decomposition into nearly-perfect rainbow matchings in a $\delta n$-regular graph which is globally $(1-\sigma)\delta n$-bounded.

\begin{lemma}\label{Lemma_Nearly_Rainbow_Decomposition}
Suppose that we have $n, \ell, \delta, \sigma$ with $n^{-1}\LL \sigma\LL  \delta \leq 1$ and $\ell\LL n$.

Let $G$ be a coloured balanced bipartite graph on $2n$ vertices which is $\delta n$-regular, locally $\ell$-bounded, and globally $(1-\sigma)\delta n$-bounded.
Then, $G$ has $(1- \sigma)\delta n$ edge-disjoint rainbow matchings $M_1, \dots, M_{(1-\sigma)\delta n}$  with $e(M_i) =(1-\sigma)n$  for all $i$.
\end{lemma}
\begin{proof}
Choose $k, \gamma, \nu, p$  so that $n^{-1}\LL \gamma \LL  p\LL  \nu\LL  k^{-1}\LL   \sigma \LL  \delta \leq 1$.
Set $D= \lceil (1-k^{-1})\delta n\rceil$.

\begin{claim}
There is a $D$-regular subgraph $G_D$ of $G$ and a  $(6pnD, \nu D/6)$-dense  graph $E_D$ with $G_D$ and $E_D$ edge-disjoint.
\end{claim}
\begin{proof}
Apply Lemma~\ref{Lemma_Bipartite_Complement_expander}  to $G$ with $\mu= \nu \delta/8$ and $p'= k^{-1}\delta$ in order to  find a $(\delta n-\lfloor k^{-1}\delta n\rfloor)$-regular subgraph $G_D$ of $G$ so that $E_D=\overline{G_D}$ is $(  0.48k^{-1}\delta(\nu\delta n/8)^2, \nu \delta n/8)$-dense (to apply Lemma~\ref{Lemma_Bipartite_Complement_expander} we use that $1\geq k^{-1}\delta, \nu \delta/8 \GG n^{-1}$).
Since  $\delta n-\lfloor k^{-1}\delta n\rfloor=D$, the graph $G_D$ is $D$-regular. Notice that $E_D$ is $(6pnD, \nu \delta n/8)$-dense (using $6pnD\leq  0.48k^{-1}\delta(\nu\delta n/8)^2$ which comes from $p\LL k^{-1}, \nu, \delta$). Now $\nu D/6\geq \nu \delta n/8$ implies the claim.
\end{proof}

We will define a random process producing spanning graphs $G_{D-1}, \dots, G_{\nu D}, E_{D-1}, \dots, E_{\nu D}$ and rainbow matchings $M_{D-1}, \dots, M_{\nu D}$. They will always have the following properties:
\begin{enumerate}[(a)]
\item $G_d$ is $(\gamma, d/n, n)$-regular.
\item $M_d\subseteq E(G_d)$ with $e(M_d)\geq (1-p)n$.
\item $\delta(G_d)\geq d$ and $e(G_d)\leq dn+kpn$.
\item $E_d$ is $\left(1, \frac{d-1}2\right)$-dense and edge-disjoint from $G_d$.
\item
\begin{itemize}
\item If $k\nmid d$ then $G_{d-1}\subseteq G_d$.
\item If  $k\mid d$ then $G_{d-1}\setminus G_d$ is a matching of size $\leq 4kpn$ in $E_d$. This matching has a dummy colour $c_d$.
\end{itemize}
\end{enumerate}
Notice (a) (c), and (d) hold for $G_D$ and $E_D$: indeed $G_D$ is $D$-regular (implying (a) and (c)), while $E_D$ is $(6pnD, \nu D/6)$-dense (which combined with $(D-1)/2\geq \nu D$ implies (d)).

The process producing  $G_{D-1}, \dots, G_{\nu D}, E_{D-1}, \dots, E_{\nu D}$, and  $M_{D-1}, \dots, M_{\nu D}$ is the following.
\begin{itemize}
\item If $G_{d}$ is globally $d$-bounded  then $M_{d}$ is the (random) matching produced from Lemma~\ref{Lemma_near_perfect_matching} applied to $G_d$ with $n=n$, $\gamma=\gamma$, $p'=p/20$, $\ell = \ell$, $\delta'=d/n$. Notice that  $1\geq \delta' \GG p' \GG \gamma  \GG  n^{-1}$ and $n\GG \ell$  hold with these parameters (using $\delta \geq \delta'\geq  \nu \delta/2$), allowing us to apply the lemma. Additionally:
\begin{itemize}
\item  If $k\nmid d$ then let $G_{d-1}=G_d\setminus M_d$ and $E_{d-1}=E_d$.
\item If $k\mid d$ then apply Lemma~\ref{Lemma_Regularization_Expander} to $G_d\setminus M_d$ and $E_d$ with $d'=d-1$ and $m=e(G_d\setminus M_d)-(d-1)n$ (using that by (d) $E_d$ is $\left(1, {d'}/2\right)$-dense,  by (c), $e(M_d)\geq (1-p)n$, and by $p\LL k^{-1}, \nu, \delta$ we have $2m=2(e(G_d\setminus M_d)-(d-1)n)\leq 4kpn\leq  \nu \delta n/2\leq d'$). We get a subgraph $H_d$ and a matching $N_d$ in $E_d$ of size $\leq 4kpn$.
Let $G_{d-1}=H_d\cup N_d$ with the edges of $N_d$ given the dummy colour $c_d$ to get a $(d-1)$-regular graph. Let $E_{d-1}=E_d\setminus N_d$.
\end{itemize}
\item If $G_{d}$ is not globally $d$-bounded then we stop the process.
\end{itemize}

We show that the properties we need in the process hold.
\begin{claim}
As long as the process goes on (a) -- (e) hold.
\end{claim}
\begin{proof}
First we prove (b), (d), and (e).
\begin{enumerate}
\item [(b)]  This is immediate since $M_d$ was produced by applying Lemma~\ref{Lemma_near_perfect_matching} to $G_{d}$ with $p'=p/20$.
\item [(d)] We have $E_d=E_D\setminus \bigcup_{\substack{i\in [d+1, D], k\mid i}}N_{i}$ and $e(N_i)\leq 4kpn$. As $E_D$ is $(6pnD, \nu D/6)$-dense, this implies that $E_d$ is $(6pnD - 4kpn\lceil(D-d)/k\rceil, \nu D/6)$-dense (using Lemma~\ref{Lemma_Expander_Delete_Edges}). Since $\frac{d-1}2\geq \nu D/6 $ and $6pnD - 4kpn\lceil(D-d)/k\rceil\geq 1$, this implies that $E_d$ is $(1, \frac{d-1}2)$-dense
\item [(e)] This is immediate from the construction of $G_{d-1}$: When $k\nmid d$, then $G_{d-1}=G_d\setminus M_d\subseteq G_d$ holds. When $k\mid d$, then $G_{d-1}=N_d\cup H_d\subseteq N_d\cup(G_d\setminus M_d)$ with $N_d$ ``a matching in $E_d$ of size $\leq 4kpn$'' which has dummy colour $c_d$.
\end{enumerate}
Next we prove (a) and (c). First recall that (a) and (c) hold for the starting graph $G_D$.  If $k\mid d+1$, then, by construction, $G_{d}$ is $d$-regular which implies both (a) and (c).

Suppose then that $k\nmid d+1$ and $d<D$.
Fix $\hat d= \min(k\lceil (d+1)/k\rceil-1, D)$. Notice that $G_{\hat d}$ is $\hat d$-regular (as $\hat d=D$ or $k\mid\hat d+1$ in which case $G_{\hat d}$ is $\hat d$-regular from the application of Lemma~\ref{Lemma_Regularization_Expander}). Also notice that $0< \hat d - d\leq k-1$ (the first inequality comes from $d<D$ and $k\nmid d+1$. The second inequality comes from $k\lceil (d+1)/k\rceil-1\leq d+k-1$).
Because $k\nmid \hat d, \hat d-1, \dots, d+1$, we have that
\begin{equation}\label{Eq_GdForm}
G_{d}=G_{\hat d}\setminus (M_{\hat d}\cup M_{\hat d-1}\cup\dots \cup M_{d+1}).
\end{equation}
Since the graphs $M_i$ are matchings, this implies $\delta(G_d)\geq \delta(G_{\hat d})-(\hat d -d)=d$.
Next, (\ref{Eq_GdForm}) implies $\Delta(G_{d-1})\leq \Delta(G_{\hat d})= \hat d \leq d+k$, which combined with $\delta(G_{d})\geq d$ implies that $G_d$ is $(k/d, d/n, n)$-regular and hence $(\gamma, d/n, n)$-regular since $k/d\leq k/\nu D\leq 2k/\nu\delta n\leq  \gamma$. Thus, (a) holds.
Finally, (\ref{Eq_GdForm}) implies  $e(G_{d})\leq e(G_{\hat d})-e(M_{\hat d})- \dots -e(M_{d+1})\leq \hat d n - (\hat d - d)(1-p)n\leq dn+kpn$, completing the proof that (c) holds.
\end{proof}

To show that the process does not end too early, it remains to show that $G_{d}$ is  globally $d$-bounded.
\begin{claim}
With probability $\geq 1-n^{-1}$, $G_{d}$ is globally $d$-bounded for $d=D, \dots, \nu D$.
\end{claim}
\begin{proof}
Notice that by (e) and $p\LL \nu, k^{-1}, \delta$, for any dummy colour $c_i$ we always have $|E_{G_d}({c_i})|\leq 4kpn\leq  \nu D/2\leq d$.
For a non-dummy colour $c$, if $|E_{G_D}(c)|\leq \nu D/2$ then $|E_{G_d}(c)|\leq\nu D/2\leq d$ for all $d$.
It remains to prove the claim for non-dummy colours with $|E_{G_D}(c)|\geq \nu D/2$.
Let $c$ be such a  colour.

Let $Y_d^c$ be the number of colour $c$ edges in $G_d$.
From Lemma~\ref{Lemma_near_perfect_matching} we have that as long as the process goes on, we have
$\P(e\in E(M_d)| e\in E(G_d))\geq (1-  p)\frac{1}{d}$.
Since $M_d$ is rainbow, for any colour $c$ the events ``$e\in E(M_d)$'' and ``$f\in E(M_d)$'' are disjoint for distinct $e,f\in E(c)$. This implies
$\P(c\in C(M_d)| Y_d^c=y)\geq (1-  p)\frac{y}{d}.$
Recall that, $G_{d-1}$ consists of a subgraph of $G_d\setminus M_d$ together with maybe some dummy colour edges.  This implies that the event ``$c\in C(M_d)$'' is contained in the event ``$Y_{d-1}^c\leq Y_{d}^c-1$'', which gives following.
$$Y^c_{d-1}= \begin{cases}
\leq Y^c_d-1 \text{ with probability $\geq (1-p)\frac{1}{d}Y^c_d$}\\
Y^c_d \text{ with probability $\leq 1-(1-p)\frac{1}{d}Y^c_d$}.
\end{cases}$$
Let $Z_D^c, \dots, Z_{\nu D}^c$ be random variables with $Z_D^c=Y^c_D=|E_{G_D}(c)|$ and, for each $d=D,\ldots,\nu D+1$, for any values of $Z_D^c, \dots, Z_{d}^c$, we have
$$Z_{d-1}^c=\begin{cases}
Z^c_d-1 \text{ with probability $(1-p)\frac{1}{d}Z^c_d$}\\
Z^c_d \text{ with probability $1-(1-p)\frac{1}{d}Z^c_d$}
\end{cases}$$
Notice that $Y_d$ is stochastically dominated by $Z_d$, i.e.\ specifically we have $\P(Y_d^c\geq x)\leq \P(Z_d^c\geq x)$ (to see this note, that $Y_d$ and $Z_d$ can be coupled so that $Y_d$ is always bounded above by $Z_d$).

Let $X_t^c=Z_{D-t}^c$ and notice that $X_0^c, \dots, X^c_{m}$ satisfy the assumptions of Lemma~\ref{Lemma_martingale_colours_packing} with $n=D$, $\varepsilon=p$, $C=2\nu^{-1}$, $m=D-d$.
Therefore we have
$$\P\left(Z^c_d\geq  \frac{d}{D} Z^c_D(1 + 6\nu^{-1}p) \right)=\P\left(X^c_{D-d}\geq  \left(1-\frac{D-d}{D}\right) X^c_0(1 + 6\nu^{-1}p) \right)  \leq 4e^{- 0.001\nu^{4} p^2 n }\leq n^{-4}.$$
The last inequality comes from  $0.001\nu^{4} p^2 \GG n^{-1}$.
Note that $(1+6\nu^{-1}p)\frac{d}{D}Z^c_D\leq (1+6\nu^{-1}p)\frac{d}{D}(1-\sigma)\delta n\leq d$ (using the global $(1-\sigma)\delta n$-boundedness of $G$ and  $p\nu^{-1}, k^{-1}\LL\sigma$).

Thus for every colour $c$ with $|E_{G_D}(c)|\geq \nu D/2$, we have $\P(|E_{G_d}(c)|\geq d)=\P(Y_d\geq d)\leq \P(Z_d\geq d)\leq n^{-4}$ for all $d$.
By the union bound over all $c$ and $d$, we have that $G_d$ is globally $d$-bounded for all $d\geq \nu D$ with probability $\geq 1-n^{-1}$.
\end{proof}

Since there are $\leq \lceil D/k\rceil\leq  \sigma n/2$ dummy colours we have that each $M_i$ contains a submatching $M_i'$ of size $\geq e(M_i)-\lceil D/k\rceil\geq (1-p)n -\sigma n/2\geq (1 - \sigma)n$ containing no dummy colours (using (b)). Finally notice that we have $(1-\nu)D\geq (1-\sigma) \delta n$ matchings (using $\sigma\GG\nu, k^{-1}$).
\end{proof}

The following lemma takes a decomposition into rainbow matchings (as in the previous lemma) and outputs another such decomposition where the matchings are nicely spread out around the vertex set.
\begin{lemma}\label{Lemma_Nearly_Rainbow_Decomposition_Spread_Out}
Let $t\geq pn$, $pt \geq 1$, and $p\leq \frac12$.
Let $G$ be a properly coloured balanced bipartite graph on $2n$ vertices with $\delta(G)\geq t$ and $M_1, \dots, M_{t}$ edge-disjoint rainbow matchings in $G$ with $e(M_i)\geq (1-p^3)n$ for all $i$.
Then $G$ has edge-disjoint rainbow matchings $M'_1, \dots, M'_{(1-p)t}$  with $e(M_i') \geq(1- 10p)n$  for all $i$ and also  $\delta(M_1'\cup \dots \cup M_{t}')\geq {(1-101p)t}$.
\end{lemma}
\begin{proof}
For a bipartite graph $H$ on $2n$ vertices with $\Delta(H)\leq t$, let $f_H(v)= \max((1-100p)t - d_H(v), 0)$ and $f(H)=\frac 12\sum_{v\in V(H)}f(v)$. Notice that $f(H)\leq tn-e(H)$ (since $\Delta(H)\leq t$ implies that all vertices $v$ have $f_H(v)\leq t-d_H(v)$).
For $H= M_1\cup \dots \cup M_{t}$ notice that $e(H)\geq (1-p^3)tn$ and hence $f(H)\leq tn-(1-p^3)tn= p^3tn$.

Let $N_1, \dots, N_t$ be a family of $t$ edge-disjoint rainbow matchings in $G$ such that $H=N_1\cup\dots \cup N_t$ has $e(H)-4f(H)$ as large as possible. By the previous paragraph we have $e(H)-4f(H)\geq (1-5p^3)tn$.
\begin{claim}
$f(H)=0$.
\end{claim}
\begin{proof}
Suppose, for contradiction, that {$f(H)>0$}.
Let $U$ be the set of vertices $u$ in $H$ with $d_H(u)\leq (1-10p)t$.
Notice that vertices $v$ outside $U$ have $(1-100p)t - d_H(v)\leq -90pt <-1$ (using $pt\geq 1$).
We have $|U|(1-10p)t+ (2n-|U|)t \geq 2e(H)\geq 2e(H)-8f(H)\geq 2(1-5p^3)tn$, which is equivalent to $|U|\leq p^2 n$.

Let $u$ be a vertex with {$f_H(u)>0$}. This implies that  $d_H(u)< (1-100p)t$ and hence $d_G(u)-d_H(u)\geq 100p t > 4|U|$ (using $t\geq pn$).
Choose some matching $N_i$  with $u\not\in V(N_i)$ (one must exist since $d_H(u)< (1-100p)t$).
Let $N_i'\subset N_i$ be the set of edges in $N_i$ that touch a vertex in $U$, so that $|N_i'|\leq 2|U|$. As $d_G(u)-d_H(u)\geq 4|N_i|$ and $G$ is properly coloured, we can pick some $y\in N_{G\setminus H}(u)$ so that $y$ is in no edge in $N_i'$ and $uy$ has colour outside of $C(N_i')$.

Let $F_y\subset N_i'$ be the set of edges in $N_i$ with the same colour as $uy$ or which contain $y$, noticing that, by construction, $V(F_y)\cap U=\emptyset$.
Let $N_i''=(N_i\setminus F_y)\cup \{uy\}$ to get a rainbow matching. Let $H'=(H\setminus N_i)\cup N_i''$ to get another union of $t$ rainbow matchings. Notice that $f_{H'}(u)=f_{H}(u)-1$ and $f_{H'}(w)=f_H(w)$ for $w\neq u$ (using $V(F_y)\cap U=\emptyset$), so that $f(H')\leq f(H)-1/2$. We also have $e(H')\geq e(H)+1-e(F_y)\geq e(H)-1$.
Combining these gives $e(H')-4f(H')\geq (e(H)-1)-4(f(H)-\frac12)>e(H)-4f(H)$, which is a contradiction to the maximality of $H$.
\end{proof}
Notice that $f(H)=0$ is equivalent to $\delta(H)\geq (1-100p)t$.
Recall that $\sum_{i=1}^t{e(N_i)}=e(H)\geq e(H)-4f(H)\geq (1-5p^3)tn$.
These imply that $H$ has at most $pt$ matchings of size $\leq (1- 10p)n$ (since otherwise we would have $e(H)\leq pt(1-10p)n + (t-pt)n= (1-10p^2)tn< (1-5p^3)tn$, giving a contradiction to $e(H)\geq (1-5p^3)tn$). The union of the remaining matchings has minimum degree $\geq (1-101p)t$ and so satisfies the lemma.
\end{proof}

The following lemma strengthens our previous rainbow matching decomposition result (Lemma~\ref{Lemma_Nearly_Rainbow_Decomposition}). It bootstraps that lemma in two different ways. First it removes the condition that the host graph is regular, replacing this with the condition that it is $(\gamma, \delta, n)$-regular. Secondly, the decomposition produced is nicely spread out as in Lemma~\ref{Lemma_Nearly_Rainbow_Decomposition_Spread_Out}.
\begin{lemma}\label{Lemma_NearMatching_Decomposition_Not_Regular}
Let $n^{-1}\LL  \gamma \LL  p, \delta \leq 1$.
Let $G$ be a properly coloured balanced bipartite graph on $2n$ vertices, which is $(\gamma, \delta, n)$-regular, and globally $(1 -  p)\delta n$-bounded.
Then, $G$ has edge-disjoint rainbow matchings $M_1, \dots, M_{(1-p)\delta n}$  with $e(M_i) \geq (1-p)n$  for all $i$ and also  $\delta(M_1\cup \dots \cup M_{(1-p)\delta n})\geq {(1-2p)\delta n}$.
\end{lemma}
\begin{proof}
We will first prove the lemma when we additionally have ``$p\LL \delta\leq 0.01$''.
Assuming that, choose $n^{-1}\LL \gamma \LL    \sigma   \LL  \rho \LL    p\LL  \delta \leq 0.01$.

By Lemma~\ref{Lemma_Regularization_Add_Vertices}, there is a balanced bipartite, regular graph $G'$ containing $G$ which has $\leq 11\gamma n$ extra vertices in each part with $G'[V(G)]=G$. Colour the edges of $G'$ so that the edges of $G$ retain their colours, and any edge $e\not\in E(G)$ gets a new colour $c_e$ (which only occurs on $e$). Let $n'=|G'|/2\leq (1+11\gamma)n$ be the size of the parts of $G'$. Since $G'[V(G)]=G$ and $G$ was $(\gamma, \delta, n)$-regular, the graph $G'$ is $(\delta' n')$-regular for some $\delta'n'=(1\pm\sigma)\delta n$ (using $n^{-1}\LL \gamma\LL \sigma \LL \delta \leq  1$).
Notice that $G'$ is globally $(1-\sigma)\delta' n'$-bounded (using $(1-p)\delta n\leq  (1-\sigma)\delta' n'$ which comes from $\gamma\LL\sigma \LL p$).

By Lemma~\ref{Lemma_Nearly_Rainbow_Decomposition} applied to $G'$ with $\ell=1$, we have $(1-\sigma)\delta' n'\geq (1-2\sigma)\delta n$ edge-disjoint rainbow matchings $M_1, \dots, M_{(1-\sigma)\delta' n'}$  with $e(M_i) =(1- \sigma)n'$. For each $i$, let $M'_i=M_i[V(G)]$ to get a rainbow matching in $G$ with $e(M_i) \geq (1- \sigma)n'- 22\gamma  n\geq (1-\rho^3)n$ (using $\gamma\LL  \sigma \LL  \rho$).   Notice that $\delta(G)\geq (1-\gamma)\delta n\geq (1-2\sigma)\delta n$.

Apply Lemma~\ref{Lemma_Nearly_Rainbow_Decomposition_Spread_Out} to $G$ and $M'_1, \dots, M'_{(1-2\sigma)\delta n}$ with $t=(1- 2\sigma)\delta n$, $n=n$,  and $p=\rho$. This gives   edge-disjoint rainbow matchings $M''_1, \dots, M''_{(1 - \rho)(1-2\sigma)\delta n}$  with $e(M''_i) =(1- 10\rho)n$  for all $i$ and also
$\delta(M''_1\cup \dots \cup M''_{(1 - \rho)(1-2\sigma)\delta n})\geq (1-101\rho)(1-2\sigma)\delta n$.
Notice that $(1 -  \rho)(1-2\sigma)\delta n\geq (1-p)\delta n$ (using $\rho, \sigma\LL p$) and consider the matchings $M''_1, \dots, M''_{(1 - p)\delta n}$.
Now the lemma holds for these matchings  because  $e(M_i'')= (1- 10\rho)n\geq (1-p)n$, and $\delta(M''_1\cup \dots \cup M''_{(1 - p)\delta n})\geq \delta(M''_1\cup \dots \cup M''_{(1 - \rho)(1-2\sigma)\delta n}) -  ((1 - \rho)(1-2\sigma)\delta n - (1 - p)\delta n) \geq (1-101\rho)(1-2\sigma)\delta n- (p - \rho -2\sigma+2\rho\sigma)\delta n\geq (1-2p)\delta n$  (using $\rho, \sigma\LL p$).

Now we will prove the general case when we just have  $n^{-1}\LL  \gamma \LL  p, \delta \leq 1$.
Choose $\hat p$ such that $n^{-1}\LL  \gamma\LL \hat p\LL p,\delta\leq 1$.
Apply Lemma~\ref{Lemma_Random_Subgraph_Bipartite} (b) to $G$ with $p'=0.01$ in order to partition the edges of $G$ into $100$ spanning subgraphs $G_1, \dots, G_{100}$ with each $G_i$ $(2\gamma, 0.01\delta, n)$-regular and globally $(1+\gamma)(1-p)0.01\delta n\leq (1-\hat p)0.01\delta n$-bounded.
 For $i=1, \dots, 100$, applying the ``$\hat p \LL \delta\leq 0.01$'' case of the lemma to  $G_i$ gives a family of rainbow matchings  $M^i_1, \dots, M^i_{(1-\hat p)0.01\delta n}$  with $e(M_j^i) \geq(1-\hat p)n$  for all $i$ and also  $\delta(M_1^i\cup \dots \cup M^i_{(1-\hat p)\delta n})\geq {(1-\hat p)0.01\delta n}$. Taking the union of these families for $i=1, \dots, 100$ gives the required edge-disjoint rainbow matchings (using $\hat p \leq p$).
\end{proof}

\subsection{Perfect matchings}\label{Section_Perfect_Matchings}
In this section we find near-decompositions of graphs into perfect rainbow matchings.
This is done by taking the near-decompositions into nearly-perfect rainbow matchings produced in the previous section and then
 using a completion lemma from Section~\ref{Section_Completion} to turn them into perfect matchings. The most straightforward to prove version of this is the following.
\begin{lemma}\label{Lemma_decomposition_completion}
Let $1\GG \theta, p \GG \varepsilon \GG n^{-1}$ and $t\leq n$.
Suppose that we have the following edge-disjoint subgraphs in a properly coloured balanced bipartite graph on $2n$ vertices.
\begin{itemize}
\item Rainbow matchings $M_1, \dots, M_t$   with $e(M_i)\geq (1- \varepsilon)n$ and $\delta(M_1\cup \dots \cup M_t)\geq t-  10\varepsilon n$.
\item A  $(2p (\theta n)^{2}, \theta n)$-dense graph $E$.
\item Graphs $D_X, D_Y$  with $\delta(D_X), \delta(D_Y)\geq 4\theta n$.
\end{itemize}
Additionally suppose that $M_1\cup\dots\cup M_t$, $E$, $D_X$, and $D_Y$ are colour-disjoint.
Then there there are edge-disjoint perfect rainbow matchings $M'_1, \dots, M'_t$ in $E\cup D_X\cup D_Y\cup M_1\cup \dots \cup M_t$.
\end{lemma}
\begin{proof}
Let $X$ and $Y$ be the parts of the bipartition.
We construct the matchings $M'_1, \dots, M'_t$ one-by-one using Lemma~\ref{Lemma_completion_matching}.  They will have the following properties.
\begin{enumerate}[(i)]
\item $xy\in E(D_X\cap M'_i) \implies x\not\in V(M_i)\cap X$.
\item $xy\in E(D_Y\cap M'_i) \implies y\not\in V(M_i)\cap Y$.
\item $e(E \cap M'_i)\leq \varepsilon n$.
\end{enumerate}
Suppose that we have constructed matchings $M'_1, \dots, M'_s$ satisfying the above properties. Let $E^s=E\setminus (M'_1\cup \dots\cup M'_s)$, $D^s_X=D_X\setminus (M'_1\cup \dots\cup M'_s)$, $D^s_Y=D_Y\setminus (M'_1\cup \dots\cup M'_s)$. Using (i), $\varepsilon\LL \theta$, and  $\delta(M_1\cup \dots \cup M_{t})\geq t- 10\varepsilon n$,  we have $d_{D^s_X}(x)\geq d_{D_X}(x)-|\{i: x\not\in V(M_i)\}|\geq \delta(D_X)-10\varepsilon n\geq 3\theta n$ for any $x\in X$. Similarly  $d_{D^s_Y}(y)\geq 3 \theta n$ for $y\in Y$. By (iii) and Lemma~\ref{Lemma_Expander_Delete_Edges}, the graph $E_s$ is $(2p(\theta n)^2-  s\varepsilon n, \theta n)$-dense.  Hence, using $s\varepsilon n\leq \varepsilon n^2\leq p\theta^2 n^2$, $E_s$ is $(p(\theta n)^2, \theta n)$-dense. By Lemma~\ref{Lemma_completion_matching} applied to $M_{s+1}$, $E^s, D_X^s$ and $D_Y^s$, there is a rainbow perfect matching $M'_{s+1}$ satisfying (i) -- (iii).
\end{proof}

Dense graphs are less convenient to work with than typical ones. The following lemma is a version of the previous one which replaces the colour-disjoint dense graphs by a single colour-disjoint typical one.
\begin{lemma}\label{Lemma_decomposition_completion_codegrees}
Let $1\geq p\GG\varepsilon\GG\gamma\GG n^{-1}$. {Suppose that we have the following edge-disjoint subgraphs in a properly coloured balanced bipartite graph on $2n$ vertices for some $t\leq n$.
\begin{itemize}
\item Rainbow matchings $M_1, \dots, M_t$  with $e(M_i)\geq (1-\varepsilon)n$ for each $i$, and $\delta(M_1\cup \dots \cup M_t)\geq t-10\varepsilon n$.
\item A $(\gamma, p, n)$-typical, balanced bipartite graph $G$ which is colour-disjoint from $M_1,\ldots,M_t$.
\end{itemize}}
Then, there there are edge-disjoint perfect rainbow matchings $M'_1, \dots, M'_t$ in $M_1\cup \dots\cup M_t\cup G$.
\end{lemma}
\begin{proof}
Choose $1\geq p\GG p_1 \GG\theta \GG\varepsilon\GG\gamma\GG n^{-1}.$
Choose   three disjoint sets of colours $C_E, C_{D_X},C_{D_Y}$ from $G$, with each colour put independently into  $C_E, C_{D_X}$ and $C_{D_Y}$  with probability $2p_1p^{-1}/0.99$, $5\theta p^{-1}$ and $5\theta p^{-1}$ respectively (this is possible since  $p\GG p_1 \GG\theta $ implies $2p_1p^{-1}/0.99+5\theta p^{-1} + 5\theta p^{-1}\leq 1$). Let $E$, $D_X$, $D_Y$ be the subgraphs of $G$ with colours from  $C_E, C_{D_X}, C_{D_Y}$ respectively.
By Lemma~\ref{Lemma_Random_Subgraph_Bipartite} (a), with high probability   $E$ is  $(2\gamma, 2p_1/0.99, n)$-typical and $D_X$, $D_Y$ are $(2\gamma, 5\theta , n)$-typical.

We have that $\delta(D_X), \delta(D_Y)\geq (1-2\gamma)5\theta  n\geq 4\theta  n$.
By Lemma~\ref{Lemma_Codegrees_Pseudorandom} applied with $\mu=\theta $, $p=p_1$ and $\gamma=\gamma$, $E$ is $( 2p_1(\theta  n)^2, \theta  n)$-dense. By Lemma~\ref{Lemma_decomposition_completion} applied with $\theta=\theta $ and $p=p_1$ we obtain the required perfect matchings.
\end{proof}

Combining the above with Lemma~\ref{Lemma_NearMatching_Decomposition_Not_Regular} we get the following versatile lemma guaranteeing near-decompositions into perfect rainbow matchings.
\begin{lemma}\label{Lemma_PerfectMatching_Decomposition_ExtraPseudorandom}
Suppose that we have $n, \delta, p$ with $n^{-1}\LL  \gamma \LL{} p, \delta \leq 1$.
Let $G$ be a  properly coloured, $(\gamma, \delta, n)$-regular, globally $(1- p)\delta n$-bounded, balanced bipartite graph of order $2n$.
Let $H$ be a {properly coloured}, colour-disjoint, $(\gamma, p, n)$-typical, balanced bipartite graph on the same vertex set as $G$.
Then $G\cup H$ has edge-disjoint perfect rainbow matchings $M_1, \dots, M_{(1-p)\delta n}$.
\end{lemma}
\begin{proof}
Choose $n^{-1}\LL  \gamma\LL{} \varepsilon \LL{}p, \delta \leq 1$.
By Lemma~\ref{Lemma_NearMatching_Decomposition_Not_Regular} there are  edge-disjoint rainbow matchings $M_1,$ $\dots,$ $M_{(1-\varepsilon)\delta n}$  with $e(M_i) =(1-\varepsilon)n$  for all $i$ with  $\delta(M_1\cup \dots \cup M_{(1-\varepsilon)\delta n})\geq {(1-2\varepsilon)\delta n}$. By Lemma~\ref{Lemma_decomposition_completion_codegrees}  applied with $t=(1-\varepsilon)\delta n$ there are edge-disjoint perfect rainbow matchings $M'_1, \dots, M'_{(1-\varepsilon)n}$ in $H\cup G$.
\end{proof}

As a corollary, we obtain that a  typical properly coloured graph can be nearly-decomposed into perfect rainbow matchings as long as there is a gap between its global boundedness and its degrees.
\begin{corollary}\label{Lemma_PerfectMatching_Decomposition_Typical}
Suppose that we have $n, \delta, p, \gamma$ with $n^{-1}\LL  \gamma \LL{} p, \delta \leq 1$.
Every properly coloured,  $(\gamma, \delta, n)$-typical,  globally $(1- p)\delta n$-bounded, balanced bipartite graph $G$  of order $2n$ has
 $(1-p)\delta n$ edge-disjoint perfect rainbow matchings.
\end{corollary}
\begin{proof}
By Lemma~\ref{Lemma_Random_Subgraph_Bipartite} (a), $G$ can be partitioned into a $(2\gamma, p\delta/2, n)$-typical graph $H$ and a  colour-disjoint $(2\gamma, \delta-p\delta/2, n)$-typical graph $G'$. Since $(1- p)\delta n\leq (1-p\delta/2)(1-p/2)\delta n$, $G'$ is globally $(1-p\delta/2)(\delta-p\delta/2) n$-bounded.
By Lemma~\ref{Lemma_PerfectMatching_Decomposition_ExtraPseudorandom} applied with $p'=p\delta/2$, and $\delta'=\delta-p\delta/2$, $G'\cup H$ has  $(1-p\delta/2)(\delta-p\delta/2) n\geq (1-p)\delta n$ edge-disjoint perfect rainbow matchings.
\end{proof}

Applying the above lemma when the host graph is $K_{n,n}$ we can show that any proper colouring of $K_{n,n}$ has a  near-decomposition into perfect rainbow matchings under natural conditions on the sizes of the colour classes.
\begin{lemma}\label{Lemma_PerfectMatching_Decomposition_Few_Large_Colours}
Let $1\GG \varepsilon\GG n^{-1}$.
Let $K_{n,n}$ be properly coloured  with $\leq (1-20\varepsilon)n$ colours having $\geq (1-20\varepsilon)n$ edges. Then $K_{n,n}$  has $(1-\varepsilon)n$ edge-disjoint  perfect rainbow matchings.
\end{lemma}
\begin{proof}
Choose $1\GG \varepsilon\GG p \GG \gamma\GG n^{-1}$. By Lemma~\ref{Lemma_Typicality_K_n}, $K_{n,n}$ is $(\gamma, 1, n)$-typical.
Apply Lemma~\ref{Lemma_Random_Subgraph_Bipartite} (a) with $p=p$, $\delta=1$, and  $\gamma=\gamma$ in order to partition $K_{n,n}$ into a $(2\gamma, p, n)$-typical graph $J$ and a colour-disjoint graph $G$ with $\delta(G)\geq (1-2\gamma)(1-p)n \geq (1-\varepsilon^2)n$.
Apply Lemma~\ref{Lemma_regular_subgraph_few_large_colours_bipartite} to $G$ in order to find a subgraph $G'$ which is  globally $(1-\varepsilon)n$-bounded and $(\gamma, \delta, n)$-regular for some $\delta \geq 1-\varepsilon+9\varepsilon^2$.
Since $\varepsilon\GG p$ we have that $G'$  is  globally $(1-p)\delta n$-bounded.
By Lemma~\ref{Lemma_PerfectMatching_Decomposition_ExtraPseudorandom} applied with $\gamma'=2\gamma$ to $G'$ and $J$ there are edge-disjoint perfect rainbow matchings $M_1, \dots, M_{(1-p)\delta n}$ in $G'\cup J$. Since $\varepsilon\GG p$ and $\delta \geq 1-\varepsilon+9\varepsilon^2$, we have the required perfect rainbow matchings in $K_{n,n}$.
\end{proof}

As a corollary we show that having quadratically many colours guarantees perfect rainbow matchings.
\begin{lemma}\label{Lemma_many_colours_implies_few_large_colours}
Let $1\geq \varepsilon\GG n^{-1}$.
Let $K_{n,n}$ be coloured with at least $2\varepsilon n^2$ colours. Then $K_{n,n}$ has $\leq (1-\varepsilon)n$ colours having $\geq (1-\varepsilon)n$ edges.
\end{lemma}
\begin{proof}
Suppose otherwise. Then $K_{n,n}$ has $\leq n^2- (1-\varepsilon)n\cdot (1-\varepsilon)n=2\varepsilon n^2-\varepsilon^2 n^2$ edges outside of the $(1-\varepsilon)n$ largest colours. This means that $K_{n,n}$ has $\leq 2\varepsilon n^2-\varepsilon^2 n^2+(1-\varepsilon)n$ colours. By  $\varepsilon\GG n^{-1}$, this  is smaller that $2\varepsilon n^2$, contradicting the lemma's assumption.
\end{proof}

\begin{corollary}\label{Corollary_many_colours_perfect}
Let $1\GG \varepsilon\GG n^{-1}$.
Let $K_{n,n}$ be coloured with at least $\varepsilon n^2$ colours. Then $K_{n,n}$ has $(1-\varepsilon)n$ edge-disjoint perfect rainbow matchings.
\end{corollary}
\begin{proof}
Let $\varepsilon'=\varepsilon/40$.
By Lemma~\ref{Lemma_many_colours_implies_few_large_colours}, $K_{n,n}$  has $\leq (1-20\varepsilon')n$ colours having $\geq (1-20\varepsilon')n$ edges.
By Lemma~\ref{Lemma_PerfectMatching_Decomposition_Few_Large_Colours}, $K_{n,n}$ has $(1-\varepsilon')n\geq (1-\varepsilon)n$ edge-disjoint perfect rainbow  matchings.
\end{proof}

\subsection{2-Factors}
Here we use the perfect matching decomposition results from the previous section in order to show that suitable properly coloured complete graphs have near-decompositions into rainbow $2$-factors. These $2$-factor results are a stepping stone for finding Hamiltonian cycles.
The basic idea of the proof is to join rainbow matchings together into $2$-factors. Suppose that we have partitioned the vertices of a graph into sets $V_1, \dots, V_k$ of equal size, and that we have rainbow matchings $M_1, \dots, M_k$ with $M_i$ going from $V_i$ to $V_{i+1\pmod k}$. Notice that if the matchings $M_1, \dots, M_k$ are all colour-disjoint, then  $M_1\cup \dots\cup M_k$ is a rainbow $2$-factor. The proof strategy in this section is to partition the edges of a general graph $G$ into balanced bipartite subgraphs in which we can find perfect rainbow matchings using results from the previous section. Then we can put these matchings together in the way just described and obtain $2$-factors.

First we will need the following standard lemma which asserts that there exist complete graphs with rainbow Hamiltonian decompositions.
\begin{lemma}\label{Lemma_Existence_Hamiltonian_Decomposition}
For  prime $n\geq 3$, there exist properly $n$-coloured $K_n$ with decompositions into $(n-1)/2$ rainbow Hamiltonian cycles.
\end{lemma}
\begin{proof}
Identify the vertices of $K_n$ with $\mathbb{Z}/n\mathbb{Z}$.
Colour $ij$ by $i+j \pmod n$. For $i=1, \dots, (n-1)/2$, let $C_i=\{a(a+i): a=1, \dots, n\}$. Notice that $C_i$ is rainbow since we have $a(a+i), a'(a'+i)\in C_i \implies c(a(a+i))=2a+i,  c(a'(a'+i))=2a'+i$, and so distinct edges in $C_i$ have distinct colours.
Since $n$ is prime and $i\leq (n-1)/2$, $a+ki=a \pmod n \implies n\mid k$ which implies that $C_i$ is a cycle. The cycles $C_{1}, \dots, C_{(n-1)/2}$ are disjoint because for $i\neq j$, $a(a+i)=b(b+j)$ implies that  $b+j=a$ and $a+i=b$, which implies $i+j\equiv 0 \pmod n$.
\end{proof}

Using the above, and results about perfect matching decompositions, we can prove our first result about 2-factor decompositions. The following should be compared with Lemma~\ref{Lemma_PerfectMatching_Decomposition_ExtraPseudorandom}. It shows that under analogous assumptions to that lemma, one can find a near-decomposition into rainbow $2$-factors. It also has a divisibility condition on the size of the host graph. This divisibility condition will later be removed.
\begin{lemma}\label{Lemma_2_Factor_Decomposition_Divisibility}
Let $1\geq \delta, p\GG k^{-1}\GG \gamma \GG n^{-1}$, with  $k$  prime and $k\mid n$.
Let $G$ be a properly coloured, globally $(1-p)\delta n/2$-bounded,  $(\gamma, \delta, n)$-regular graph, $J$ a properly coloured $(\gamma, p, n)$-typical graph which is edge-disjoint and colour-disjoint  from $G$.
Then $G\cup J$ has $(1-2p)\delta n/2$ edge-disjoint rainbow $2$-factors with cycles of length $\geq k$.
\end{lemma}
\begin{proof}
Partition $V(K_n)$ randomly into $k$ sets $V_1, \dots, V_k$ of size $m=n/k$. Partition $C(G\cup J)$ randomly into $k$ sets $C_1, \dots, C_k$, with each colour ending up in each set independently with probability $1/k$.

By Lemma~\ref{Lemma_Random_Subgraph_General} (d), with probability $1-o(k^2n^{-1})$, the balanced bipartite graphs $G[V_i, V_j]$ are  $(2\gamma, \delta, k^{-1} n)$-regular and globally $(1+ \gamma)(1-p)\delta k^{-2} n$-bounded.
Also by Lemma~\ref{Lemma_Random_Subgraph_General} (d), with probability $1-o(kn^{-1})$, the graphs $J[V_i, V_j]$ are  $(2\gamma, p, k^{-1} n)$-typical.
Let $G_{a,b,c}=G_{C_{c}}[V_a, V_b]$ and $J_{a,b,c}=J_{C_{c}}[V_a, V_b]$. By Lemma~\ref{Lemma_Random_Subgraph_Bipartite} (a), with  probability $1-o(kn^{-1})$,  the balanced bipartite  graphs $G_{a,b,c}$ are $(4\gamma, k^{-1}\delta, k^{-1} n)$-regular, and the  graphs $J_{a,b,c}$ are $(4\gamma, k^{-1}p, k^{-1} n)$-typical. By a union bound, with high probability these hold for all the graphs $G_{a,b,c}, J_{a,b,c}$ simultaneously (using $n^{-1}\LL k^{-1}$).
Since $\gamma\LL p$, the graphs $G_{a,b,c}$ are globally $(1-p/2)(k^{-1}\delta)(k^{-1} n)$-bounded.

For any $a,b,c\in \{1, \dots, k\}$ with $a\neq b$, apply Lemma~\ref{Lemma_PerfectMatching_Decomposition_ExtraPseudorandom} to $G_{a,b,c}$ and $J_{a,b,c}$ with $n'=k^{-1}n$, $\delta'=k^{-1}\delta$, $p'=k^{-1}p$, $\gamma'=4\gamma$. This gives a family $\mathcal{M}_{a,b,c}=\{M^i_{a,b,c}:1\leq i\leq (1-p)k^{-2}\delta n\}$ of $(1-p)k^{-2}\delta n$ rainbow matchings  with every matching $M_{a,b,c}^i\in \mathcal{M}_{a,b,c}$ having $V(M_{a,b,c}^i)=V_a\cup V_b$ and $C(M_{a,b,c}^i)\subseteq C_c$.

Consider a proper $k$-colouring of $K_{k}$ with vertex set $\{1,\ldots,k\}$ and a decomposition into rainbow Hamiltonian cycles $H_1,$ $\dots,$ $H_{(k-1)/2}$ (which exists by Lemma~\ref{Lemma_Existence_Hamiltonian_Decomposition}).
For every $i\in [(k-1)/2]$,  $j\in  [(1-p)k^{-2}\delta n]$, and $t\in [k]$, let $F_{i,t}^j= \bigcup_{\substack{ab\in E(H_i)\\ c=c(ab)+t\pmod k}}M_{a,b,c}^j$.

\begin{claim}
For all $i, j, t$, $F_{i,t}^j$ is a rainbow $2$-factor with cycles of length $\geq k$.
\end{claim}
\begin{proof}
Without loss of generality, by reordering the vertex sets $V_1$ and colour sets $C_j$, we can suppose that $K_k$ is ordered and coloured so that the vertex sequence of $H_i$ is $1,2,\dots, k$ and so that $c(a(a+1))+t\mod k =a$. Then  $F_{i,t}^j= \bigcup_{a=1}^kM_{a,a+1,a}^j$.

Let $v\in V(F_{i,t}^j)$ with $v\in V_a$. We claim that $N_{F_{i,t}^j}(v)=\{x,y\}$  for some $x\in V_{a+1}, y\in V_{a-1}$.
To see this, notice that since $M_{a, a+1, a}^j$, $M_{a-1, a, a-1}^j$ are perfect matchings between $V_a$ and $V_{a+1}, V_{a-1}$ respectively, they each have one edge through $v$. Let these edges be $vx$ and $vy$ to get two vertices with $x,y\in N(v)$. To see that there are no  edges other than $vx, vy$ containing $v$, notice that $v\in V_a$ and none of the matchings forming $F_{i,t}^j$ other than $M_{a, a+1,  a}^j$ and $M_{a-1, a, a-1}^j$ touch $V_a$.
We have shown that $F_{i,t}^j$ is $2$-regular and so a $2$-factor.

To see that $F_{i,t}^j$ has no cycles shorter than $k$, consider a cycle $C$ with vertex sequence $v_1, v_2, \dots, v_s$.  Without loss of generality, we can suppose that $H_i$ is labelled so that $v_1\in V_1$ and $v_2\in V_2$. By the previous paragraph, this implies that $v_i\in V_{i \pmod k}$ must hold for all $i$, implying that $s\geq k$.

To see that $F_{i,t}^j$ is rainbow, notice first that it is the union of rainbow sets of edges $M_{a,a+1,a}^j$ for $a=1, \dots, k$. For any such matching we have $C(M_{a,a+1,a}^j)\subseteq C_{a}$.   Together with the colour-disjointness of $C_{a}$ and $C_{a'}$,  we get that  $M_{a,a+1,a}^j$ and $M_{a',a'+1,a'}^j$ are colour-disjoint for $a\neq a'$. Thus, $F_{i,t}^j$ is rainbow.
\end{proof}
Notice that $F_{i,t}^j$ and $F_{i',t'}^{j'}$ are edge-disjoint for $(i,j,t)\neq (i', j',t')$ (since any matching $M_{a,b,c}^d$ is contained in exactly one of the 2-factors $F_{i,t}^j$, and  matchings $M_{a,b,c}^d$, $M_{a',b',c'}^{d'}$  are edge-disjoint for $(a,b,c,d)\neq (a',b',c',d')$).
The total number of $2$-factors we have is $k\times  (1-p)k^{-2}\delta n\times (k-1)/2=(1-p)(1-k^{-1})\frac{\delta n}{2}\geq (1-2p)\frac{\delta n}{2}$.
\end{proof}

In the remainder of this section we prove that the above lemma is true even without the divisibility condition on $n$. The idea of the proof is to randomly partition the graph $G$ into subgraphs which do satisfy the divisibility condition. Applying Lemma~\ref{Lemma_2_Factor_Decomposition_Divisibility} to each of these subgraphs gives a decomposition of them into $2$-factors. By carefully putting  the $2$-factors together we get a decomposition of the whole graph into $2$-factors. First we need the following standard number-theoretic result.
\begin{lemma}\label{Lemma_Number_Decompose_Into_Primes}
Let $1\GG \varepsilon\GG s^{-1}\GG k^{-1} \GG n^{-1}$.
There exist prime numbers $k_1, k_2\in [k, (1+\varepsilon) k]$,  integers $s'= (1\pm \varepsilon)s$ and  $n_1, \dots, n_{s'}=(1\pm \varepsilon)n/s$ so that $n_1+\dots+n_{s'}=n$  and for each $i=1, \dots, s'$ either $k_1\mid n_i$ or $k_2\mid n_i$.
\end{lemma}
{
\begin{proof}
Since $1\GG \varepsilon \GG k^{-1}$, we can choose two distinct primes $k_1, k_2\in [k, (1+0.2\varepsilon) k]$. (When $\varepsilon$ is constant, this is a consequence of the Prime Number Theorem. More generally, we need the result of Hoheisel that there is some fixed number $\alpha>0$ such that for sufficiently large $n$, there is a prime in the interval $[n, (1+n^{-\alpha})n]$. (See \cite{baker1996difference}). Since $\varepsilon\GG n^{-1}$ implies $0.2\varepsilon \geq n^{-\alpha}$, we get that there is a prime in $[x, (1+0.2\varepsilon)x]$ for sufficiently large $x$).
As $\varepsilon\GG k_1^{-1},k_2^{-1},s^{-1} \GG n^{-1}$, there are integers {$z_1, z_2\geq 100n/\varepsilon s k$} with $k_1z_1+k_2z_2=n$. (See \cite[pg 25--26, Corollary~2]{sierpinski1988elementary}. Applying this corollary with $a=k_1, b=k_2$, $n'=n-(k_1+k_2)\lceil 100n/\varepsilon s k\rceil$ gives non-negative integers $x,y$ with $k_1 x+k_2 y= n-(k_1+k_2)\lceil 100n/\varepsilon s k\rceil$. Letting $z_1=x+\lceil 100n/\varepsilon s k\rceil$, $z_2=y+\lceil 100n/\varepsilon s k\rceil$ gives the numbers we want.)

For some appropriate $s''\geq 50\varepsilon^{-1}$, pick integers $m_1,\ldots,m_{s''}$ so that $m_i=(1\pm 0.5\varepsilon)n/sk_1$ and $\sum_{i=1}^{s''}m_i=z_1$. This is possible as {$z_1\geq 100n/\varepsilon s k$}.
Similarly, for some appropriate $s'$, pick integers $m_{s''+1},\ldots,m_{s'}$ so that $m_i=(1\pm 0.5\varepsilon)n/sk_2$ and $\sum_{i=s''+1}^{s'}m_i=z_2$.

For each $1\leq i\leq s''$, let $n_i=k_1m_i$, and for each $s''<i\leq s'$, let $n_i=k_2m_i$. Then, for each $1\leq i\leq s'$, $n_i=(1\pm 0.5\varepsilon)n/s$. Thus, as $n_1+\ldots+n_{s'}=n$, we have $s'=(1 \pm \varepsilon)s$. The numbers $k_1$, $k_2$, $s'$, and $n_1,\ldots,n_{s'}$ then satisfy the conditions of the lemma.
\end{proof}}

The following lemma shows how any large set can be evenly covered by subsets whose sizes satisfy the divisibility condition of Lemma~\ref{Lemma_2_Factor_Decomposition_Divisibility}.
\begin{lemma}\label{Lemma_Near_Design_Existence}
Let $1\GG  \varepsilon\GG \hat s^{-1} \GG k^{-1}\GG \gamma  \GG n^{-1}$.
There exists a family $\mathcal H$ of {partitions} of $[n]$ and a number  $s=(1\pm \varepsilon)\hat s$ with the following properties.
\begin{enumerate}[(i)]
\item $\mathcal H=\mathcal M_1 \cup \dots \cup \mathcal M_{s^2 \log^2 n}$ with each $\mathcal M_i=\{M_i^1, \dots, M_i^s\}$ for disjoint sets $M_i^1, \dots, M_i^s$ satisfying $|M_i^j|=(1\pm \varepsilon) n/s$ and $\bigcup_{j=1}^s M_i^s=[n]$.
\item For each $i,j$ there is some  prime number $k_i^j\in[k, (1+\varepsilon)k]$ with $k_i^j\mid |M_i^j|$.
\item
For each distinct pair  $x,y\in [n]$, there are $(1\pm 5\varepsilon)s\log^2 n$ sets $M_i^j$ containing both $x$ and $y$.
\end{enumerate}
\end{lemma}
\begin{proof}
Apply Lemma~\ref{Lemma_Number_Decompose_Into_Primes}  in order to find a number $s=(1\pm\varepsilon )\hat s$, primes $k_1, k_2= (1\pm \varepsilon)k$ and numbers $n_1, \dots, n_{s}=(1\pm \varepsilon)n/s$ so that $n_1+\dots+n_{s}=n$  and, for each $i=1, \dots, s$, either $k_1\mid n_i$ or $k_2\mid n_i$.
Let $M^1,\dots, M^{s}$ be disjoint subsets of $[n]$ with $|M^i|= n_i$.

Choose $s^2\log^2 n$ permutations $\sigma_1, \dots, \sigma_{s^2 \log^2 n}$ of $[n]$ uniformly at random. Let $\mathcal M_i=\{\sigma_i(M^1), \dots, \sigma_i(M^s)\}$ and $\mathcal H=\mathcal M_1 \cup \dots \cup \mathcal M_{s^2\log^2 n}$.  Notice that as a consequence of  the properties from Lemma~\ref{Lemma_Number_Decompose_Into_Primes}, all the conditions of the lemma hold for $\mathcal H$ aside from (iii).
We will show that this condition holds with high probability.

Let $x,y\in [n]$ be distinct vertices. We have $\P(x,y \in \mathcal M_i)= \sum_{j=1}^s\binom{n_j}{2}\big/\binom n2=(1\pm 4\varepsilon)s^{-1}$. Let $X$ be the number of families $\mathcal M_i$ which contain $x,y$. We have that $X$ is bounded above and below by random variables with distributions $\mathrm{Binomial}(s^2\log^2 n, (1+4\varepsilon)s^{-1})$ and $\mathrm{Binomial}(s^2 \log^2 n, (1-4\varepsilon)s^{-1})$  respectively.
From Chernoff's Bound we have
$\mathbb P\big(|X- s\log^2 n|> 5\varepsilon s\log^2 n\big)\leq 4e^{-\frac{\varepsilon^2 s \log^2 n}{100}}= o(n^{-2})$ (using $\varepsilon\GG \hat s^{-1}$). By the union bound taken over all pairs $x,y$, we have that with high probability all  pairs  $x,y\in [n]$, have $(1\pm 5\varepsilon)s\log^2 n$ families $\mathcal M_i$ containing both $x$ and $y$.
\end{proof}

By combining the lemmas of this section we can prove Lemma~\ref{Lemma_2_Factor_Decomposition_Divisibility} without the divisibility assumption.
\begin{lemma}\label{Lemma_2_Factor_Decomposition}
Let $1\geq \delta, p,  \log^{-1}n  \GG k^{-1}\GG\gamma \GG n^{-1}$.
Let $G$ be a properly coloured, globally $(1-p)\delta n/2$-bounded,  $(\gamma, \delta, n)$-regular graph and let $J$ be a properly coloured $(\gamma, p, n)$-typical graph which is edge-disjoint and colour-disjoint  from $G$ but has the same vertex set.
Then $G\cup J$ has $(1-p)\delta n/2$ edge-disjoint rainbow $2$-factors with cycles of length $\geq k$.
\end{lemma}
\begin{proof}
Choose $1\geq \delta, p,\log^{-1}n\GG \varepsilon \GG \hat s^{-1}\GG k^{-1}\GG\gamma \GG n^{-1}$.

Apply Lemma~\ref{Lemma_Near_Design_Existence} to find a family $\mathcal {H}$ of partitions of $V(G)$ and a number $s=(1\pm \varepsilon)\hat s$ {so that the properties in that lemma hold (with the associated notation)}. Let $\sigma$ be a random permutation of $V(G)$.
Lemma~\ref{Lemma_Random_Subgraph_General} (c) implies that for each $i,j$ with probability $1-o(n^{-1})$ we have that ${G}[\sigma(M_{i}^j)]$ is $(2\gamma, \delta, |M_{i}^j|)$-regular and globally $(1+\gamma) (1-p)\frac{\delta |M_i^j|^2}{2n}$-bounded, while $J[\sigma(M_{i}^j)]$ is $(2\gamma, p, |M_{i}^j|)$-typical.

Partition the colours independently at random into sets $C_1, \dots, C_s$. Let ${G}^{C_j}$ and $J^{C_j}$ denote the subgraphs of ${G}$ and $J$ respectively consisting of colour $C_j$ edges.
By Lemma~\ref{Lemma_Random_Subgraph_General} (a), for all $i, t$, with  probability $1-o(sn^{-1})$ the graph ${G}^{C_t}[\sigma(M_{i}^j)]$ is $(4\gamma, \delta/s, |M_{i}^j|)$-regular and globally $(1+\gamma) (1-p)\frac{\delta |M_{i}^j|^2}{2n}$-bounded, while $J^{C_t}[\sigma(M_{i}^j)]$ is $(4\gamma, p/s, |M_{i}^j|)$-typical.

For any edge $e$, let $d_{\mathcal H}(e)$ be the number of partitions in $\mathcal{H}$ which contain some set containing $e$. From the properties from Lemma~\ref{Lemma_Near_Design_Existence}, we have $d_{\mathcal H}(e)= (1\pm 5\varepsilon)s\log^2 n$. For each edge $e$, choose an arbitrary injection $f_e:[d_{\mathcal H}(e)]\to  [ s^2\log^2 n]$ for which $e$ is containing in some set in $\mathcal M_{f_e(m)}$ for all $m\in [d_{\mathcal H}(e)]$.
For each edge $e$, choose a number $m_e$ out of  $1,\dots, (1+ 5\varepsilon)s\log^2 n$ at random. For $i=1, \dots, s^2\log^2 n$, let ${G}_i$ and $J_i$ be subgraphs of ${G}$ and $J$ respectively consisting of edges $e$ with $f_e(m_e)=i$ (here it is possible that $m_e>d_{\mathcal H}(e)$, in which case $f_e(m_e)$ is undefined. When this happens, the edge $e$ is placed in neither of the graphs ${G}_i, J_i$). Notice that edges of $G$ are placed into ${G}_i$ and  $J_i$ independently with probability $1/(1+ 5\varepsilon)s\log^2 n$.

For $i=1, \dots, s^2\log^2 n$  and $j,t=1, \dots, s$, define ${G}_{i,j,t}$ and $J_{i,j,t}$ to be the subgraphs of ${G}$ and $J$ with vertex set $\sigma(M_i^j)$ consisting of edges $xy$ which are simultaneously contained   in $E({G}_i)$, contained in $\sigma(M_i^j)$, and whose colours are in $C_{j+t\pmod s}$.
Notice that  ${G}_{i,j,t}$ and $J_{i,j,t}$ are formed from ${G}^{C_t}[\sigma(M_{i}^j)]$ and $J^{C_t}[\sigma(M_{i}^j)]$ by choosing every edge with probability $1/(1+ 5\varepsilon)s\log^2 n$.
By Lemma~\ref{Lemma_Random_Subgraph_General} (b) with  probability $1-o(n^{-1})$, the graph ${G}_{i,j,t}$ is $(8\gamma, \frac{\delta}{(1+ 5\varepsilon)s^2\log^2n}, |M_i^j|)$-regular and globally $(1+\gamma)^2 (1-p)\frac{\delta |M_i^j|^2}{(1+ 5\varepsilon)2n s\log^2 n}$-bounded, while $J_{i,j,t}$ is $(8\gamma, \frac{p}{(1+ 5\varepsilon)s^2\log^2n}, |M_i^j|)$-typical. Using $|M_i^j|=(1\pm \varepsilon)\frac ns$ and $p\GG\varepsilon, \gamma$, we have that ${G}_{i,j,t}$ is globally $(1-p/4)\frac{\delta |M_i^j|}{(1+ 5\varepsilon)2 s^2\log^2 n}$-bounded. Notice that $100s^4 \log^2 n\leq n$, and so by a union bound the choices of $\sigma, C_1, \dots,  C_s, {G}_i, J_i$ can be done so that for all $i,j,t$, the graphs ${G}_{i,j,t}$ and $J_{i,j,t}$ have all these properties simultaneously.

Apply Lemma~\ref{Lemma_2_Factor_Decomposition_Divisibility} with $\delta'=\frac{\delta}{(1+ 5\varepsilon)s^2\log^2n}$, $p'=\frac{p}{(1+ 5\varepsilon)s^2\log^2n}$, $n'=|M_i^j|$, $\gamma'=12\gamma$, $k'=k_i^j$ in order to find a family $\mathcal F_{i,j,t}$ of $(1-p/2)\frac{\delta |M_i^j|}{(1+ 5\varepsilon)2 s^2 \log^2n}$ of edge-disjoint rainbow $2$-factors with cycles of length $\geq k$ in ${G}_{i,j,t}\cup J_{i,j,t}$ (for this application we are using  $1\geq \delta, p, \log^{-1}n \GG  s^{-1}\GG k^{-1} \GG  \gamma\GG n^{-1}$ to conclude that  $1\geq \delta',p'\GG  k'^{-1}\GG \gamma' \GG n'^{-1}$. The divisibility condition in Lemma~\ref{Lemma_2_Factor_Decomposition_Divisibility} comes from the property of $k_i^j$ in Lemma~\ref{Lemma_Near_Design_Existence}).
Since $|M_i^j|=(1\pm \varepsilon)\frac ns$ and $p\GG \varepsilon$, we can choose a subfamily $\mathcal F'_{i,j,t}$ of size $(1-p)\frac{\delta n}{2 s^3 \log^2n}$.
Let $\mathcal F_{i,t}=\bigcup_{j=1}^s \mathcal F'_{i,j,t}$ to get a family of $(1-p)\frac{\delta n}{2 s^3 \log^2n}$  edge-disjoint rainbow $2$-factors in $G\cup J$ with cycles of length $\geq k$. To see that these are rainbow $2$-factors notice that for $j\neq j'$ the graphs ${G}_{i,j,t}\cup J_{i,j,t}$ and ${G}_{i,j',t}\cup J_{i,j',t}$  are vertex-disjoint (their vertex sets are $\sigma(M_i^j)$ and $\sigma(M_i^{j'})$ respectively) and colour-disjoint (their colours are contained in $C_{j+t\pmod s}$ and $C_{j'+t\pmod s}$ respectively).
Since the $2$-factors in $\{\mathcal F_{i,t}: 1\leq i\leq s^2\log^2 n, 1\leq t\leq s\}$ are all edge-disjoint, we have a total of $(1-p)\delta n/2$ edge-disjoint rainbow 2-factors as required.
\end{proof}

By combining the above with a regularization lemma we can find $2$-factor decompositions in nearly-complete graphs which have few large colours.
\begin{lemma}\label{Lemma_2_Factor_Decomposition_No_Gap}
Let $1\geq \varepsilon, \log^{-1} n\GG  k^{-1}\GG\gamma\GG n^{-1}$.
Let $H$ be a properly coloured, $(\gamma, 1-\varepsilon^2, n)$-typical graph with $\leq (1-60\varepsilon)n$ colours having $\geq (1-60\varepsilon)n/2$ edges. Then $H$  has $(1-3\varepsilon)n/2$ edge-disjoint rainbow spanning $2$-factors with cycles of length $\geq k$.
\end{lemma}
\begin{proof}
Choose $1\geq \varepsilon, \log^{-1}n\GG p \GG k^{-1}\GG\gamma\GG n^{-1}$.
Apply Lemma~\ref{Lemma_Random_Subgraph_General} (a) with $p=p/(1-\varepsilon^2)$, $\delta=(1-\varepsilon^2)$, and  $\gamma=\gamma$ in order to partition $H$ into a $(2\gamma, p,n)$-typical graph $J$ and a colour-disjoint graph $G$ with $\delta(G)\geq \delta(H)-\Delta(J)\geq  (1-3\varepsilon^2)n$.

Apply Lemma~\ref{Lemma_regular_subgraph_few_large_colours} to $G$ with $\varepsilon'=3\varepsilon$, $\gamma=\gamma$ in order to find a subgraph $G'$ which is globally $(1-3\varepsilon)n/2$-bounded and $(\gamma,\delta ,n)$-regular for some $\delta \geq  1-3\varepsilon+81\varepsilon^2$.
Notice that $G'$ is globally $(1-p)\delta n/2$-bounded (since $p\LL \varepsilon$).

By Lemma~\ref{Lemma_2_Factor_Decomposition} applied to $G'$ and $J$ with $\gamma'=2\gamma$, $\delta=\delta$, $p=p$, $k=k$ there are $(1-p)\delta n/2\geq (1-3\varepsilon)n/2$ edge-disjoint rainbow $2$-factors with cycles of length $\geq k$.
\end{proof}

\subsection{Hamiltonian cycles}
Here we  take the $2$-factor decompositions from the previous section and modify them into Hamiltonian decompositions. The proofs and results in this section are very similar to the ones in Section~\ref{Section_Perfect_Matchings} where we took nearly-perfect matchings and modified them into perfect matchings.
The following is the first result we prove about turning a family of $2$-factors into a family of Hamiltonian cycles. It parallels Lemma~\ref{Lemma_decomposition_completion} for turning nearly-perfect matchings into perfect matchings.

As in Lemmas~\ref{Lemma_rotation_small_cycle} and~\ref{Lemma_completion_Hamiltonian} the following lemma has a mix of directed and undirected graphs. As before in the cycles we build, we do not care about the directions of their edges.
\begin{lemma}\label{Lemma_decomposition_completion_Hamiltonian}
Let $1\GG \delta\GG p\GG  \theta \GG k^{-1}\GG n^{-1}$ and $t\leq n$.
Suppose that we have the following edge-disjoint, properly coloured graphs on a set of $n$ vertices.
\begin{itemize}
\item $F_1, \dots, F_t$  rainbow  $2$-factors with cycles of length $\geq k$.
\item  $(3p(\theta n)^2,  \theta n)$-dense graphs $E_1, E_2, E_3$.
\item Digraphs $D_X, D_Y$  with $\delta^+(D_X), \delta^+(D_Y)\geq 3\delta n$.
\end{itemize}
Additionally suppose that $F_1\cup\dots\cup F_t$, $E_1, E_2, E_3$, $D_X$, and $D_Y$ are colour-disjoint.
Then there are edge-disjoint rainbow Hamiltonian cycles $C_1, \dots, C_t$ in $E_1\cup E_2\cup E_3\cup D_X\cup D_Y\cup F_1\cup \dots \cup F_t$.
\end{lemma}
\begin{proof}
For $i=1, \dots, t$ let $m_i$ be the number of cycles in $F_i$, and note that $m_i\leq k^{-1}n\leq p\theta n$.

\begin{claim}
There are matchings $M_1, \dots, M_t$ with $\Delta(M_1\cup \dots \cup M_t)\leq 4n/k$ such that $M_i\subseteq F_i$ is a matching of size $m_i$ containing exactly one edge from each cycle of $F_i$.
\end{claim}
\begin{proof}
  Choose each matching $M_i$  uniformly at random from all matchings containing exactly one edge from each cycle of $F_i$.
If $t\leq 4n/k$ then we trivially have $\Delta(M_1\cup \dots \cup M_t)\leq 4n/k$.
  Otherwise, notice that for any vertex $v$, its degree in  $M_1\cup \dots \cup M_t$  is stochastically dominated by $\mathrm{Binomial}(t, 2k^{-1})$. By Chernoff's Bound with $\varepsilon=1/2$, the union bound, and $t\leq n$ we have $\P(\Delta(M_1\cup \dots \cup M_t)> 4t/k)\leq 4ne^{-2t/12k}<1)$ (using $n^{-1}\LL k^{-1}$ and $4n/k\leq t$). A choice of matchings  satisfying the claim thus exists.
\end{proof}
Let $M_i=\{x_i^1y_i^1, \dots, x_i^{m_i}y_i^{m_i}\}$.
We construct the Hamiltonian cycles $C_1, \dots, C_t$ one-by-one using Lemma~\ref{Lemma_completion_Hamiltonian}.  They will have the following properties.
\begin{enumerate}[(i)]
\item $xy\in D_X\cap C_i \implies x=x_i^j$ for some $j\in \{1, \dots, m_i\}$.
\item $yx\in D_Y\cap C_i \implies y=y_i^j$ for some $j\in \{1, \dots, m_i\}$.
\item $e(E_i \cap C_i)\leq k^{-1} n$ for $i=1,2,3$.
\end{enumerate}
Suppose that we have constructed Hamiltonian cycles $C_1, \dots, C_s$ satisfying the above properties. Let $E_i^s=E_i\setminus (C_1\cup \dots\cup C_s)$ for $i=1,2,3$, $D^s_X=D_X\setminus (C_1\cup \dots\cup C_s)$, $D^s_Y=D_Y\setminus (C_1\cup \dots\cup C_s)$. Using (i), $\delta^+(D_X)\geq 3\delta n$, and  $\Delta(M_1\cup \dots \cup M_t)\leq 4n/k\leq \delta n$,  we have $d^+_{D^s_X}(x)\geq d^+_{D_X}(x)- d^+_{D_X\cap (C_1\cup \dots \cup C_s)}(x)\geq\delta^+(D_X)- 4n/k\geq \delta n + nk^{-1}\geq  \delta n + m_{s+1}$ for any $x\in X$. Similarly  $d^+_{D^s_Y}(y)\geq \delta n+m_{s+1}$ for $y\in Y$. Using Lemma~\ref{Lemma_Expander_Delete_Edges} and (iii), $E_1^s, E_2^s, E_3^s$ are $(3p(\theta n)^2-s(k^{-1}n), \theta n)$-dense. Since $k^{-1} ns,  m_{s+1}\theta n\leq  p(\theta n)^2$, they are also  $(p(\theta n)^2+m_{s+1}\theta n, \theta n)$-dense. By Lemma~\ref{Lemma_completion_Hamiltonian} applied to $F_{s+1}$, $E^s_1, E^s_2, E^s_3, D^s_X, D^s_Y$  with $\theta=\theta$, $p=p$, $m=m_{s+1}$, $\delta=\delta$, $(x_j, y_j)=(x_i^j, y_i^j)$ there is a rainbow Hamiltonian cycle $C_{s+1}$ satisfying (i) -- (iii).
\end{proof}

We will need the following easy lemma.
\begin{lemma}\label{Lemma_High_Degree_Orientation}
Let $\delta\GG n^{-1}$.
Every graph $G$ with $\delta(G)\geq \delta n$ has an orientation $D$ such that $\delta^+(D)\geq \delta n/3$.
\end{lemma}
\begin{proof}
Orient the graph at random. Notice that for any vertex $d^+(v)\sim\mathrm{Binomial}(0.5, d(v))$. By Chernoff's Bound we have $\mathbb P(d^+(v)<d(v)/3)\leq
2e^{-{d(v)}/{54}}\leq 2e^{-{\delta n}/{54}} = o(n^{-1})$ (using $\delta\GG n^{-1}$). Taking a union bound over all vertices shows that some suitable orientation exists.
\end{proof}

The following  version of Lemma~\ref{Lemma_decomposition_completion_Hamiltonian} will be easier to apply.
\begin{lemma}\label{Lemma_decomposition_completion_Hamiltonian_Codegrees}
Let $1\GG p\GG\gamma,  k^{-1}\GG n^{-1}.$
For $t\leq n$, let $F_1, \dots, F_t$ be  edge-disjoint rainbow  $2$-factors with cycles of length $\geq k$.
Let $G$ be an edge-disjoint, colour-disjoint  $(\gamma, p, n)$-typical graph.
Then  there are edge-disjoint  rainbow Hamiltonian cycles $C_1, \dots, C_t$ in $F_1\cup \dots\cup F_t\cup G$.
\end{lemma}
\begin{proof}
Choose $1\GG p\GG  \delta \GG p_1 \GG \theta  \GG \gamma, k^{-1}\GG n^{-1}.$
Choose   five disjoint sets of colours $C_{E_1},$ $C_{E_2},$ $C_{E_3},$  $C_{D_X},$ $C_{D_Y}$ from $G$, with each colour put independently into  $C_{E_1}, C_{E_2}, C_{E_3},   C_{D_X}, C_{D_Y}$  with probabilites $4p_1p^{-1}$, $4p_1p^{-1}$, $4p_1p^{-1}$,  $10\delta p^{-1}$, $10\delta p^{-1}$ respectively (this is possible since $p\GG  \delta \GG p_1$ implies $4p_1p^{-1}+4p_1p^{-1}+4p_1p^{-1}+10\delta p^{-1} + 10\delta p^{-1}\leq 1$). Let $E_1, E_2, E_3$, $D_X$, $D_Y$ be the subgraphs of $G$ with colours from  $C_{E_1}, C_{E_2}, C_{E_3},  C_{D_X}, C_{D_Y}$ respectively.
By Lemma~\ref{Lemma_Random_Subgraph_General} (a), with positive probability   $E_1, E_2, E_3$ is  $(2\gamma, 4p_1, n)$-typical and $D_X$, $D_Y$ are $(2\gamma, 10\delta, n)$-typical.
By Lemma~\ref{Lemma_High_Degree_Orientation} and  $(2\gamma, 10\delta, n)$-typicality, $D_X$ and  $D_Y$ can be oriented so that $\delta^+(D_X)$, $\delta^+(D_Y)\geq 3\delta n$.

By Lemma~\ref{Lemma_Codegrees_Pseudorandom} applied with $\mu=\theta$, $\gamma'=\gamma/2$, $E_1, E_2, E_3$ are $( 3p_1(\theta  n)^2, \theta  n)$-dense. By Lemma~\ref{Lemma_decomposition_completion_Hamiltonian} applied with  $\theta=\theta,$ $p=p_1$,  we obtain the required Hamiltonian cycles.
\end{proof}

The following lemma should be compared with Lemmas~\ref{Lemma_PerfectMatching_Decomposition_ExtraPseudorandom} and~\ref{Lemma_2_Factor_Decomposition}. It produces a near-decomposition into Hamiltonian cycles under a similar assumption to those lemmas.
\begin{lemma}\label{Lemma_Hamiltonian_Decomposition_gap}
Let $1\geq \delta, p,  \log^{-1}n\GG  \gamma \GG n^{-1}$.
Let $G$ be a properly coloured $(\gamma, \delta, n)$-typical graph which is globally $(1-p)\delta n/2$-bounded. Then $G$  has $(1-p)\delta n/2$ edge-disjoint rainbow Hamiltonian cycles.
\end{lemma}
\begin{proof}
Choose $1\GG \delta, p, \log^{-1}n\GG p_1 \GG k^{-1}\GG\gamma\GG n^{-1}$.
Apply Lemma~\ref{Lemma_Random_Subgraph_General} (a) with $p'=p_1$, $\delta=\delta$, and  $\gamma=\gamma$ in order to partition $H$ into three colour-disjoint graphs $G'$, $J_1$ and $J_2$ so that $J_1$ and $J_2$ are $(2\gamma, p_1\delta, n)$-typical and $G'$ is $(2\gamma, (1-2p_1)\delta, n)$-typical. Setting $\delta_1=(1-2p_1)\delta$ and using  $p\GG p_1$, we have that $G'$ is $(2\gamma, \delta_1, n)$-typical and globally $(1-p_1)\delta_1 n/2$-bounded.

 Apply Lemma~\ref{Lemma_2_Factor_Decomposition} to $G'$ and $J_1$ with $\gamma'=2\gamma, \delta'=\delta_1$, $p'=p_1$, $k=k$ in order to find $(1-p_1)\delta_1 n/2$ edge-disjoint rainbow spanning $2$-factors $F_1, \dots, F_{(1-p_1)\delta_1n/2}$ in $G'$ whose cycles have length $\geq k$.

Apply Lemma~\ref{Lemma_decomposition_completion_Hamiltonian_Codegrees} to $F_1, \dots, F_{(1- p_1)n/2}$ and $J_2$ with $p=\delta p_1$, $\gamma'=2\gamma$, $k=k$, $t=(1-p_1)\delta_1 n/2$ in order to find $(1- p_1)\delta_1 n/2$ edge-disjoint rainbow Hamiltonian cycles  in $G$. Since $p\GG p_1$ implies $(1- p_1)\delta_1 n/2\geq (1- p)\delta n/2$,  we have enough cycles for the lemma.
\end{proof}

The following lemma should be compared with Lemma~\ref{Lemma_2_Factor_Decomposition_No_Gap}. Under similar assumptions, it produces a near-decomposition into Hamiltonian cycles rather than $2$-factors.
\begin{lemma}\label{Lemma_Hamiltonian_Decomposition_no_gap}
Let $1\GG \varepsilon, \log^{-1}n\GG  \gamma \GG n^{-1}$.
Let $G$ be a properly coloured, $(\gamma, 1-\varepsilon^2, n)$-typical graph with $\leq (1-180\varepsilon)n$ colours having $\geq (1-180\varepsilon)n/2$ edges. Then $G$  has $(1-6\varepsilon)n/2$ edge-disjoint rainbow Hamiltonian cycles.
\end{lemma}
\begin{proof}
Choose $1\GG \varepsilon, \log^{-1}n\GG p \GG k^{-1}\GG\gamma\GG n^{-1}$.
Apply Lemma~\ref{Lemma_Random_Subgraph_General} (a) with $p'=p/(1-\varepsilon^2)$, $\delta=(1-\varepsilon^2)$, and  $\gamma=\gamma$ in order to partition $H$ into a $(2\gamma, p, n)$-typical graph $J$ and a colour-disjoint $(2\gamma, 1-p-\varepsilon^2, n)$-typical graph $G'$.

Let $\varepsilon'=\sqrt{p+\varepsilon^2}$, and notice that  $\varepsilon\GG p$ implies $\varepsilon\leq \varepsilon'\leq 2\varepsilon$.
Hence $G'$ has  $\leq (1-60\varepsilon')n$ colours having $\geq (1-60\varepsilon')n/2$ edges.
Apply Lemma~\ref{Lemma_2_Factor_Decomposition_No_Gap} to $G'$ with $\gamma'=2\gamma, \varepsilon'=\varepsilon'$, $k=k$ in order to find $(1-3\varepsilon')n/2$ edge-disjoint rainbow spanning  $2$-factors $F_1, \dots, F_{(1-3\varepsilon')n/2}$ in $G'$ whose cycles have length $\geq k$.

Apply Lemma~\ref{Lemma_decomposition_completion_Hamiltonian_Codegrees} to $F_1, \dots, F_{(1-3\varepsilon')n/2}$ and $J$ with $p=p$, $\gamma'=2\gamma$, $k=k$, $t=(1-3\varepsilon')n/2$ in order to find $(1-3\varepsilon')n/2$ edge-disjoint rainbow Hamiltonian cycles  in $G$.
Since $(1-3\varepsilon')n/2\geq (1-6\varepsilon)n/2$, we have enough cycles.
\end{proof}

We can show that when a properly coloured $K_n$ has few large colours, then it has a near-decomposition into Hamiltonian cycles. This is ``half'' of our proof of the asymptotic version of the Brualdi-Hollingsworth and Kaneko-Kano-Suzuki Conjectures. The other half will be in the case when there are many large colours, which is performed in Section~\ref{Section_Trees}.
\begin{lemma}\label{Lemma_Hamiltonian_Decomposition_Kn}
Let $1\GG \varepsilon \GG  n^{-1}$.
Let $K_n$ be properly coloured with $\leq (1-\varepsilon)n$ colours having $\geq (1-\varepsilon)n/2$ edges. Then $K_n$  has $(1-\varepsilon)n/2$ edge-disjoint rainbow Hamiltonian cycles.
\end{lemma}
\begin{proof}
Choose $1\GG \varepsilon, \log^{-1}n\GG \varepsilon_1 \GG \gamma \GG n^{-1}$. Let $G$ be an arbitrary $(\gamma, 1-\varepsilon_1^2, n)$-typical subgraph of $K_n$ (it exists e.g.\ by Lemmas~\ref{Lemma_Typicality_K_n} and~\ref{Lemma_Random_Subgraph_General} (a) or (b)). Notice that since $\varepsilon\GG \varepsilon_1$, $G$ has $\leq (1-180\varepsilon_1)n$ colours having $\geq (1-180\varepsilon_1)n/2$ edges. By Lemma~\ref{Lemma_Hamiltonian_Decomposition_no_gap}, $G$ has $(1-6\varepsilon_1)n/2\geq (1-\varepsilon)n/2$ edge-disjoint rainbow Hamiltonian cycles.
\end{proof}

\section{Rainbow Trees}\label{Section_Trees}
In this section we show that the Brualdi-Hollingsworth and Kaneko-Kano-Suzuki Conjectures hold asymptotically. Part of this result was already proved in Lemma~\ref{Lemma_Hamiltonian_Decomposition_Kn}, which shows that the asymptotic versions of the Brualdi-Hollingsworth and Kaneko-Kano-Suzuki Conjectures hold in colourings of $K_n$ which have few large colours. In this section we focus on colourings of $K_n$ which have many large colours. Such colourings should be thought of as being close to $1$-factorizations.

The basic idea of the proof is to notice that for \emph{any} properly coloured $K_n$, we know how to find a large set of vertices $S$ so that the induced subgraph $K_n[S]$ has a near-decomposition into Hamiltonian paths. Indeed, a random set $S$ will have this property (by combining Lemmas~\ref{Lemma_Random_Subgraph_General} (c) and~\ref{Lemma_Hamiltonian_Decomposition_gap}). To find a near-decomposition into spanning trees we modify the paths in $K_n[S]$ by extending them one vertex at a time to cover all of $V(K_n)$.

\subsection{Small rainbow trees}
Here we prove a result about near-decompositions of globally bounded graphs into rainbow forests which are sufficiently small.
We remark that this lemma is only needed to deal with properly coloured complete graphs which are not $1$-factorizations---if one only wants to prove an asymptotic version of the Brualdi-Hollingsworth Conjecture, then this section can be omitted.

The result we prove in this section is essentially the following: for $m\GG k$ every properly coloured globally $m$-bounded graph with $\geq (1+o(1))mk$ edges has a near-decomposition into $m$ rainbow $k$-edge forests $F_1, \dots, F_m$. This is relatively straightforward (see Lemma~\ref{smalltree0}), however, we need to find such a near-decomposition that interacts well with a large \emph{vertex cover}. Here, a  \emph{vertex cover} $S$ is a set of vertices which contains at least one vertex in each edge. We develop Lemma~\ref{smalltree0} through Lemma~\ref{smalltree1} to arrive at the result we need, Lemma~\ref{Lemma_Small_Disjoint_Trees}.

\begin{lemma}\label{smalltree0}
Let $1\GG \beta\GG k/n$, $m\geq \beta n$ and $0\leq \ell<m$.
Let $G$ be a properly coloured, globally $m$-bounded, $n$-vertex graph, with $e(G)\geq (1+\beta)(k m+\ell)$.
 Then, $G$ has $m$ edge-disjoint rainbow forests $F_1, \dots, F_m$, so that each $F_i$ has $k+\mathbf{1}_{\{i\leq \ell\}}$ edges.
\end{lemma}
\begin{proof}  Note that if $k=0$ then selecting $\ell$ edges gives the required forests. Assume then that $k\geq 1$. By deleting edges if necessary, assume that $e(G)=(1+\beta)(k m+\ell)$.

Choose integers $d_c$, $c\in C(G)$, such that $\lfloor \frac{|E_G(c)|}{1+\beta}\rfloor\leq d_c\leq \lceil \frac{|E_G(c)|}{1+\beta}\rceil$ and $\sum_{c\in C(G)}d_c=km+\ell$, where we have used that
\[
\sum_{c\in C(G)}\left\lfloor \frac{|E_G(c)|}{1+\beta}\right\rfloor\leq
\sum_{c\in C(G)}\frac{|E_G(c)|}{1+\beta}
=\frac{e(G)}{1+\beta}=km+\ell\leq \sum_{c\in C(G)}\left\lceil \frac{|E_G(c)|}{1+\beta}\right\rceil.
\]
Let $C_1,\ldots,C_m$ be sets in $C(G)$, so that each $C_i$ has size $k+\mathbf{1}_{\{i\leq \ell\}}$, and each colour $c$ appears in $d_c$ sets $C_i$. Note that this is possible as $d_c\leq m$ for each $c\in C(G)$ and $\sum_{i=1}^m(k+\mathbf{1}_{\{i\leq \ell\}})=km+\ell$.
Let $F'_1,\ldots,F'_m$ be edge-disjoint rainbow forests in $G$ with $C(F'_i)\subset C_i$ for each $i$ and so that $|\cup_{i=1}^mE(F'_i)|$ is maximised. Suppose that $|\cup_{i=1}^mE(F'_i)|<km+\ell$, for otherwise $F_1',\ldots,F'_m$ satisfy the lemma.

\begin{claim}\label{sunny}
For each colour $c$, $|E_G(c)\setminus (\cup_{i=1}^mE(F'_i))|\leq \frac{\beta m}{1+\beta}+2k$.
\end{claim}
\begin{proof}
Fixing a colour $c$, let $M$ be the edges with colour $c$ not in $\cup_{i=1}^mE(F'_i)$, and suppose  that $|M|\geq |E_G(c)|-d_c+ k+1$. As $\sum_{i=1}^m|E(F'_i)\cap E_G(c)|<d_c$, there is some $j$ for which $c\in C_j$ but $F'_j$ contains no colour $c$ edge, so that, furthermore, $|V(F'_j)|\leq 2k$. But then, as $|M|\geq k+1$, there is some colour $c$ edge in $M$ which is not contained in $V(F'_j)$, contradicting the maximality of  $|\cup_{i=1}^mE(F'_i)|$.

Thus, we must have $|M|<  |E_G(c)|-d_c+ k+1\leq |E_G(c)|-\lfloor \frac{|E_c(G)|}{1+\beta}\rfloor+2k\leq m-\lfloor \frac{m}{1+\beta}\rfloor+2k= \lceil \frac{\beta m}{1+\beta}\rceil+2k$.
\end{proof}

Next, let $F_1,\ldots,F_m$ be a set of edge-disjoint rainbow forests in $G$ with $F'_i\subset F_i$ and $|E(F_i)|\leq k+\mathbf{1}_{\{i\leq \ell\}}$ for each $i$, so that $|\cup_{i=1}^mE(F_i)|$ is maximised. Suppose there is some $1\leq j\leq m$ for which $|E(F_j)|< k+\mathbf{1}_{\{j\leq \ell\}}$. Any edge outside of $\cup_{i=1}^mE(F_j)$ must be contained in $V(F_j)$ or share a colour with $F_j$. Thus, by Claim~\ref{sunny}, we have
\begin{align*}
e(G)&\leq |\cup_{i=1}^mE(F_i)|+\binom{2k}{2}+k\left(\frac{\beta m}{1+\beta}+2k\right)\\
&\leq km+\ell+4k^2+k\frac{\beta m}{1+\beta}
\\
&\leq km+\ell+4k^2+k\beta m(1-\beta/2)
\\
&\leq (1+\beta)(km+\ell)+k(4k-\beta^2 m/2)
\\
&\leq (1+\beta)(km+\ell)+k(4k-\beta^3 n/2)
\\
&<(1+\beta)(km+\ell),
\end{align*}
where we have used that $m\geq \beta n$ and $1\GG\beta \GG k/n$. This contradicts $e(G)= (1+\beta)(km+\ell)$, and thus there is no such $j$ with $|E(F_j)|<k+\mathbf{1}_{\{i\leq \ell\}}$.
\end{proof}

Given a large vertex cover $S$ in a graph $G$, we wish to find edge-disjoint $k$-edge rainbow forests so that large degree vertices outside $S$ are in every forest while small degree vertices outside $S$ have degree at most 1 in every forest. Lemma~\ref{smalltree0} can almost cover the edges within $S$ with forests. We now expand this to prove a lemma almost covering these edges as well as the edges next to vertices with small degree in $A:=V(G)\setminus S$.

\begin{lemma}\label{smalltree1}
Let $1\GG \beta\GG\varepsilon\GG k/n$ and $m\geq \beta n$.
Let $G$ be a properly coloured $n$-vertex graph, with $e(G)\geq (1+\beta)k m$ and let $A$ be a set of vertices with $|A|\leq \varepsilon n$ which contains no edges in $G$. Furthermore, suppose $d(v)\leq m+2k$ for each $v\in A$.

Then, $G$ has edge-disjoint $k$-edge rainbow forests $F_1, \dots, F_m$, where, additionally, for any $v\in A$ and $1\leq i\leq m$, $d_{F_i}(v)\leq 1$.
\end{lemma}
\begin{proof} Note that the lemma is trivial if $k=0$, and follows immediately from Lemma~\ref{smalltree0} when $A=\emptyset$. Suppose then that $k,|A|\geq 1$. Let $G_1$ be the subgraph of edges of $G$ contained within $S$. Let $k'$ and $\ell$ be integers with $0\leq \ell<m$ maximising $k'm+\ell$ subject to $e(G_1)\geq (1+\beta^2)(k'm+\ell)$ and $k'm+\ell\leq km$. By Lemma~\ref{smalltree0}, $G_1$ contains edge-disjoint rainbow forests $F_1',\ldots,F_m'$ so that $F_i'$ has $k'+\mathbf{1}_{\{i\leq \ell\}}$ edges.
If $k'm+\ell=km$, then these forests satisfy the lemma, so suppose that $k'm+\ell<km$, and therefore $e(G_1)\leq (1+\beta^2)(k'm+\ell)+2$.

 Let $G_2=G- G_1$ and pick an integer $\lambda$ so that $\beta km/4\geq \lambda|A|\geq \beta km/8$ (which is possible as $|A|\leq \varepsilon n$, $m\geq \beta n$ and $\beta\GG\varepsilon$). Note that
\begin{align*}\label{hottall}
e(G_2)&\geq (1+\beta)km-(1+\beta^2)(k'm+\ell)-2\geq (1+\beta^2)(km-k'm+\ell)+\beta km/2-2
\\
&\geq (1+\beta^2)(km-k'm+\ell)+\beta km/4\geq (1+\beta^2)(km-k'm-\ell)+\lambda|A|.
\end{align*}
By deleting edges if necessary, assume that $e(G_2)= (1+\beta^2)(km-k'm-\ell)+\lambda|A|$.
Choose integers $d_v$, $v\in A$, such that $\lfloor \frac{d_{G_2}(v)-\lambda}{1+\beta^2}\rfloor\leq d_v\leq \lceil \frac{d_{G_2}(v)-\lambda}{1+\beta^2}\rceil$ and $\sum_{v\in A}(d_v-\lambda)=km-k'm-\ell$. Let $A_1,\ldots,A_m$ be sets in $A$, so that each $A_i$
has size $k-k'-\mathbf{1}_{\{i\leq \ell\}}$ and each vertex $v$ appears in $d_v$ sets $A_i$, noting that this is possible as for each $v$ we have, as $\beta\GG k/n$ and $m\GG \beta n$,
\[
d_v\leq 1+\frac{d(v)-\lambda}{1+\beta^2}\leq 1+d(v)(1-\beta^2/2)\leq 1+(m+2k)(1-\beta^2/2)\leq m,
\]
and $\sum_{i=1}^m(k-k'-\mathbf{1}_{\{i\leq \ell\}})=km-k'm-\ell$.

Let $F''_1,\ldots,F''_m$ be a set of edge-disjoint rainbow forests in $G$ with, for each $i$, $F_i'\subset F''_i$ and $d_{F'_i}(v)\leq 1$ for each $v\in A_i$ and $d_{F'_i}(v)=0$ for each $v\in A\setminus A_i$. Furthermore, suppose $|\cup_{i=1}^mE(F''_i)|$ is maximised subject to these conditions.

\begin{claim}\label{sunny2}
For each $v\in A$, $|E_G(v)\setminus (\cup_{i=1}^mE(F''_i))|\leq \frac{\beta^2 m+\lambda}{1+\beta^2}+2k$.
\end{claim}
\begin{proof}
Fixing a vertex $v\in A$, let $E$ be the edges through $v$ not in $\cup_{i=1}^mE(F''_i)$, and suppose  that $|E|\geq d_{G_2}(v)-d_v+ k$. As $\sum_{i=1}^m|V(F''_i)\cap \{v\}|=d_{G_2}(v)-|E|<d_v$, there is some $j$ for which $v\in A_j$ but $F''_j$ contains no edge adjacent to $v$. But then, as $|E|\geq k$ and $G$ is properly coloured, there is some edge in $E$ with colour outside of $C(F''_j)$, contradicting the maximality of  $|\cup_{i=1}^mE(F''_i)|$.

Thus, we must have $|E|\leq  d_{G_2}(v)-d_c+ k\leq d_{G_2}(v)-\lfloor \frac{d_{G_2}(v)-\lambda}{1+\beta^2}\rfloor+k\leq m+k-\lfloor \frac{m-\lambda}{1+\beta^2}\rfloor+k\leq \lceil \frac{\beta^2 m-\lambda}{1+\beta^2}\rceil+2k$.
\end{proof}
Let $F_1,\ldots,F_m$ be a set of edge-disjoint rainbow forests in $G$ with $F_i''\subset F_i$ and $|E(F_i)|\leq k$
for each $1\leq i\leq m$, and $d_{F_i}(v)\leq 1$ for each $1\leq i\leq m$ and $v\in A$, so that $|\cup_{i=1}^mE(F_i)|$ is maximised. Suppose there is some $1\leq j\leq m$ for which $|E(F_j)|< k$. Any edge in $G_2$ outside of $\cup_{i=1}^mE(F_i)$ must contain a vertex in $V(F_j)\cap A$
or share a colour with $F_j$. Note that
\[
|V(F_j)\cap A|=|E(F_j)\setminus {E(F'_j)}|<k-k'-\mathbf{1}_{\{i\leq \ell\}}\leq k-k',
\]
and, hence, $|V(F_j)\cap A|\leq \min\{|A|,k-k'-1\}$. Thus, by Claim~\ref{sunny2}, and noticing that $G_2$ is globally $|A|$-, and hence $\eps n-$, bounded, we have
\begin{align*}
e(G_2)&\leq |\cup_{i=1}^mE(F_i)\setminus {E(F'_i)}|+\min\{|A|,k-k'-1\}\cdot\left(\frac{\beta{^2} m+\lambda}{1+\beta^2}+k\right)+k\eps n\\
&\leq km-(k'm+\ell)+(k-k'-1)\frac{\beta^2 m}{1+\beta^2}+|A|\frac{\lambda }{1+\beta^2}+k^2+k\eps n
\\
&\leq km-k'm-\ell+\beta^2 m(k-k'-1)+\lambda(1-\beta^2/2)|A|+k^2+k\eps n
\\
&\leq (1+\beta {^2})(km-k'm-\ell)+\lambda|A|+(k^2+k\eps n-\lambda  \beta^2|A|/2) \\
&\leq (1+\beta  {^2})(km-k'm-\ell)+\lambda|A|+(k^2+k\eps n-\beta^4 k n/16) \\
&< (1+\beta  {^2})(km-k'm-\ell)+\lambda|A|,
\end{align*}
where we have used that $1\GG\beta\GG\eps, k/n$ and $\lambda|A|\geq \beta k m/8\geq \beta^2 kn/8$. This contradicts $e(G_2)=(1+\beta {^2})(km-k'm-\ell)+\lambda|A|$, so $F_1,\ldots,F_m$ must satisfy the lemma.
\end{proof}

We can now prove the main result of this section.

\begin{lemma}[Near-decomposition into small rainbow trees]\label{Lemma_Small_Disjoint_Trees}
Let $1\GG \beta\GG\varepsilon\GG k/n$ and $m\geq \beta n$.
Let $G$ be a properly coloured, globally $m$-bounded, $n$-vertex graph, with $e(G)\geq (1+\beta)k m$
 and $S$ a vertex cover of $G$ with $|S|\geq (1-\varepsilon)n\geq 2m$.
 Then $ G$ has edge-disjoint $k$-edge rainbow forests $F_1, \dots, F_m$, where, additionally, for any $v\not\in S$ either $v\in V(F_i)$ for all $i$, or $d_{F_i}(v)\leq 1$ for all $i$.
\end{lemma}
\begin{proof}
For each $v\notin S$, let $d_v=\max\{0,\lfloor \frac{d(v)-2k}{m}\rfloor\}$, and $k'={\min}\{\sum_{v\notin S}d_v,k\}$. Form $G_1$ by selecting, for each $v\notin S$ with $d_v\geq 1$, $d_vm+2k$ arbitrary edges next to $v$. Let $G_2=G-G_1$ and note that, by construction, for each $v\notin S$, $d_{G_2}(v)\leq m+k$. Furthermore, if $k'<k$, then
\[
e(G_2)\geq e(G)-k'm-{2}k\cdot \eps n\geq (k-k')m+ \beta km-{2}k\eps n\geq (1+\beta/2)(k-k')m,
\]
where we have used that $m\geq \beta n$ and $\beta \GG \varepsilon$. As it is trivially true if $k'=k$, we thus always have that $e(G_2)\geq (1+\beta/2)(k-k')m$. Therefore, by Lemma~\ref{smalltree1}, $G_2$ contains $m$ edge-disjoint rainbow $(k-k')$-edge forests $F'_1,\ldots,F'_m$, in which $d_{F'_i}(v)\leq 1$ for each $v\notin S$.

Pick integers $0\leq d_v'\leq d_v$, $v\notin S$, so that $\sum_{v\notin S}d'_v=k'$.
Greedily, for each $1\leq i\leq m$ and $v\notin S$, add $d'_v$ edges in $G_1$ adjacent to $v$ to $F'_i$, so that the resulting subgraphs, $F_1,\ldots,F_m$ say, are still edge-disjoint rainbow forests. This is possible as each vertex $v\notin S$ has at least $md'_v+2k$ adjacent edges not in $E(\cup_{i=1}^mF'_i)$, so that at least $2k+1$ edges are  uncovered and adjacent to $v$, at least one of which will have its colour and other vertex not in the forest we are augmenting. Note that, in this process, if an edge is added to a forest $F_j'$ adjacent to $v\notin S$, then an edge is added adjacent to $v$ in all other forests $F_i$ as well.

Thus, noting that the resulting forests each have $k$ edges, $F_1,\ldots,F_m$ are edge-disjoint rainbow $k$-edge forests, with $\{v\notin S:d'_v\geq 1\}\subset F_i$ for each $1\leq i\leq m$ and $d_{F_i}(v)\leq 1$ for each $1\leq i\leq m$ and $v\in S$ with $d'_v=0$.
\end{proof}

\subsection{Completion}
In this section we show how to modify  nearly-spanning rainbow trees into spanning ones.
The starting point of this section is Lemma~\ref{Lemma_Hamiltonian_Decomposition_gap}. That lemma implies that in any properly coloured $K_n$, there is a set $S$ of size $(1-o(1))n$ such that $K_n[S]$ has a near-decomposition into rainbow Hamiltonian paths $P_1, \dots, P_{(1-o(1))n/2}$.  Indeed a random set $S$ will have this property since it satisfies that assumptions of Lemma~\ref{Lemma_Hamiltonian_Decomposition_gap} by Lemma~\ref{Lemma_Random_Subgraph_General}. Our goal in this section is to take such a family of rainbow paths, and modify them into a near-decomposition of $K_n$ into rainbow spanning trees.

The paths $P_1, \dots, P_{(1-o(1))n/2}$ are modified into spanning trees gradually, i.e.\ we switch edges on them one at a time to get bigger and bigger rainbow trees. During this modification procedure we always have a family of rainbow trees $T_1, \dots, T_{(1-o(1))n/2}$ which satisfy several properties that guarantee that it is possible to keep extending them. We will now informally go through these properties and explain why each is natural. The first property is the following:
\begin{itemize}
\item[(a)] $S\subseteq V(T_i)$ for all $i$.
\end{itemize}
This property simply comes from the fact that the trees $T_i$ are formed by enlarging the paths $P_i$, and the paths $P_i$ had $V(P_i)=S$. Property (a) is useful to have because we will have more control over vertices outside $S$ due to the fact that they were untouched by the starting paths $P_1, \dots, P_{(1-o(1))n/2}$.
\begin{itemize}
\item[(b)] For every $v\not\in S$, the tree $T_i$ has at most one edge through $v$.
\end{itemize}
Since we aim to produce trees which are spanning in $K_n$, every vertex $v\not\in S$ will eventually need to be added to every tree. Condition (b) will ensure that every vertex $v\not\in S$ always has enough free edges to be added to every  tree. Without it, it is possible that all the edges in $K_n$ through $v$ lie in some small subfamily of trees $T_1, \dots, T_m$, preventing the addition of $v$ to the other trees.
\begin{itemize}
\item[(c)] For a tree $T_i$, there are $n-|T_i|$ colours $c$ outside $T_i$ with $|E_{K_n}(c)|\geq (1-o(1))n/2$.
\end{itemize}
For a vertex $v\not\in V(T_i)$ it will not always be possible to add an edge from $v$ to $T_i$ in order to produce a rainbow tree. While properties (a) and (b) ensure that there are free edges from $v$ to $T_i$, it is conceivable that the colours of all these edges are already present on $T_i$, so $v$ cannot simply be added while maintaining a rainbow tree. We get around this by finding some colour $c$ outside of $T_i$ and two edges $e\in E(c)$, $f\in e(T_i)$ so that $T_i-f+e$ is a rainbow tree, i.e.\ we switch an edge on $T_i$ for an edge of some previously unused colour.  This operation frees up the colour $c(f)$, which we might be able to use to attach $v$. Property (c) ensures that there are \emph{many} colours $c(f)$ which can be freed using this operation.
\begin{itemize}
\item[(d)] There is a graph $H$ disjoint from $T_1, \dots, T_{(1-o(1))n/2}$ in which any set of $k$ colours  covers at least $(1-o(1))n$ vertices (where $k$ is a large constant).
\end{itemize}
Property (d) plays a similar role to property (c), i.e.\ it allows us to free up more colours, with the hope that eventually we free a colour which is present at some vertex $v\not\in V(T_i)$ (and then add that vertex to the tree $T_i$). The reason we need both properties (c) and (d) is a bit technical. In general, property (d) is more powerful, except that to invoke it we need $k$ colours outside the tree $T_i$. This will not happen towards the end of our process when there might be only one colour outside the tree. On the other hand (c) can always be invoked to free up a small number of colours. The strategy is to combine the applications of (c) and (d), i.e.\ first we apply (c) to free up $k$ colours, and then we use (d) to free up enough colours to add $v$.

The following lemma is what we use to exchange edges on a tree with edges outside it.
\begin{lemma}\label{Lemma_Free_Edges_On_Tree}
Let $T$ be a tree and $G$ a  graph with no isolated vertices with $V(G)\subseteq V(T)$.
Then for every $v\in V(G)$, there are edges $xv\in E(T)$ and $yv\in E(G)$ with $T-xv+yv$ a tree. In particular, there are $\geq |G|/2$ edges $e\in T$ for which there is an edge $f\in E(G)$ with $T-e+f$ a tree.
\end{lemma}
\begin{proof}
Let $yv$ be an arbitrary edge of $G$ containing $v$. Since $T$ is a tree and $\{y,v\}\subseteq V(T)$,  $T+yv$ has a  cycle $C$ containing the edge $yv$. Let $xv\neq yv$ be the other edge of $C$ containing $v$. Now $T+yv-xv$ is the required tree.

Thus to every $v\in V(G)$ we can assign a pair of edges $e_v\in T$, $f_v\in G$ containing $v$ with $T-e_v+f_v$ a tree. Since $v\in e_v$, for an edge $e\in E(T)$ there can be at most two vertices $v\in V(G)$ with $e=e_v$. This gives $|\{e_v:v\in V(G)\}|\geq |G|/2$ as required.
\end{proof}

The following is the basic extension lemma which drives our proof. Under conditions to be compared to (b) -- (d), it shows how to extend a tree by one vertex. The idea of the proof of the lemma is to show that by performing two switches as in Lemma~\ref{Lemma_Free_Edges_On_Tree}, we can free up nearly half of the colours on $T$. At least one of these colours will have an edge going to $T$, which can be added to extend the tree.

\begin{lemma}\label{Lemma_Tree_Extend_By_One_Vertex}
In a properly coloured $n$-vertex graph $G$, suppose that we have:
\begin{itemize}
\item  $T$  a rainbow tree with $|T|=n-1$.
\item $v\not\in V(T)$ with $d(v)\geq \frac12 n+b$.
\item $c\not\in C(T)$ with $e(c)\geq b$.
\item $H$ a graph on $V(T)$ in which any set of $b$ colours of $C(T)$ covers  $\geq n-2b$ vertices.
\end{itemize}
Then, there is a rainbow tree $T'$ in $T\cup H\cup E(c)\cup E(v)$ with $V(T')=V(T)\cup\{v\}$,  $e(T'\setminus  T)\leq 3$, and $d_{T'}(v)=1$.
\end{lemma}
\begin{proof} If there is a colour $c$ edge next to $v$, then clearly we can add such an edge to $T$ to get the required tree. Assume, then, that every colour in $C(v)$ is on $T$, and thus, in particular, $V(c)\subset V(G)\setminus\{v\}=V(T)$.

Let $J$ be the set of edges $j\in E(T)$ for which there is a colour $c$ edge $e_j$ so that $T_j:=T-j+e_j$ is a (rainbow) tree. By Lemma~\ref{Lemma_Free_Edges_On_Tree} and $e(c)\geq b$, we have $e(J)\geq b$ (for the application of this lemma, we take $G$ to be the set of colour $c$ edges).

For each $j\in J$, let $H_j$ be the graph of colour $c(j)$ edges in $H$ with no isolated vertices. By Lemma~\ref{Lemma_Free_Edges_On_Tree}, we have
\begin{equation}\label{frozen}
V(H_j)\subset V(\{e\in E(T_j):\exists e'\in E(H_j)\text{ s.t.\ } T_j-e+e'\text{ is a tree}\}).
\end{equation}
Notice that the trees $T_j-e+e'$ above are always rainbow (since $T_j$ is a rainbow tree on $V(T)$ missing colour $c(j)$ and $V(H_j)\subseteq V(T)$).
Let
\[
J'=\{e\in E(c)\cup E(T):\text{$\exists$ a rainbow tree $T'_e$ in $T\cup H\cup E(c)$ with $V(T')=V(T)$,  $e(T'\setminus  T)\leq 2$, $c(e)\notin C(T'_e)$}\}.
\]
Then,  for each $j\in J$,  by \eqref{frozen},
we have $V(H_j)\subset V(J')$. Therefore, $V(\bigcup_{j\in J} H_j)\subset V(J')$.

As $|J|\geq b$, we have $|V(\bigcup_{j\in J} H_j)|\geq n-2b$. Thus, $|V(J')|\geq n-2b$, so that $|J'|\geq \frac12 n-b$. As $C(v)\subset C(T)$, $C(J')\subset \{c\}\cup C(T)$, $|T|=n-1$ and $d(v)\geq  \frac12 n +b$, there is some edge $e$ adjacent to $v$ and $f\in J'$ with $c(e)=c(f)$. Then, using the tree $T'_f$ from the definition of $J'$, the tree $T'_f+e$ satisfies the conditions in the lemma.
\end{proof}

Iterating the above lemma, we can turn nearly-spanning trees into spanning trees. The conditions we need are to be compared with (a) -- (d).
\begin{lemma}[Completing rainbow trees]\label{Lemma_Many_Trees_Completion}
Let $1\GG \beta , k^{-1}\GG \varepsilon \GG n^{-1}$. In a properly coloured $K_n$ suppose that we have the following:
\begin{enumerate}[(i)]
\item  $S\subseteq V(K_n)$ with $|S|\geq n- \varepsilon n$.
\item  $T_1, \dots, T_{n(1-8\beta)/2}$ rainbow trees with $V(T_i)\supseteq S$.
\item  For each $T_i$, there is a set $C_L^i$ of $n-|T_i|$ colours outside of $C(T_i)$ with $\geq n(1-\beta)/2$ edges.
\item  For each $v\not\in S$, $d_{T_i}(v)\leq 1$ for all $i$.
\item  $H$ a  subgraph on $S$ disjoint  from $T_1, \dots, T_{n(1-8\beta)/2}$ in which any set of $k$ colours  covers at least $n(1-\beta)$ vertices.
\end{enumerate}
Then, there are $n(1-8\beta)/2$ spanning rainbow trees in $K_n$.
\end{lemma}
\begin{proof}
Set $r=n-|S|\leq \varepsilon n$.
\begin{claim}\label{Claim_HSubgraph}
Let $H'$ be a subgraph  of $H$ with $e(H')\geq e(H)-4rn$. Any set of $\beta n/2$ colours in $H'$ covers at least $n(1-2\beta)$ vertices.
\end{claim}
\begin{proof}
Consider a set $Y$ of $\beta n/2$ colours in $H'$. Since $\beta,k^{-1}\GG \varepsilon$ and $r\leq \varepsilon n$, we have $8\beta^{-1}r\cdot k\leq \beta n/2$ and thus $Y$ can be partitioned into disjoint subsets $Y_1, \dots, Y_{8\beta^{-1}r}$ of order $\geq k$. Since $e(H\setminus H')\leq 4rn$, one of these subsets $Y_i$ has $\leq  4rn/(8\beta^{-1}r)=\beta n/2$ edges in $E(H)\setminus E(H')$. Since $|Y_i|\geq k$, by the assumptions of the lemma, $Y_i$ covers at least $n(1-\beta)$ vertices in $H$. At most $\beta n$ of these might be uncovered in $H'$ (any uncovered vertex like this must have a colour $Y_i$ edge of $H\setminus H'$ passing through it. There are $\leq \beta n/2$ such edges).  This shows that $Y_i$ covers at least $n(1-2\beta)$ vertices in $H'$.
\end{proof}
Let  $T_1', \dots, T'_{n(1-8\beta)/2}$ be a set of edge-disjoint rainbow trees in $K_n$ satisfying (ii)--(iv) and also
\begin{itemize}
\item[(vi)]
$ e(T_i'\setminus T_i)\leq 3(|T_i'|-|T_i|).$
\end{itemize}
Additionally, choose this family of trees so that $\sum_{i=1}^{n(1-8\beta)/2} e(T_i')$ is as large as possible. We claim that all the rainbow trees $T_i'$ are spanning.
Suppose for the sake of contradiction there is a vertex $v\not\in V(T_j')$ for some $j$.
By (iii), and as $|T_i'|<n$, there is a colour $c\in C_L^j$ outside $C(T_j')$ with $\geq n(1-\beta)/2$ edges.
Since $T_j'$ satisfies (ii), we have $v\not \in S$.
Let $G^-$ be the subgraph of $K_n$ on $V(T_j)\cup \{v\}$ with the edges of $T_i'$ deleted for all $i$,  the edges not touching $S$ deleted, and edges with colour in $C_L^j\setminus \{c\}$ deleted. Let $G=G^-\cup T_j'$.

Since the trees $T_i'$ satisfy (iv), the number of trees is $n(1-8\beta)/2$,  and $|S|\geq n- \varepsilon n$, we have $d_G(v)\geq \frac12|G|+\beta n$. Since the trees $T_i'$ are rainbow, $|S|\geq n- \varepsilon n$, and $|E_{K_n}(c)|\geq n(1-\beta)/2$  we have $|E_G(c)|\geq \beta n$.
Let $H'=H\cap G$ to get a graph with $e(H')\geq e(H)-\sum_{c'\in C_L^i}|E(c')|- \sum_{j=1}^{n(1-8\beta)/2} e(T_j'\setminus T_j)\geq e(H)-rn/2-3rn$ (using $|C_L^i|\leq r$, (vi), $|T_j|\geq |S|$, and $|T_j'|\leq n$). By Claim~\ref{Claim_HSubgraph}, any set of $\beta n/2$ colours in $H'$ covers at least $n(1-2\beta)$ vertices.

Apply Lemma~\ref{Lemma_Tree_Extend_By_One_Vertex} to  $G$, with the tree $T_j'$, vertex $v$, colour $c$, graph $H'$, $n'=|T_j'|+1$, and $b=\beta n$. This gives a rainbow spanning tree $T_j''$ in $G$ containing at most $3$ edges outside $T_j'$ and having $d_{T_j''}(v)=1$. Notice that the family of trees $\{T_i':i\neq j\}\cup \{T_j''\}$ satisfies (ii) -- (iv) and (vi).
Indeed $S\subseteq V(T_j')\subseteq V(T_j'')$ implies that (ii) holds. For (iii) we have that $C_L^i\setminus \{c\}$ is a set of $n-|T_j'|-1=n-|T_j''|$ colours outside $C(G)\cup C(T_j')\supseteq C(T_j'')$ with $\geq n(1-\beta)/2$ edges. For (iv) we have $d_{T_{i}''}(v)\leq 1$ by the property from Lemma~\ref{Lemma_Tree_Extend_By_One_Vertex}  and $d_{T_{i}''}(u)\leq d_{T_{i}'}(u)\leq 1$ for $u\not\in S\setminus \{v\}$ since there are no edges in $G\setminus T_j'$ through such $u$.
Finally, (vi) comes from the properties from Lemma~\ref{Lemma_Tree_Extend_By_One_Vertex} since $e(T''_j\setminus T'_j)\leq 3$.
Thus we have a larger family of trees satisfying (ii) -- (iv) and (vi), contradicting the maximality of the original family.
\end{proof}

\subsection{Near-decompositions into spanning rainbow trees}
Now we combine everything from this section to prove the asymptotic version of the Brualdi-Hollingsworth and Kaneko-Kano-Suzuki Conjectures.
We will need the following standard lemma.
\begin{lemma}\label{Lemma_Min_Degree_Subgraph_Large}
Every graph $G$ with $e(G)\geq (1-(\varepsilon/2)^2)n^2/2$ has an induced subgraph $H$ with $\delta(H)\geq (1-\varepsilon)n$.
\end{lemma}
\begin{proof}
Let $S$ be the set of vertices $v$ in $G$ with $d(v)\leq (1-\varepsilon/2)n$. We have $2e(G)\leq (n-|S|)n+|S|(1-\varepsilon/2)n$ which combined with $e(G)\geq (1-(\varepsilon/2)^2)n^2/2$ gives $|S|\leq \varepsilon n/2$. Let $H=G\setminus S$ to get a graph with $\delta(H)\geq (1-\varepsilon/2)n-|S|\geq (1-\varepsilon)n$.
\end{proof}

We will also need the following lemma about switching edges between a tree and a forest.
\begin{lemma}\label{Lemma_Switch_Forest_On_Tree}
Let $T$ be a tree and $F$ a forest all of whose edges touch $V(T)$. Then, there is a tree  $T'$ which contains $F$ and is contained in $T\cup F$.
\end{lemma}
\begin{proof}
Notice that $T\cup F$ is connected since $T$ is a tree and all edges of $F$ touch $T$.
Let $T'$ be a connected subgraph of $T\cup F$ which contains $F$ and has $e(T')$ as small as possible.
If $T'$ is acyclic then we are done.
Otherwise, $T'$ contains a cycle $C$. Since $F$ is a forest $C$ must contain at least one edge of $T$. Deleting this edge gives a smaller connected graph contradicting the minimality of $e(T')$.
\end{proof}

By combining everything in this section with our earlier Hamiltonian decompositions we can show that the Brualdi-Hollingsworth and Kaneko-Kano-Suzuki Conjectures hold asymptotically when the colouring on $K_n$ is close to a $1$-factorization.
\begin{lemma}\label{Lemma_Tree_Decomposition_Many_Large_Colours}
Let $1\GG\varepsilon, \log^{-1}n\GG \gamma \GG n^{-1}$.
Let $K_n$ be properly coloured with $\geq (1-\gamma)n$ colours having $\geq (1-\gamma)n/2$ edges. Then, $K_n$ has $(1-8\varepsilon)n/2$ edge-disjoint spanning rainbow trees.
\end{lemma}
\begin{proof}
Choose $1\GG\varepsilon, \log^{-1}n \GG  \eta \GG \beta,\hat{k}^{-1} \GG \nu\GG \gamma_1\GG \gamma \GG n^{-1}$.

\smallskip

\textbf{Set aside small colours:}
Let $C$ be the set of  colours with $\geq (1-\gamma)n/2$ edges in $K_n$.
 By the assumption of the lemma and $\gamma_1\GG \gamma$ we have $e_{K_n}(C)\geq (1-\gamma)^2n^2/2\geq (1-(\gamma_1/2)^2)n^2/2$.

By Lemma~\ref{Lemma_Min_Degree_Subgraph_Large} applied to $K_n[C]$ with $\varepsilon=\gamma_1$ there is a subgraph $G$ of $K_n$ with $\delta(G)\geq (1- \gamma_1)n$,  having only colours of $C$. Set $n_1=|G|\geq (1- \gamma_1)n$ and notice that $G$ is $(3\gamma_1, 1, n_1)$-typical.

\smallskip

\textbf{Choose set $S$:} Fix $n_2= \lceil(1- \nu)n\rceil$. Apply Lemma~\ref{Lemma_Random_Subgraph_General} (c) with $p=n_2/n_1$, $n'=n_1$, $\mu=1/2$, and  $\gamma'= 3\gamma_1$ in order to find a set of vertices $S\subseteq V(G)$  of order $n_2$ with $G[S]$  globally $(1+3\gamma_1)(n_2/2n_1)n_2$-bounded, and $G[S]$ $(6\gamma_1, 1, n_2)$-typical.
Notice that   $G[S]$ is globally $(1-0.9\nu) n_2/2$-bounded (using  $n_1\geq(1- \gamma_1)n$, $n_2= \lceil(1- \nu)n\rceil$ and $1\GG\nu\GG\gamma_1$). Notice that in $G[S]$ any colour of $C$ covers at least $\geq (1-\gamma)n- (n-n_2)\geq (1-2\nu)n_2$ vertices.

\smallskip

\textbf{Set aside a pseudorandom graph $H$:} Partition $G[S]$ into subgraphs $G_1$ and $H$ with every edge placed in $H$ independently with probability $\eta$.
By Lemma~\ref{Lemma_Random_Subgraph_General} (b), $G_1$ is $(12\gamma_1, 1-\eta, n_2)$-typical and globally $(1+6\gamma_1)(1-\eta)(1-0.9\nu) n_2/2$-bounded (for the application take $p=1-\eta$, $\mu=(1-0.9\nu)/2$, $n'=n_2$, $\delta=1$, $\gamma'=6\gamma_1$). Since $\nu\GG \gamma_1$,   $G_1$ is globally $(1-0.5\nu)(1-\eta) n_2/2$-bounded.
By Lemma~\ref{Lemma_Erdos_Renyi_Subgraph} applied with $p=\eta$, $\varepsilon'=\varepsilon$, $k'=\hat k$, $\nu'=2\nu$, $H$ has the property that any set of $\hat k$ colours of $C$ cover $\geq (1-8\nu)n\geq (1-\varepsilon)n$ vertices.

\smallskip

\textbf{Find near-decomposition of $K_n[S]$ into rainbow paths:} Apply Lemma~\ref{Lemma_Hamiltonian_Decomposition_gap} to $G_1$ with $n'=n_2$, $\gamma'=12\gamma_1$, $p = 0.5\nu$, $\delta=1-\eta$ in order to find $(1-0.5\nu)(1-\eta)n_2$ edge-disjoint rainbow Hamiltonian paths in $G_1$. Using $(1-0.5\nu)(1-\eta)n_2\geq (1-\varepsilon)n/2$ choose a subcollection
 $P_1, \dots, P_{\lfloor(1- \varepsilon)n/2\rfloor}$ of these paths. Since $G_1$ is a subgraph of $G$, these paths only use edges with colour in $C$.

\smallskip

\textbf{Add small colours into trees:}
Let $C_L$ be the set of colours with $\geq (1-\varepsilon)n/2$ edges in $K_n$. Choose $k=\max(n-1-|C_L|,0)$. By assumption we have $k\leq \gamma n$.
Let $G_2$ be the subgraph of $K_n$ consisting of edges with colour outside $C_L$ which touch $S$.
We claim that $e(G_2)\geq  (1+\eta )k \lfloor(1-\varepsilon)n/2\rfloor$. When $k=0$, this is obvious. Otherwise since $\delta(K_n)=n-1$ and $K_n$ is properly coloured,   we have $e(G_2)\geq \frac12\sum_{v\in S} d_{C(K_n)\setminus C_L}(v) \geq |S|(\delta(K_n)-|C_L|)/2= k\lceil(1-\nu)n\rceil/2\geq   (1+\eta )k (1-\varepsilon)n/2$.
By definition of $C_L$, the graph $G_2$ is globally $\lfloor(1-\varepsilon)n/2\rfloor$-bounded. Apply Lemma~\ref{Lemma_Small_Disjoint_Trees} to $G_2$ with $m=\lfloor(1-\varepsilon)n/2\rfloor$, $\varepsilon'=1.01\nu$, $\beta'=\eta$, $S=S$. This gives us edge-disjoint rainbow forests $F_1, \dots, F_{\lfloor(1- \varepsilon)n/2\rfloor}$ of size $k$ in $G_2$.

Apply Lemma~\ref{Lemma_Switch_Forest_On_Tree} for $i=1, \dots, \lfloor(1- \varepsilon)n/2\rfloor$ to $P_i$ and $F_i$ in order to find a rainbow tree $T_i$ containing $F_i$ and contained in $P_i\cup F_i$ ($T_i$ is rainbow since $P_i$ and $F_i$ are colour-disjoint which happens because $C(P_i)\subseteq C\subseteq C_L$ and $C(F_i)\cap C_L=\emptyset$). In particular, each $T_i$ contains $k$ edges outside $C_L$ (the edges of $F_i$).
Since $k\geq n-1-|C_L|$, this implies that each $T_i$ avoids  $k+ |C_L|-e(T_i) \geq n-1-e(T_i)$ colours of $C_L$, each of which has $\geq (1-\varepsilon)n/2$ edges in $K_n$.
Additionally, from Lemma~\ref{Lemma_Small_Disjoint_Trees}, we have that for every vertex $v\not\in S$ either $v\in T_i$ for all $i$ or $d_{T_i}(v)\leq 1$ for all $i$. Let $S'=S\cup \{v\not\in S:  v\in T_i\text{ for each } i\}$ and notice that $|S'|\geq |S|=  \lceil(1-\nu)n\rceil$.
 Now for each $i$ and $v\not\in S'$, we have  $d_{T_i}(v)\leq 1$ and also $S'\subseteq V(T_i)$.

\smallskip

\textbf{Make trees spanning:}
Observe that $H$ is disjoint from $G_1$ and $G_2$. (The former holds by construction of $G_1$. The latter by $C(H)\subseteq C\subseteq C_L$ and $C(G_2)\cap C_L=\emptyset$), and hence $H$ is disjoint from the trees $T_1, \dots, T_{\lfloor(1-\varepsilon)n/2\rfloor}$.
Apply Lemma~\ref{Lemma_Many_Trees_Completion} with $S=S'$, trees $T_1, \dots, T_{(1-8\varepsilon)n/2}$, $H=H$, $\beta'=\varepsilon$, $k'=\hat k$ and $\eps'=\nu$
in order to find $(1-8\varepsilon)n/2$ edge-disjoint spanning rainbow trees in $K_n$, where we have used that $\varepsilon,1/\hat{k}\GG \nu$.
\end{proof}

Combining the above with our earlier result about Hamiltonian decompositions, we prove that the Brualdi-Hollingsworth and Kaneko-Kano-Suzuki Conjectures hold asymptotically.
\begin{theorem}
Let $1\GG \varepsilon\GG n^{-1}$.
Every properly coloured $K_n$ has $(1-\varepsilon)n/2$ edge-disjoint spanning rainbow trees.
\end{theorem}
\begin{proof}
Fix $1 \GG \varepsilon, \log^{-1}n\GG \gamma \GG n^{-1}$.
If $K_n$ has $\geq (1-\gamma)n$ colours having $\geq (1-\gamma)n/2$ edges, then the theorem follows from Lemma~\ref{Lemma_Tree_Decomposition_Many_Large_Colours}.
Otherwise, $K_n$ has $\leq (1-\gamma)n$ colours having $\geq (1-\gamma)n/2$ edges, and the theorem follows from Lemma~\ref{Lemma_Hamiltonian_Decomposition_Kn}.
\end{proof}

\section{Concluding remarks}
There are various other areas in which our results have implications. We mention some of them here.
\begin{itemize}
\item
Constantine made the following generalization of the Brualdi-Hollingsworth Conjecture.
\begin{conjecture}[Constantine \cite{constantine2002multicolored}]\label{Conjecture_Constantine}
Every properly $(2n-1)$-coloured $K_{2n}$ can be decomposed into edge-disjoint rainbow spanning trees which are all isomorphic to each other.
\end{conjecture}
The best known result about this is due to the second and third author~\cite{pokrovskiy2017linearly} who showed that it is possible to find $10^{-{12}}n$ edge-disjoint rainbow copies of some particular tree.
While we did not do this, our results still have implications for Constantine's Conjecture. In particular Corollary~\ref{Corollary_Hamiltonian} is relevant --- it shows that under the assumptions of Constantine's Conjecture we can nearly-decompose the graph into nearly-spanning rainbow paths.

Additionally we expect that the methods in this paper can be generalized to prove the true asymptotic version of Constantine's Conjecture, i.e.\ to find $(1-o(1))n$ edge-disjoint isomorphic spanning rainbow trees under the assumption of the theorem. We think this is plausible as the trees we find in the proof of Theorem~\ref{Theorem_Trees} are all quite similar to each other --- they are all built from a length $(1-o(1))n$ path by making $o(n)$ modifications. It seems likely that, with some additional ideas, the modifications can be controlled in order to give a copy of the same tree.

\item
Notice a parallel between Theorem~\ref{Theorem_Trees} and Lemma~\ref{Lemma_Small_Disjoint_Trees} --- both of these results give a near-decomposition of a graph into forests of the same size.  We wonder if there is a common generalization of these results.
\begin{conjecture}
For $\varepsilon>0$, there exists an $m\in \mathbb{N}$ so that the following holds for all $k$.
Every properly coloured, globally $m$-bounded graph $G$ with $km$ edges has $\geq (1-\varepsilon)m$ edge-disjoint rainbow forests of order $k$.
\end{conjecture}
Currently there are two extremes of this conjecture which are known to be true. Theorem~\ref{Theorem_Trees} shows that it holds when $|G|=2m$ and $k=|G|-1$.
Lemma~\ref{Lemma_Small_Disjoint_Trees} shows that it holds when $k= o(|G|)$. It would be interesting to prove or disprove it in general.

\item
Recall that the randomized rainbow matching $M$ in Lemma~\ref{Lemma_near_perfect_matching} behaves like a uniformly random perfect matching in a sense that any edge of $G$ ends up in $M$ with (approximately) at least the expected probability $\ad(G)^{-1}$. One can ask whether more can be proven, i.e.\ whether $M$ shares more features with a uniformly random perfect matching. This is indeed the case --- Lemma~\ref{Lemma_near_perfect_matching} can easily be strengthened to say more about the matching $M$. For example, with some work the following can be added to that lemma.
\begin{align*}
\P(e\in E(M))&=(1\pm p)\frac{1}{\delta n}\hspace{1cm}\text{ for each $e\in E(G)$}.\\
\P(e,f\in E(M))&=(1\pm p)\frac{1}{\delta^2 n^2}\hspace{1cm}\text{ for each $e\neq f\in E(G)$}.\\
\P(v\not\in V(H))&=(1\pm p)p\hspace{1cm}\text{ for each $v\in V(G)$}.
\end{align*}
The randomness of the matching produced in Lemma~\ref{Lemma_near_perfect_matching} may have applications in future work.

\item
Notice that some of our results (particularly
Lemma~\ref{Lemma_Nearly_Rainbow_Decomposition}) are about graphs which
may not be properly coloured, but are only locally
$n^{\epsilon}$-bounded. It is natural to ask whether our other
theorems can be proved with ``proper colouring'' replaced by ``local
boundedness'', or perhaps even with the proper colouring assumption
removed entirely. Some results in this direction were recently
obtained by Kim, K\"uhn, Kupavskii, and Osthus in
\cite{kim2018rainbow} (see note below) .

It would be extremely interesting to prove new results about spanning
rainbow structures in graphs with \emph{no local boundedness
assumptions at all}. For example in
\cite{pokrovskiy2017counterexample}, the second and third authors
asked whether every globally $(1-o(1))n$ bounded $K_{n,n}$ has a
perfect rainbow matching. If true, this would be a natural weakening
of the recently disproved Stein's Equi-$n$-Square Conjecture (see \cite{pokrovskiy2017counterexample}).

\end{itemize}

\subsection*{Note added in proof}
The results of Theorem 1.2 and its corollaries were  presented at the
``Workshop on Probabilistic and Extremal Combinatorics'' in Harvard
07/02/2018 (see \cite{AlexeySlides}). After the presentation we
learned from Keevash and Yepremyan that they also found a proof of the
Akbari-Alipour Conjecture (Conjecture 1.3) for large $n$
(see~\cite{KY18}).

Also after hearing our Theorem 1.2 at the workshop, Kim, K\"uhn, Kupavskii,
and Osthus published the preprint \cite{kim2018rainbow} on 22.5.2018. 
In this paper they proved (amongst others) that every coloured $K_{n,n}$ which is
globally $(1-o(1))n$-bounded and locally $o(n/\log^2 n)$-bounded has
$(1-o(1))n$ edge-disjoint rainbow perfect matchings. This is on one hand
stronger than Theorem~\ref{Theorem_LatinSquare} since it also works
for locally bounded colorings, but it is also weaker since it requires
all (rather than just few) colors to have size  less than $(1-o(1))n$.
In particular it does not imply the Akbari-Alipour Conjecture or our
results on multiplication tables of groups. Independently from our
work, Kim, K\"uhn, Kupavskii, and Osthus also proved results similar
to our
Theorem 1.10 about decompositions into rainbow Hamiltonian cycles
(that are both stronger and weaker as we explain above). The main
focus of their work is quite different from ours and they deduce their
result from a general theorem about rainbow $F$-factors for arbitrary
graphs $F$.

\subsection*{Acknowledgement}
The authors would like to thank Matthew Kwan and Vincent Tassion for help with probabilistic arguments.
Parts of this work were carried out when the first author visited the 
Institute for
Mathematical Research (FIM) of ETH Zurich. We would like to thank
FIM for its hospitality and for creating a stimulating research 
environment.

\bibliographystyle{abbrv}
\bibliography{rainbowtrees}

\end{document}